\documentclass[11pt, a4paper]{amsart}
\usepackage{mathrsfs}

\usepackage{amssymb}
\usepackage{amsfonts}

\usepackage{amscd}
\usepackage{bbm}
\usepackage{mathabx}
\usepackage{skak}

\usepackage{color}

\usepackage{hyperref}
\hypersetup{
    colorlinks,
    citecolor=black,
    filecolor=black,
    linkcolor=black,
    urlcolor=black
}

\usepackage[all]{xy}

\newtheorem{proposition}{Proposition}[section]

\newtheorem{lemma}[proposition]{Lemma}
\newtheorem{definition}[proposition]{Definition}
\newtheorem{theorem}[proposition]{Theorem}

\newtheorem{corollary}[proposition]{Corollary}

\newtheorem{example}[proposition]{Example}
\newtheorem{remark}[proposition]{Remark}

\newtheorem{lemma-definition}[proposition]{Lemma-Definition}

\newcounter{tmp}

\def\coh{\operatorname{coh}}
\def\Qcoh{\operatorname{Qcoh}}

\def\lto{\longrightarrow}

\def\A{{\mathcal A}}
\def\B{{\mathcal B}}

\def\D{{\mathcal D}}
\def\F{{\mathcal F}}
\def\E{{\mathcal E}}

\def\H{{\mathcal H}}
\def\K{{\mathcal K}}
\def\N{{\mathcal N}}
\def\cO{{\mathcal O}}
\def\R{{\mathcal R}}
\def\L{{\mathcal L}}

\def\V{{\mathcal V}}
\def\U{{\mathcal U}}

\def\M{{\mathcal M}}

\def\T{{\mathcal T}}
\def\I{{\mathcal I}}

\def\P{{\mathcal{P}}}

\def\ZZ{{\mathbb Z}}
\def\CC{{\mathbb C}}

\def\bD{{\mathbf D}}
\def\bR{{\mathbf R}}

\def\bp{{\mathbf p}}
\def\bL{{\mathbf L}}
\def\be{\overline{\mathbf{e}}}

\def\Ga{\Gamma}
\def\La{\Lambda}

\def\NN{{\mathbb N}}
\def\ZZ{{\mathbb Z}}

\def\CC{{\mathbb C}}

\def\PP{{\mathbb P}}

\def\Hom{\operatorname{Hom}}
\def\End{\operatorname{End}}
\def\Ext{\operatorname{Ext}}
\def\Pic{\operatorname{Pic}}

\def\Spec{\operatorname{Spec}}

\def\id{{\operatorname{id}}}

\def\kk{{\mathbf k}}

\def\tors{\operatorname{tors}\,}

\def\rad{\mathfrak{R}}
\def\op{\circ}

\newcommand{\Ho}{{\H^0}}

\newcommand{\Ob}{\operatorname{Ob}}

\newcommand{\SF}{\dS\!\dF\!\operatorname{--}\!}
\newcommand{\SFf}{\dS\!\dF_{fg}\!\operatorname{--}\!}
\newcommand{\prfdg}{\mathscr{P}\!\mathit{erf}\!\operatorname{--}}
\newcommand{\dRep}{\mathscr{R}\!\mathit{ep}}
\newcommand{\dMor}{\mathscr{M}\!\mathit{or}}

\newcommand{\Ac}{\dA\!\mathit{c}\!\operatorname{--}\!}

\newcommand{\prf}{\mathcal{P}\!\mathit{erf}\!\operatorname{--}\!}

\def\dA{\mathscr A}
\def\dB{\mathscr B}
\def\dC{\mathscr C}
\def\dD{\mathscr D}
\def\dE{\mathscr E}

\def\dF{\mathscr F}

\def\dL{\mathscr L}
\def\dM{\mathscr M}
\def\dN{\mathscr N}
\def\dP{\mathscr P}
\def\dR{\mathscr R}
\def\dS{\mathscr S}
\def\dT{\mathscr T}
\def\dU{\mathscr U}
\def\dV{\mathscr V}
\def\dW{\mathscr W}
\def\dX{\mathscr X}
\def\dY{\mathscr Y}
\def\dZ{\mathscr Z}

\def\Mod{{\mathscr M}\!\mathit{od}\!\operatorname{--}\!}

\def\mE{\mathsf E}
\def\mG{\mathsf G}
\def\mF{\mathsf F}
\def\mM{\mathsf M}
\def\mN{\mathsf N}
\def\mP{\mathsf P}
\def\mQ{\mathsf Q}
\def\mR{\mathsf R}
\def\mS{\mathsf S}
\def\mT{\mathsf T}
\def\mU{\mathsf U}

\def\mX{\mathsf X}
\def\mY{\mathsf Y}
\def\mZ{\mathsf Z}

\def\ma{\mathsf a}
\def\mb{\mathsf b}
\def\mc{\mathsf c}

\def\mh{\mathsf h}
\def\mf{\mathsf f}
\def\mg{\mathsf g}
\def\mi{\mathsf i}
\def\mj{\mathsf j}

\def\ml{\mathsf l}

\def\mp{\mathsf{p}}

\def\mr{\mathsf{r}}

\def\mw{\mathsf{w}}
\def\mv{\mathsf{v}}

\def\mpr{\mathsf{pr}}
\def\mId{\mathsf{id}}

\def\mPhi{\mathsf \Phi}

\def\m0{\mathsf 0}

\def\dHom{\mathsf{Hom}}
\def\dEnd{\mathsf{End}}

\def\ptr{\operatorname{pre-tr}}

\def\lHom{\underline{\H\!\mathit{om}}}

\def\wt{\widetilde}

\def\Rep{{\R\!\mathit{ep}}}
\def\Mor{{\M\!\mathit{or}}}

\def\rd{\mathfrak{r}}
\def\md{\operatorname{mod}\!\operatorname{--}\!}
\def\Md{\operatorname{Mod}\!\operatorname{--}\!}

\def\gldim{\operatorname{gl.dim}\;}
\def\Fl{\operatorname{Filt}}
\def\pt{\mathbf{pt}}

{\endgroup\hfill$\Box$}

{\endgroup\hfill$\Box$}

\def\hy{\mbox{-}}

\def\bi{\mathbf i}

\def\wL{\widetilde{\L}}

\def\Si{\Sigma}
\def\Iota{\operatorname{I}}

\def\mCone{\mathsf{Cone}}
\def\supp{\operatorname{supp}}
\def\length{\operatorname{length}}

\def\LL{\operatorname{L}}

\def\ldot{{\mathbf\cdot}\hspace{0.05em}}

\title[]{Derived noncommutative schemes, geometric realizations, and finite dimensional algebras}

\author[]{Dmitri Orlov}
\thanks{This work is supported by the Russian Science Foundation (RSF) under grant 14-50-00005}

\address{ Algebraic Geometry Department, Steklov Mathematical Institute of Russian Academy of Sciences,
Gubkin str. 8, Moscow 119991, RUSSIA}
\email{orlov@mi.ras.ru}

\dedicatory{Dedicated to the blessed memory of my friend Vladimir Voevodsky}

\date{}


\begin{document}
\begin{abstract}
The main purpose of this paper is to describe various phenomena and certain constructions arising in the process of studying derived noncommutative schemes.
Derived noncommutative schemes are defined as differential graded categories of a special type.
We review and discuss different properties of both noncommutative schemes and morphisms between them.
In addition, the concept of geometric realization for derived noncommutative scheme is introduced and
problems of existence and construction of such realizations are discussed.
We also study the construction of gluing noncommutative schemes via morphisms and consider some new phenomena, such as phantoms, quasi-phantoms, and Krull-Schmidt partners,
arising in the world of noncommutative schemes and allowing us to find new noncommutative schemes.
In the last sections we consider noncommutative schemes that are related to basic finite dimensional algebras.
It is proved that such noncommutative schemes have special geometric realizations under which the algebra goes to a vector bundle on a smooth projective scheme.
Such realizations are constructed in two steps, one of which is the well-known construction of Auslander, while the second step is connected with a new concept of a
well-formed quasi-hereditary algebra for which there are very particular geometric realizations sending standard modules to line bundles.
\end{abstract}

\maketitle

\section*{Introduction}

Noncommutative algebraic geometry is based on the fact that, as in commutative geometry affine schemes are directly related to rings or algebras over a field.
And in spite of the fact that in contrast to the commutative case for a non-commutative algebra $A$ we do not have any construction of the topological space $\Spec A,$ we can nevertheless freely speak about the category of quasi-coherent sheaves on an affine noncommutative scheme having in view the category of (right) modules $\Md A$ over the algebra $A.$
This observation  is a serious reason in order to shift the focus from varieties (or schemes) and to go directly to the categories of sheaves on these varieties (schemes), making these categories the main object of research.
This is in many ways natural also because the theory of sheaves is one of the most powerful methods of studying algebraic varieties.
In addition, it should be noted that quasi-coherent sheaves do not depend on a choice of topology on schemes, being sheaves in all natural topologies, and they best reflect the algebraic structure of schemes.

To the next question, which is already quite nontrivial, how to glue noncommutative affine schemes, there are several different approaches.
However, the most fruitful here is again the point of view related to the consideration of the category of sheaves, but with some natural generalizations dictated by homological algebra.
This approach consists in working in fact with the derived category of quasi-coherent sheaves and with the category of perfect complexes on a noncommutative scheme, when we talk about such a noncommutative scheme.
It is natural to begin a more detailed consideration by returning to commutative algebraic geometry.

Consider a scheme $X$ over a field $\kk $ and put on it some finiteness conditions.
We will assume that $X$ is quasi-compact and quasi-separated, i.e. it has a covering by affine schemes
whose intersections have the same property.
With any such scheme we associate the unbounded derived category of complexes of $\cO_X $\!--modules with quasi-coherent cohomology $\D_{\Qcoh}(X).$
In the papers \cite{Ne} and \cite{BVdB} it was shown that this category has enough compact objects and the triangulated subcategory of compact objects actually coincides with the category of perfect complexes $\prf X.$
It should be recalled that a complex is called perfect if it is locally quasi-isomorphic to a bounded complex of locally free sheaves of finite type.
In addition, it has also been shown in the papers \cite{Ne, BVdB} that the category of perfect complexes  $ \prf X $ can be generated by a single object, which is called a classic generator.
This means that the minimal full triangulated subcategory of $ \prf X$ containing this object and closed under taking direct summands, coincides with the whole category $ \prf X.$
Such an object will also be a compact generator for the category $ \D_{\Qcoh} ( X).$
Note that the article \cite{Ne} was dealing with separated schemes, for which the category
$ \D_{\Qcoh} (X) $ is equivalent to the usual unbounded derived category of quasi-coherent sheaves $ \D (\Qcoh X),$ while  in the paper
\cite{BVdB} these statements were proved in the general case of a quasi-compact and quasi-separated scheme.

The existence of such a generator
$\mE \in \prf X$ gives us an opportunity to look at the derived category  $ \D_{\Qcoh}(X),$ as well as  the triangulated category $\prf X ,$ from a different angle.
The results of the papers \cite {Ke, Ke2} allow us to assert that in this situation the category $ \D_{\Qcoh} (X) $ is equivalent to the unbounded derived  category of differential graded (DG) modules
$\D(\dR)$ over some differential graded (DG) algebra $\dR,$ and the triangulated category of perfect complexes $\prf X$ is equivalent to the category of perfect DG modules $\prf\dR.$
The differential graded algebra $\dR$ directly depends on the choice of the generator $\mE \in \prf X$ and is obtained as the  DG algebra of endomorphisms
$\dEnd (\mE)$ of the given generator, but not as an object of category  $\prf X,$ but as its lift to a differential graded category $\prfdg X,$ which is a natural enhancement of the category $\prf X.$
In particular, this means that the category $ \prf X $ is equivalent to the homotopy category $ \Ho(\prfdg X).$
A differential graded (DG) category $\dA$ is a category whose morphisms have the structure of complexes of $\kk$\!--vector spaces.
Passing from the complexes to their zero cohomology spaces we obtain a $\kk $\!--linear category $ \Ho (\dA) $ with the same objects,
which is called the homotopy category for DG category $ \dA.$ If there is an equivalence $ \epsilon: \Ho (\dA) \stackrel{\sim}{\to} \T,$ then $ (\dA, \epsilon) $
is called a DG enhancement for the category $\T.$

Usually triangulated categories have natural enhancements arising in the process of constructing these categories.
In our case, the triangulated category $ \D _ {\Qcoh} (X) $ has several natural enhancements: DG category of $h$\!--injective complexes, DG quotient of all complexes by
  acyclic complexes, DG quotient of $h$\!--flat complexes by acyclic $h$\!--flat complexes and so on.
All these enhancements are also in a natural way quasi-equivalent to each other, and when working with them,
  we can choose any convenient for us model from the class of quasi-equivalent DG categories.
A DG enhancement of the category $ \D_{\Qcoh} (X) $ induces a DG enhancement of the triangulated subcategory of perfect complexes $\prf X,$ which will be denoted by $\prfdg X.$

The results of the papers \cite{Ne, BVdB, Ke, Ke2} mentioned above, combined together, tell us that the DG category $\prfdg X $ is quasi-equivalent to a category of the form
$ \prfdg \dR, $  where $ \dR $ is the DG algebra of endomorphisms of some generator $ \mE \in \prfdg X.$
We can also note that in the case of quasi-compact and quasi-separated scheme, the DG algebra $ \dR $ is cohomologically bounded, i.e. it has only a finite number of nontrivial cohomology spaces.
Thus, the following definition of a derived noncommutative scheme over the field $ \kk $ arises (see. Definition \ref{noncommutative_scheme}).
By a derived noncommutative scheme $ \dX $ we shall call a $\kk$\!--linear DG category of the form  $ \prfdg \dR, $ where $ \dR $ is cohomologically bounded DG algebra over $\kk.$
It is natural to call the derived category $ \D(\dR) $ by the derived category of quasi-coherent sheaves on this noncommutative scheme,
  while the triangulated category $ \prf \dR $ will be called by the category of perfect complexes on $\dX.$
It was mentioned above that there is an equivalence of the triangulated categories $ \Ho (\prfdg \dR) \cong \prf \dR.$

Note also that, having considered the DG categories of the from  $ \prfdg \dR ,$ we not only were able to glue noncommutative scheme from affine pieces, but also in fact received derived noncommutative affine schemes
as well by passing from algebras to DG algebras.
Thus, this definition allows us to talk not only about noncommutative schemes, but also about derived noncommutative schemes together.

Many important properties of usual schemes can be extended to derived noncommutative schemes.
In particular, we can talk about smoothness, regularity and properness for noncommutative schemes.
We can also define morphisms between schemes as quasi-functors between DG categories.
Despite the fact that the usual morphisms between schemes are also morphisms between them in  noncommutative sense, it should be noted that there are much more morphisms in the noncommutative world
and they form a category (and even a DG category).
This implies that they can be added together and one can talk about maps between morphisms.
Many natural concepts from usual commutative algebraic geometry also generilize to noncommutative schemes: the concepts of compactification, resolution of singularities, Serre functor are defined.

In this paper we discuss some properties of derived noncommutative schemes and draw various analogies with the commutative case.
However, one of the important concepts, to which a large part of the article is devoted,  a gluing of noncommutative schemes  is an operation existing in  the noncommutative world and having no analogue in the commutative one. Another important concept is the geometric realization of derived noncommutative schemes.
It arises naturally for two reasons.
First, for each abstract algebraic structure it is always useful and interesting to find some geometric representations.
On the other hand, many noncommutative schemes come to us from usual geometry with a given geometric realization.

The most natural and often appearing but highly nontrivial examples of geometric realizations are related to admissible subcategories $ \N \subset \prf X $
in categories  of perfect complexes on smooth projective schemes $X.$ In this case, we obtain a noncommutative scheme as the DG category $ \dN \subset \prfdg X.$
It is easy to see that such a DG category can be realized in the form $ \prfdg \dR $ and, moreover, the noncommutative scheme $ \dN $ itself is smooth and proper.
The initial embedding $ \dN \subset \prfdg X $ is a particular but most interesting case of a geometric realization of smooth and proper noncommutative schemes (see Definition \ref{geom_real}).
Such realizations are called pure.

As mentioned above, for any two DG categories $ \dA $ and $ \dB $
 and a $ \dB^{\op}\hy\dA $\!--bimodule $ \mT$ we can define a new DG category $ \dC = \dA \underset{\mT} {\oright} \dB,$ which is called the gluing of these two DG categories via the bimodule $ \mT. $
This construction allows us to introduce gluing of derived noncommutative schemes $ \dX $ and $ \dY$ by imposing some boundedness condition on the gluing bimodule $ \mT. $
Noncommutative schemes of the form $ \dX \underset {\mT} {\oright} \dY $ adopt many properties of the schemes $\dX$ and $\dY$ under appropriate conditions on $ \mT.$
For example, the smoothness and properness of $ \dX $ and $ \dY $ imply the smoothness and properness of the gluing
   $ \dX \underset {\mT} {\oright} \dY $ assuming that the bimodule $ \mT $ is perfect.

In paper \cite{O_glue} we study the problem of geometric realizations of noncommutative schemes that are obtained by gluing smooth and proper noncommutative schemes $ \dX $ and $ \dY.$
In particular, it was proved that if noncommutative schemes $ \dX $ and $ \dY $ arise as admissible subcategories in categories of perfect complexes on smooth projective schemes, then their gluing $ \dX \underset {\mT} {\oright} \dY $ via a perfect bimodule $ \mT $ can be realized in the same way.
This problem is discussed in detail in sections \ref{glu_der_schem}--\ref{real_gluings}.

Derived noncommutative schemes may be quite different from commutative schemes in general even when they are smooth and proper.
In section \ref{phantoms}, a phenomenon such as quasi-phantoms and phantoms is discussed.
Without going into details, we can say that phantoms are such smooth and proper noncommutative schemes $ \dX, $ for which K-theory  $ K _ * (\dX) $ is completely trivial.
Moreover, we also assume that a phantom noncommutative scheme has a geometric realization in the form of an admissible subcategory in the category of perfect complexes on a smooth projective scheme.
In the paper \cite{GO} it was proved that  phantoms exist and the procedure of their construction, connected with the product of surfaces of the general type with $p_g = q = 0,$ was described.
In the paper \cite{BGKS} one of the phantoms was constructed as an admissible subcategory in the category of perfect complexes on the Barlow surface.
In the next section \ref{Kru_Sch} we discuss the so-called Krull--Schmidt partners, which were introduced in the paper \cite{O_Krull}.
These are smooth and proper noncommutative schemes $ \dX $ and $ \dX',$ for which there exists a smooth proper noncommutative scheme $\dY$ with the condition that some gluings
$ \dX \underset {\mT} {\oright} \dY $ and $ \dX'\underset {\mT'} {\oright} \dY $ are isomorphic to each other.
In particular, we give a new procedure for constructing smooth and proper noncommutative schemes that are Krull--Schmidt partners for usual schemes and have the same additive invariants.

The last two sections are devoted to the study of geometric realizations for finite dimensional algebras.
Any finite dimensional algebra $ \La $ gives a noncommutative scheme $\dV = \prfdg\La$ that is proper.
In the paper \cite{O_glue} it was proved that for any such noncommutative scheme $ \dV $ under the condition that the semisimple part $ \overline {\La} = \La/\rd $ is separable over the field $ \kk, $
one can find a geometric realization $ \prfdg \La \to \prfdg X, $ for which the scheme $X$ is smooth and projective.
There was also given an explicit construction of such a geometric realization.
On the other hand, in this case of special interest are such geometric realizations for which the image of the algebra $ \La $ is not an arbitrary perfect complex,  but some vector bundle on a scheme $ X.$
This problem can be reformulated as follows: for an arbitrary finite dimensional algebra $\La, $ find and construct a smooth projective scheme $X$ and a vector bundle $ \E $ on it such that
$ \End_X (\E) \cong \La $ and $ \Ext^j_{X} (\E, \E) = 0 $ for all $ j> 0.$
In the last section such construction is suggested for an arbitrary basic algebra, i.e. for an algebra $ \La, $ the semisimple part $ \overline {\La} = \La / \rd $ of which
is the product of the base field $ \underbrace {\kk \times \cdots \times \kk}_{m}.$
If the field $ \kk $ is algebraically closed, then this result implies a positive answer for any finite dimensional algebra.

The construction of a smooth projective scheme $ X $ and a vector bundle $ \E $ takes place in two steps and uses the theory of quasi-hereditary algebras.
In the section \ref{quasi_her_alg} a new concept of well-formed quasi-hereditary algebra is introduced, and for such algebras we construct  very special geometric realizations that
send standard modules to line bundles.
This construction is a generalization of the construction for quiver algebras described in the paper \cite{O_q}.
Applying this new procedure to the algebra $ \Ga$ that is obtained from a basic finite dimensional algebra $ \La $ by the Auslander construction,
we obtain  a geometric realization for the algebra $ \La $ such that $ \La $ goes to a vector bundle $ \E $ on a smooth projective scheme $ X.$
We also note that the scheme $ X $ is a tower of projective bundles, and the rank of the bundle $ \E $ is exactly equal to the dimension of the algebra $ \La.$

The author is very grateful to Anton Fonarev, Alexander Kuznetsov, and Amnon Neeman  for useful discussions and valuable comments.

\section{Preliminaries on triangulated and differential graded categories}

\subsection{Triangulated categories, generators, and semi-orthogonal decompositions}

Let $\T$ be a triangulated category. We say that a set of objects $S\subset\T$ {\sf classically generates} the triangulated category $\T$
if the smallest full triangulated subcategory of
$\T$ containing $S$ and  closed under taking direct summands coincides with the whole category $\T.$ If the set $S$ consists of a single object $E\in \T,$
then $E$ is called a {\sf classical generator} for $\T.$

A classical generator will be called {\sf strong} if it generates the whole triangulated category $\T$ in a finite number of steps (see, e.g. \cite{BVdB}).
To define it precisely,  we introduce a multiplication on the set of strictly full subcategories. Let $\I_1$ and $\I_2$ be two full subcategories of $\T.$
Denote by $\I_1*\I_2$ the full subcategory of $\T$
consisting of all objects $X$ such that there is a distinguished triangle $X_1\to X\to X_2$ with $X_i\in \I_i.$
For any subcategory $\I\subset\T$  denote by $\langle \I\rangle$ the smallest full subcategory of $\T$ containing $\I$ and closed under
taking finite direct sums, direct summands and shifts.
Now we can define a new multiplication on the set of strictly full subcategories closed under
finite direct sums. Put $\I_1 \diamond\I_2=\langle \I_1*\I_2\rangle$ and  define by induction
$\langle \I\rangle_k=\langle\I\rangle_{k-1}\diamond\langle \I\rangle_1,$ where $\langle \I\rangle_1=\langle \I\rangle.$
 If $\I$ consists of a single object $E,$ we denote $\langle \I\rangle$ by
 $\langle E\rangle_1$ and put by induction $\langle E\rangle _k=\langle E\rangle_{k-1}\diamond\langle E\rangle_1.$

\begin{definition}
An object $E$ will be called  a {\sf strong generator} if $\langle E\rangle_n=\T$ for some $n\in\NN.$
\end{definition}

Note that $E$ is a classical generator if and only if
$\mathop\bigcup\limits_{k\in\ZZ} \langle E\rangle_k=\T.$
It is evident that if a triangulated category $\T$ has a strong generator then any
classical generator of $\T$ is  strong too, i.e. the existence of a strong generator is a property of a triangulated category (see \cite{O_gen}).

\begin{definition}\label{reg}
A triangulated category $\T$ will be called  {\sf regular} if it has a strong generator.
\end{definition}

Following \cite{Ro}, we  introduce a notion of  dimension for a regular triangulated
category $\T$ as the smallest integer $d\ge 0$ such that there exists an object $E\in \T$
for which $\langle E\rangle_{d+1}=\T.$

Now we recall the notion of a compact object. An object $E$ of a triangulated category $\T$
is called {\sf compact} (in $\T$) if the functor $\Hom
_{\T}(E, -)$ commutes with arbitrary existing in $\T$ direct sums (coproducts), i.e. for each
family of objects $\{ X_i\}\subset \T,$ for which
$\bigoplus_i X_i$ exists, the canonical map
$\bigoplus_i\Hom (E, X_i){\to}
\Hom (E, \bigoplus_i X_i)$ is an isomorphism.
Compact objects form a full triangulated subcategory of $\T$ which is usually denoted as $\T^c\subset \T.$

Let $\T$ be a triangulated category that admits arbitrary (small) direct sums. A full triangulated subcategory $\L\subseteq\T$
 which is closed
under taking all direct sums is called a {\sf localizing subcategory}.
This means that the inclusion functor preserves direct sums. Note that $\L$ is also closed under taking direct summands (see \cite{Ne}).

A set $S \subset \T^c$ is
called a set of {\sf compact generators} if the smallest localizing subcategory containing the set $S$
coincides with the whole category $\T.$ This property is equivalent to
the following: for an object $X\in\T,$ we have $X\cong 0$ if $\Hom(Y, X[n])=0$ for all $Y\in S$ and  all $n\in \ZZ.$

Let $\T$ be a triangulated category with small Hom-sets, i.e Hom between any two objects should be a set.
Assume that $\T$ admits arbitrary direct sums and let
$\L\subset\T$ be a localizing triangulated subcategory.
We can consider the Verdier quotient $\T/\L$ with a natural localization map $\pi : \T\to \T/\L.$
It is known that the category $\T/\L$ also has arbitrary direct sums and, moreover,
the functor $\pi$ preserves direct sums (see \cite[3.2.11]{Ne_book}).

Notice however that Hom-sets in $\T/\L$ need not be small. Assume that the Verdier quotient $\T/\L$ is a
category with small Hom-sets. If the triangulated category $\T$ has a set of compact generators
then the Brown representability theorem holds for
$\T$ and the quotient functor $\pi : \T\to \T/\L$ has a right adjoint $\mu : \T/\L\to \T$ (see \cite[8.4.5]{Ne_book}).
This adjoint is
called the {\sf Bousfield localization} functor.

Let $j\colon\N\hookrightarrow\T$ be a full embedding of triangulated
categories.
The subcategory $\N$ is called {\sf right admissible}
(resp. {\sf left admissible}) if there is a right
(resp. left) adjoint $q\colon\T\to \N$ to the embedding functor $j\colon\N\hookrightarrow\T.$ The
subcategory $\N$ is called {\sf admissible} if it is both right
and left admissible.

The {\sf right orthogonal} (resp. {\sf left
orthogonal}) to the subcategory ${\N}\subset \T$ is the full
subcategory ${\N}^{\perp}\subset {\T}$ (resp. ${}^{\perp}{\N}$) consisting of all objects $M$
such that ${\Hom(N, M)}=0$ (resp. ${\Hom(M, N)}=0$)  for any $N\in{\N}.$
It is clear that the subcategories ${\N}^{\perp}$ and ${}^{\perp}{\N}$ are triangulated subcategories.

\begin{definition}\label{sd}
A {\sf semi-orthogonal decomposition} of a triangulated category
$\T$ is a sequence of full triangulated subcategories ${\N}_1,
\dots, {\N}_n$ in ${\T}$ such that there is an increasing filtration
$0=\T_0\subset\T_1\subset\cdots\subset\T_n=\T$ by left admissible
subcategories for which the left orthogonals ${}^{\perp}\T_{r-1}$ in
$\T_{r}$ coincide with $\N_r$ for all $1\le r\le n.$
 We write $ {\T}=\left\langle{\N}_1, \dots,
{\N}_n\right\rangle. $
\end{definition}

In some cases one can hope that $\T$ has a semi-orthogonal decomposition
${\T}=\left\langle{\N}_1, \dots, {\N}_n\right\rangle$ in which each
$\N_r$ is as simple as possible, i.e. each of them is equivalent to the bounded
derived category of finite-dimensional vector spaces.

From now we will assume that $\T$ is a $\kk$\!--linear triangulated category, where $\kk$ is an arbitrary base field.

\begin{definition}\label{exc}
An object $E$ of a $\kk$\!--linear triangulated category ${\T}$ is
called {\sf exceptional} if  ${\Hom}(E, E[m])=0$ whenever $m\ne 0,$
and ${\Hom}(E, E)\cong\kk.$ An {\sf exceptional collection} in ${\T}$ is
a sequence of exceptional objects $\sigma=(E_1,\dots, E_n)$
satisfying the semi-orthogonality condition ${\Hom}(E_i, E_j[m])=0$
for all $m$ whenever $i>j.$
\end{definition}

If a triangulated category $\T$ has an exceptional collection
$\sigma=(E_1,\dots, E_n)$ that generates the whole of $\T,$ then
this collection is called {\sf full}.  In this case $\T$ has a
semi-orthogonal decomposition with $\N_r=\langle E_r\rangle.$ Since
$E_{r}$ is exceptional, each of these categories is equivalent to
the bounded derived category of finite dimensional $\kk$\!--vector
spaces.  In this case we write $ \T=\langle E_1,\dots, E_n \rangle.$

Recall now the notion of a proper triangulated category.

\begin{definition}\label{prop}
We say that a $\kk$\!--linear triangulated category  $\T$ is {\sf proper} if the vector space
$\bigoplus_{m\in\ZZ}\Hom(X, Y[m])$ is finite-dimensional for any pair of objects $X, Y\in\T.$
\end{definition}

Proper and regular triangulated categories have good properties. In particular, they are saturated and admissible if they are idempotent complete.
Recall that $\T$ is idempotent complete if kernels of all projectors $p: X\to X,\; p^2=p$ exist as objects of $\T.$
The following theorem is due to A.~Bondal and M.~Van den Bergh.
\begin{theorem}{\rm \cite[Th. 1.3]{BVdB}}\label{saturated}
Let $\T$ be a regular and proper triangulated category that is idempotent complete (Karoubian).
Then every cohomological functor  from $\T^{\op}$ to the category of finite-dimensional
vector spaces is representable, i.e. it has the form $h^{Y}=\Hom(-, Y).$
\end{theorem}
A triangulated category satisfying such property of representability  is called {\sf right saturated} in \cite{BK, BVdB}.
It is proved \cite[2.6]{BK} that if a right saturated  triangulated category $\T$ is a full subcategory
in  a proper triangulated category, then it is right admissible there.
By Theorem \ref{saturated} a regular and proper
idempotent complete triangulated category is right saturated. Since the opposite category is also regular and proper,
it is left saturated as well. Thus, we obtain the following proposition.
\begin{proposition}\label{admissible}
Let $\N\subset\T$ be a full triangulated subcategory in  a proper triangulated category
$\T.$ Assume that $\N$ is regular and  idempotent complete.  Then $\N$ is admissible in $\T.$
\end{proposition}

Recall now the definition of a Serre functor \cite{BK, BO, BO2}.
Let $\T$ be a proper $\kk$\!--linear triangulated category. A $\kk$\!--linear autoequivalence  $S: \T\to\T$
is called a {\sf Serre functor} if there exists
an isomorphism of bifunctors
\begin{equation}\label{serre_funct}
\Hom_{\T}(Y,\; SX)\stackrel{\sim}{\lto}{\Hom}_{\T}(X, Y)^*,
\end{equation}
where $V^*$ is the dual vector space for a space $V.$
If such a functor exists it is exact and unique up to a natural isomorphism (see \cite{BK}).
In the paper \cite{BK} it was proved that any saturated triangulated category $\T$ has a Serre functor.
Taking into account Theorem \ref{saturated}, we obtain the following proposition.
\begin{proposition}\label{serre}
Let $\T$ be a regular and proper $\kk$\!--linear triangulated category that is idempotent complete.
Then it has a Serre functor.
\end{proposition}

\subsection{Differential graded categories}
In this section we recall some facts on differential graded (DG) categories.
The main references are \cite{Ke, Ke2, Dr, LO}.
Let $\kk$ be a field and let $\dC$ be a $\kk$\!--linear differential graded (DG) category.
This means that the morphism spaces
$\dHom_{\dC} (\mX, \mY)$
are complexes of $\kk$\!--vector spaces (DG $\kk$\!--modules) and for any $\mX, \mY, \mZ\in
\Ob\dC$ the composition $\dHom (\mY, \mZ)\otimes \dHom (\mX, \mY)\to
\dHom (\mX, \mZ)$ is a morphism of DG $\kk$\!--modules.

For any DG category $\dC$ we denote by $\Ho(\dC)$
its homotopy category. The homotopy category $\Ho(\dC)$ has the same objects as the DG category $\dC,$ and its
morphisms are defined by taking the $0$\!--th cohomology
$H^0(\dHom_{\dC} (\mX, \mY))$
of the complex $\dHom_{\dC} (\mX, \mY).$

As usual, a {\sf DG functor}
$\mF:\dA\to\dB$ is given by a map $\mF:\Ob(\dA)\to\Ob(\dB)$ and
by morphisms of DG $\kk$\!--modules
$$
\mF_{\mX, \mY}: \dHom_{\dA}(\mX, \mY) \lto \dHom_{\dB}(\mF \mX,\mF \mY),\quad \mX, \mY\in\Ob(\dA)
$$
compatible with the composition and the units.

A DG functor $\mF: \dA\to\dB$ is called a {\sf quasi-equivalence} if
$\mF_{\mX, \mY}$ is a quasi-isomorphism for all pairs of objects $\mX, \mY$ of $\dA,$
and the induced functor $H^0(\mF): \Ho(\dA)\to \Ho(\dB)$ is an
equivalence. DG categories $\dA$ and $\dB$ are called {\sf quasi-equivalent} if there exist DG
categories $\dC_1,\dots, \dC_n$ and a chain of quasi-equivalences
$\dA\stackrel{\sim}{\leftarrow} \dC_1 \stackrel{\sim}{\rightarrow} \cdots \stackrel{\sim}{\leftarrow} \dC_n
\stackrel{\sim}{\rightarrow} \dB.$ Actually, for any DG category $\dA$
we can find a (quasi-equivalent) cofibrant replacement $\dA_{cof}\stackrel{\sim}{\rightarrow}\dA$ such that any chain of quasi-equivalences
between $\dA$ and $\dB$ can be realized by a simple roof $\dA\stackrel{\sim}{\leftarrow} \dA_{cof} \stackrel{\sim}{\rightarrow}\dB$ (see \cite{Ke2, Ta}).

Let $\dA$ be a small $\kk$\!--linear DG category.
A {\sf (right) DG $\dA$\!--module} is a DG functor
$\mM: \dA^{\op}\to \Mod \kk,$ where $\Mod \kk$ is the DG category of complexes of $\kk$\!--vector spaces and
$\dA^{\op}$ is the opposite DG category.
We denote by $\Mod \dA$  the DG
category of right DG $\dA$\!--modules. Let $\Ac\dA$ be the full
DG subcategory of $\Mod \dA$ consisting of all acyclic DG modules, i.e. DG modules $\mM$
for which the complex of vector spaces $\mM(\mX)$ has trivial cohomology for all $\mX\in\dA.$
The
homotopy category $\Ho(\Mod\dA)$ has a natural structure of triangulated category
and $\Ho (\Ac\dA)$ forms a localizing triangulated subcategory in it.

\begin{definition}
The {\sf derived
category} $\D(\dA)$ (of DG $\dA$\!--modules)  is defined as the Verdier quotient
$$
\D(\dA):=\Ho(\Mod\dA)/\Ho (\Ac\dA).
$$
\end{definition}

Any object $\mY\in\dA$ defines a representable right DG module
$
\mh^\mY_{\dA}(-):=\dHom_{\dA}(-, \mY).
$
This
gives the Yoneda DG functor
$\mh^\bullet :\dA \to
\Mod\dA$ that is a full embedding.
A DG module is called {\sf free} if it is isomorphic to a direct sum of  DG modules of the form
$\mh^\mY[m].$
A DG module
$\mP$ is called {\sf semi-free} if it has a filtration
$0=\mPhi_0\subset \mPhi_1\subset ...=\mP$
with free quotients  $\mPhi_{i+1}/\mPhi_i.$
The full
DG subcategory of semi-free DG modules is denoted by $\SF\dA.$

It is also natural to consider the category of h-projective DG modules.
A  DG $\dA$\!--module $\mP$ is called {\sf h-projective (homotopically projective)} if
$$\dHom_{\Ho(\Mod\dA)}(\mP, \mN)=0$$
for every acyclic DG module $\mN.$ (Dually, we can define {\sf h-injective} DG modules.)
Let  $\dP(\dA)\subset \Mod\dA$ denote the full
DG subcategory of h-projective objects. It can be easily checked that a semi-free
DG-module is h-projective and the natural embedding $\SF\dA\hookrightarrow\dP(\dA)$
is a quasi-equivalence.
Moreover, for any DG
$\dA$\!--module $\mM$ there is a quasi-isomorphism $\bp \mM\to \mM$ such that $\bp \mM$ is a semi-free
DG $\dA$\!--module (see \cite{Ke, Hi, Dr}). Hence,
the canonical DG functors $\SF\dA\hookrightarrow\dP(\dA)\hookrightarrow\Mod\dA$ induce equivalences
$\Ho(\SF\dA)\stackrel{\sim}{\to} \Ho(\dP(\dA))\stackrel{\sim}{\to} \D(\dA)$ of triangulated categories.
Dually, it can be shown that for any DG
$\dA$\!--module $\mM$ there is a quasi-isomorphism $\mM\to \bi \mM$ such that $\bi \mM$ is h-injective
(see \cite{Ke}).

We denote by $\SFf\dA\subset \SF\dA$ the full DG subcategory of finitely generated semi-free
DG modules, i.e. $\mPhi_n=\mP$ for some $n$ and $\mPhi_{i+1}/\mPhi_i$ is a finite direct sum of DG modules of the form
$\mh^Y[m]$ for any $i.$

\begin{definition} The DG category of {\sf perfect DG modules} $\prfdg\dA$
is the full DG subcategory of $\SF\dA$ consisting of all DG modules that are isomorphic  to  direct summands of objects from $\SFf\dA$
in the homotopy category $\Ho(\SF\dA).$
\end{definition}

Denote by $\prf\dA$ the homotopy category $\Ho(\prfdg\dA).$ It is triangulated
and it is equivalent to the subcategory of compact objects $\D(\dA)^c\subset \D(\dA)$ (see \cite{Ne, Ke2}).

To any DG category $\dA$ we can associate a DG category $\dA^{\ptr}$ that is called the {\sf pretriangulated hull}
and a canonical fully faithful DG
functor $\dA\hookrightarrow\dA^{\ptr}$ (see \cite{BK2, Ke2}).
The idea of the definition of $\dA^{\ptr}$ is to add to $\dA$
all shifts, all cones, cones of morphisms between cones and etc.
There is a canonical fully faithful DG functor
(the Yoneda embedding) $\dA^{\ptr}\to \Mod\dA,$ and under this embedding
$\dA^{\ptr}$ is equivalent to the DG category of finitely generated semi-free DG modules $\SFf\dA.$
If $\dA$ is small, then the pretriangulated hull $\dA^{\ptr}$ is also small, and in some sense it is a small version for the essentially small DG category $\SFf\dA.$

A DG category $\dA$ is called {\sf pretriangulated}  if the canonical DG functor $\dA\to\dA^{\ptr}$
is a quasi-equivalence.
This property is equivalent to requiring that the homotopy category $\Ho(\dA)$ is triangulated
as a subcategory of $\Ho(\Mod\dA).$
The DG category $\dA^{\ptr}$ is always  pretriangulated,
so $\Ho(\dA^{\ptr})$ is a triangulated category.

If $\dA$ is pretriangulated and $\Ho(\dA)$ is idempotent complete, then
the Yoneda functor $\mh^{\bullet}: \dA\to\prfdg\dA$ is a quasi-equivalence and, hence, the induced exact functor
$h: \Ho(\dA)\to \prf\dA$ is an equivalence of triangulated categories.

\begin{definition} Let $\T$ be a triangulated category. An {\sf
enhancement} of $\T$ is a pair $(\dA , \varepsilon),$ where $\dA$ is a
pretriangulated DG category and $\varepsilon:\Ho(\dA)\stackrel{\sim}{\to} \T$ is an equivalence of triangulated categories.
\end{definition}
Thus, the DG category $\SF\dA$ of semi-free DG modules is an enhancement of the derived category $\D(\dA)$ while the DG category $\prfdg\dA$ of perfect DG modules
is an enhancement of the triangulated category $\prf\dA.$

\subsection{Differential graded functors and quasi-functors}

Let $\mF:\dA \to \dB$ be a DG functor between small DG categories.
It induces the restriction DG functor
$
\mF_*:\Mod\dB\lto \Mod\dA
$
which sends a DG $\dB$\!--module $\mN$ to $\mN\ldot\mF.$

The restriction functor $\mF_*$ has left and right adjoint functors $\mF^*, \mF^{!}$ that are defined as follows:
\[
\mF^*\mM(\mY)=\mM\otimes_{\dA} \mF_* \mh_{\mY},\quad \mF^{!}\mM(\mY)=\dHom_{\Mod\dA}(\mF_*\mh^{\mY}, \mM),
\]
where $\mY\in \dB,\; \mM\in\Mod\dA,$ and $\mh^\mY(-):=\dHom_{\dB}(-, \mY)$ is a right DG $\dB$\!--module, while
$\mh_\mY(-):=\dHom_{\dB}(\mY, -)$ is a left DG $\dB$\!--module.
The DG functor $\mF^*$ is called the induction functor and it is an extension of $\mF$ on the category of DG modules, i.e there is an isomorphism of DG functors
$\mF^* \ldot\mh^\bullet_{\dA}\cong \mh^\bullet_{\dB}\ldot\mF.$

The DG functor $\mF_*$ preserves acyclic DG modules and induces a derived functor $F_*: \D(\dB)\to \D(\dA).$
By adjunctions, the DG functors $\mF^*, \mF^{!}$ preserve h-projective and h-injective DG modules, respectively.
Existence of h-projective and h-injective resolutions allows us to define derived functors
$\bL F^*$ and $\bR F^!$ from $\D(\dA)$ to $\D(\dB).$ For example, the derived functor $\bL F^*: \D(\dA)\to \D(\dB)$ is the induced homotopy functor
$\H^0(\mF^*)$ for the extension DG functor $\mF^*: \SF\dA\to\SF\dB.$

More generally, let $\mT$ be an
$\dA\hy\dB$\!--bimodule, that is (by definition) a DG-module over $\dA^{op}\otimes\dB.$
For each DG $\dA$\!--module $\mM$ we obtain a DG $\dB$\!--module
$\mM\otimes_{\dA} \mT.$
The DG functor $(-)\otimes_{\dA} \mT: \Mod\dA \to \Mod\dB$ admits a right adjoint
$\dHom_{\dB} (\mT, -).$
These functors do not respect
quasi-isomorphisms in general, but applying them to h-projective (h-injective) DG modules
we obtain an adjoint pair of  derived functors
$(-)\stackrel{\bL}{\otimes}_{\dA}\mT$ and
$\bR \Hom_{\dB} (\mT, -)$
between the derived categories $\D(\dA)$ and
$\D(\dB)$ (see \cite{Ke, Ke2}).

It is evident that the categories $\D(\dA)$ and $\prf\dA$ are invariant under quasi-equivalences of DG categories.
Moreover, if a DG functor $\mF: \dA\to\dB$ is a quasi-equivalence, then the functors
$\mF^{*}:\prfdg\dA\lto \prfdg\dB$ and  $\mF^*:\SF\dA\lto \SF\dB$
are quasi-equivalences too.

Furthermore, we have the
following proposition  that is essentially equivalent to
Lemma 4.2 in \cite{Ke} (see also  \cite[Prop. 1.15]{LO}).

\begin{proposition}\cite{Ke}\label{Keller2}
Let $\mF: \dA\hookrightarrow \dB$ be a full embedding of DG categories
and let $\mF^*:\prfdg\dA\to\prfdg\dB$ (resp. $\mF^{*}:\SF\dA\to\SF\dB$) be the extension DG functor. Then the derived
functor $\bL F^*: \prf\dA\to \prf\dB$ (resp. $\bL F^*: \D(\dA)\to \D(\dB)$) is fully faithful.

If, in addition, the category $\prf\dB$ is classically generated by $\Ob\dA,$ then
$\bL F^*$ is an equivalence and the DG functor $\mF^*:\prfdg\dA\to\prfdg\dB$ (resp. $\mF^{*}:\SF\dA\to\SF\dB$) is a quasi-equivalence.
\end{proposition}

\begin{remark}\label{rem_quasi}{\rm
Applying this proposition to the case $\dB=\prfdg\dA$ we obtain  quasi-equivalences $\dB=\prfdg\dA\stackrel{\sim}{\to}\prfdg\dB$ and
$\SF\dA\stackrel{\sim}{\to}\SF\dB$ that induce an equivalence between the derived categories $\D(\dA)$ and $\D(\prfdg\dA).$
}
\end{remark}

On the other hand, we can also consider the restriction DG functor
$\mF_*:\Mod\dB\lto \Mod\dA$ and  the induced derived functor $F_*: \D(\dB)\to \D(\dA).$
 The functor $F_*$ is right adjoint to the derived functor $\bL F^*: \D(\dA)\to \D(\dB).$
The composition of DG functor $\mF_*$ with the Yoneda DG functor $\mh^\bullet_{\dB}$ gives a DG functor
$\dB\to \Mod\dA$ and a homotopy functor $\Ho(\dB)\to\D(\dA).$ As result we obtain the following proposition,
a proof of which can be found in \cite{LO}.

\begin{proposition}\label{Kell1}
Let $\mF: \dA\hookrightarrow \dB$ be a full embedding of DG categories. Assume that $\Ob\dA$ forms a set
of compact generators of $\D(\dB).$ Then the derived functor $F_{*}: \D(\dB)\to \D(\dA)$ is an equivalence.
If, in addition, $\dB$ is pretriangulated and the homotopy category $\Ho(\dB)$ is idempotent complete, then the derived functor $F_{*}$ induces an equivalence between $\Ho(\dB)$ and the triangulated category
$\prf\dA,$ and the DG categories $\dB$ and $\prfdg\dA$ are quasi-equivalent.
\end{proposition}

Let $\dA$ and $\dB$ be two small DG categories. Since we consider DG categories up to quasi-equivalence, it is natural to consider  morphisms from $\dA$ to $\dB$
as roofs $\dA\stackrel{\sim}{\leftarrow} \dA_{cof}{\to}\dB,$ where $\dA\stackrel{\sim}{\leftarrow} \dA_{cof}$ is a cofibrant replacement (see, e.g. \cite{Ke2}).
These sets of morphisms are much better described in term of quasi-functors.

\begin{definition}
A DG $\dA\hy\dB$\!--bimodule $\mT$ is called a {\sf quasi-functor}
from $\dA$ to $\dB$ if the tensor functor
$
(-)\stackrel{\bL}{\otimes}_\dA \mT: \D(\dA) \to \D(\dB)
$
takes every representable DG $\dA$\!--module to an object which is
isomorphic in $\D(\dB)$ to a representable $\dB$\!--module.
\end{definition}

In other words, a quasi-functor is represented by a DG functor $\dA\to \Mod\dB$ whose essential image consists of
quasi-representable DG $\dB$\!--modules, where ``quasi-representable'' means quasi-isomorphic to a representable DG module.
Since the category of quasi-representable DG
$\dB$\!--modules is equivalent to $\Ho(\dB),$ a quasi-functor
$\mT$ defines a functor $\Ho(\mT):\Ho(\dA)\to \Ho(\dB).$
By the same reasons the restriction of the tensor functor $(-)\stackrel{\bL}{\otimes}_\dA \mT$ on the
category of perfect modules $\prf\dA$ induces a functor $\prf\dA\to \prf\dB.$
In particular case, when $\dA$ is a DG algebra, the quasi-functor $\mT$ considered as DG $\dB$\!--module is quasi-isomorphic to a representable
DG $\dB$\!--module and, hence, it is perfect $\dB$\!--module.

Denote by  $\Rep(\dA,\; \dB)$ the full subcategory of the derived
category $\D(\dA^{op}\otimes\dB)$ of $\dA\hy\dB$\!--bimodules consisting of
all quasi-functors.

It is known  that the morphisms from $\dA$ to $\dB$
in the localization of the category of all small DG $\kk$\!-–linear
categories with respect to the
quasi-equivalences are in a natural bijection with the isomorphism
classes of $\Rep(\dA,\dB)$ (see \cite{To}).
Due to this theorem, any morphism of the form $\dA\stackrel{\sim}{\leftarrow} \dA_{cof}\to\dB$
will be called a quasi-functor.

Let $\mF: \dA\to\dB$ be a quasi-functor. It can be realized as
a roof  $\dA\stackrel{\ma}{\longleftarrow}\dA_{cof}\stackrel{\mF'}{\lto}\dB,$
where $\ma$  and $\mF'$ are DG functors, and $\ma$ is a quasi-equivalence.
The quasi-functor $\mF$ induces functors
\begin{equation}\label{derived_quasi}
\bL F^*= \bL F^{'*}\ldot a_*: \D(\dA)\lto \D(\dB)
\quad
\text{and}
\quad
\bR F_*:= \bR a^{!}\ldot F'_*: \D(\dB)\lto \D(\dA).
\end{equation}

Since $\ma$ is a quasi-equivalence, the functor $a_*:\D(\dA)\to \D(\dA_{cof})$ is an equivalence, and the functor
$\bR a^{!}$ is quasi-inverse to $a_*.$ Hence, the right adjoint functor $\bR a^!$ is isomorphic to the left adjoint $\bL a^{*},$
and $\bR F_*\cong \bL a^{*}\ldot F'_*.$ Thus, we conclude that the functor $\bR F_*$  has a right adjoint functor
\[
\bR F^!=\bR F^{'!}\ldot a_*: \D(\dA)\lto \D(\dB)
\]
and, hence, commutes with all direct sums.

The inverse image functor $\bL F^*: \D(\dA)\to \D(\dB)$ induces the functor $\bL F^*: \prf\dA\to \prf\dB$ between subcategories of perfect modules, while
the functors $\bR F_*, \bR F^!$ do not necessarily send perfect modules to perfect modules.

If now we consider the quasi-functor $\mF$ as a DG $\dA\hy\dB$\!--bimodule $\mT,$ then
there are isomorphisms of functors
\[
\bL F^*\cong (-)\stackrel{\bL}{\otimes}_\dA \mT: \D(\dA) \lto \D(\dB)
\quad
\text{and}
\quad
\bR F_*\cong \bR\Hom_{\dB}(\mT, -): \D(\dB) \lto \D(\dA).
\]

Actually, quasi-functors from $\dA$ to $\dB$ form a DG category $\dRep(\dA, \dB)$ that can be defined as the full DG subcategory of the DG category of semi-free DG $\dA\hy\dB$\!--bimodules $\SF(\dA^{\op}\otimes\dB)$
consisting of all the objects of $\Rep(\dA, \dB).$ It follows from the definition that $\Ho(\dRep(\dA, \dB))\cong \Rep(\dA, \dB).$
The DG category of all quasi-functors $\dRep(\dA, \dB)$ can be also described as $\dR\lHom(\dA,  \dB),$ where
$\dR\lHom$ is an internal Hom-functor in the localization of the category of all small DG $\kk$\!-–linear
categories with respect to
quasi-equivalences (see \cite{To}).

\section{Derived noncommutative schemes, their properties, and geometric realizations}

\subsection{Derived noncommutative schemes and their properties}

Let $X$ be a quasi-compact and quasi-separated scheme over a field $\kk.$
Let $\Qcoh X$ denote the abelian category of quasi-coherent sheaves on $X.$
To any such scheme one can associate the derived category of $\cO_X$\!--modules with quasi-coherent cohomology
$\D_{\Qcoh}(X).$ It  admits arbitrary
direct sums. It is also known and proved in \cite{Ne, BVdB} that the subcategory of compact objects in this category coincides with
the subcategory of perfect complexes $\prf X$. Recall that a complex of $\cO_X$\!--modules on a scheme is called {\sf perfect} if it is
locally quasi-isomorphic to a bounded complex of locally free sheaves of finite type.

In \cite{Ne, BVdB} it was proved that the category $\prf X$ admits a classical generator.
Let us take such a generator $\mR$ as an object of the DG category $\prfdg X.$ Denote by $\dR$ its DG algebra of endomorphisms,
i.e. $\dR=\dHom(\mR, \mR).$
Proposition \ref{Keller2} implies that the DG category $\prfdg X$ is quasi-equivalent to
$\prfdg \dR.$ In this case, as a corollary,  we also obtain an equivalence between derived categories $\D(\dR)\stackrel{\sim}{\to}\D_{\Qcoh}(X)$ and the
triangulated categories $\prf\dR\to\prf X.$  Since $\mR$ is a perfect complex, the DG algebra $\dR$ has only finitely many nontrivial cohomology groups.
This fact allows us to suggest a definition of a derived noncommutative scheme over $\kk.$

\begin{definition}\label{noncommutative_scheme}
A {\sf derived noncommutative scheme} $\dX$ over a field $\kk$ is a $\kk$\!--linear DG category of the form $\prfdg\dR,$
where $\dR$ is a cohomologically  bounded DG algebra over $\kk.$ The derived category
$\D(\dR)$ is called the derived category of quasi-coherent sheaves on this noncommutative scheme while the triangulated category $\prf\dR$
is called the category of perfect complexes on it.
\end{definition}
Henceforth, for shortness, we will sometimes omit the adjective ``derived'' sometimes and we will refer to such object ``noncommutative scheme''.

For any noncommutative scheme $\dX$ we have the opposite noncommutative scheme $\dX^{\op}$ that is the DG category $\prfdg\dR^{\op},$
where $\dR^{\op}$ is the opposite DG algebra. We can also define the tensor product $\dX\otimes_{\kk}\dY$ of noncommutative schemes
$\dX=\prfdg\dR$ and
$\dX=\prfdg\dS$ as the derived noncommutative scheme $\prfdg(\dR\otimes_{\kk}\dS).$

Let us consider and discuss some natural properties of noncommutative schemes.
\begin{definition}\label{prop_def}
A noncommutative scheme $\dX=\prfdg\dR$ will be called {\sf proper} if the triangulated category $\prf\dR$ is proper, i.e. the
$\kk$\!--vector spaces
$\bigoplus_{p\in\ZZ}H^p(\dHom_{\prfdg\dR}(\mM, \mN))$ are finite-dimensional for any two perfect DG modules $\mM, \mN\in\prfdg\dR.$
\end{definition}

This property can be described in terms of the DG algebra $\dR.$ It can be checked that the noncommutative scheme $\prfdg\dR$ is proper if and only if
the cohomology algebra $\bigoplus_{p\in\ZZ}H^p(\dR)$ is finite dimensional.
It can be shown that Definition \ref{prop_def} is consistent with the usual concept of a proper scheme.

\begin{proposition}\label{prop_scheme}\cite[Prop 3.30]{O_glue} Let $X$ be a separated scheme of finite type
over a field $\kk.$ Then $X$ is proper if and only if the category
of perfect complexes $\prf X$ is proper.
\end{proposition}

Another fundamental property of usual commutative schemes that can be extended to noncommutative schemes is regularity.
\begin{definition} A noncommutative scheme $\dX=\prfdg\dR$ is called {\sf regular} if the triangulated category $\prf\dR$ is regular, i.e. it has
a strong generator.
\end{definition}
It was proved in \cite{Ne_str} that for a quasi-compact and separated scheme $X$ the triangulated category  $\prf X$ is regular if
and only if $X$ can be covered by open affine subschemes $\Spec(R_i),$ where each $R_i$ has finite
global dimension. There is also a short proof of this fact for a separated noetherian scheme over $\kk$ of finite Krull dimension  with noetherian square $X\times_{\kk} X$ (see \cite[Th.3.27]{O_glue}).

Regularity of a scheme is closely related to another important property that is called smoothness.
However, smoothness depends on the base field $\kk.$
A small $\kk$\!--linear DG category $\dA$ is called {\sf $\kk$\!--smooth} if it is perfect as a DG bimodule, i.e. as a DG module over  $\dA^{\op}\otimes_{\kk}\dA$ (see \cite{KS}). Thus, we obtain a definition of smoothness for noncommutative schemes:

\begin{definition}
A noncommutative scheme $\dX=\prfdg\dR$ is called {\sf $\kk$\!--smooth} if the the DG category $\prfdg\dR$ is $\kk$--smooth, i.e
it is perfect as a DG bimodule.
\end{definition}

Smoothness is invariant under Morita equivalence \cite{LS}.
This means that if $\D(\dA)$ and $\D(\dB)$
are equivalent through a functor of the form $(-)\stackrel{\bL}{\otimes_{\dA}}\mT,$ where
$\mT$ is an $\dA\hy\dB$\!--bimodule, then $\dA$ is smooth if and only if $\dB$ is smooth.
Thus, the DG category $\prfdg\dR$ is smooth if and only if $\dR$ is smooth.
It is proved in \cite{Lu} that a smooth DG category $\dA$ is regular. Thus, a smooth noncommutative scheme is also regular.

A usual commutative scheme $X$ over a field $\kk$ is called {\sf smooth} if it is of finite type and the scheme
$\widebar{X}=X\otimes_{\kk} \widebar{\kk}$ is regular, where $\widebar{\kk}$ is an algebraic closure
of $\kk.$
It is proved in \cite{Lu} that a separated scheme  $X$ of finite type
 is smooth if and only if the DG category
$\prfdg X$ is smooth (see also \cite{LS2, O_glue}).
Thus, we have defined and  can talk about smooth, regular and proper noncommutative schemes.

For any two derived noncommutative schemes $\dX$ and $\dY$ we can consider the tensor product $\dX\otimes_{\kk}\dY.$
If both $\dX$ and $\dY$ are proper, then $\dX\otimes_{\kk}\dY$ is also proper. It can be also shown that the tensor product
$\dX\otimes_{\kk}\dY$ is smooth when $\dX$ and $\dY$ are smooth. However, the tensor product of regular schemes is not necessarily regular
even for usual commutative schemes.

\subsection{Morphisms of derived noncommutative schemes}

Let $\dX=\prfdg\dR$ and $\dY=\prfdg\dS$ be two derived
noncommutative schemes over an arbitrary field $\kk.$
\begin{definition}
A {\sf morphism} $\mf:\dX \to \dY$ of noncommutative schemes  is a quasi-functor
$\mF:\prfdg\dS\to \prfdg\dR.$
\end{definition}
With any usual morphism $f: X\to Y$ of commutative schemes $X$ and $Y$ one can associate the inverse image functor $f^*: \prfdg Y \to \prfdg X.$
Therefore, any morphism between commutative schemes induces a morphism between the corresponding noncommutative schemes.
Meanwile,  in the derived noncommutative world we have a lot of additional morphisms between commutative schemes because there are many other quasi-functors between DG categories of perfect  complexes.

Let $\dX=\prfdg\dR$ and $\dY=\prfdg\dS$ be two noncommutative schemes and let $\mf:\dX \to \dY$  be a morphism, i.e.  a quasi-functor
$\mF:\prfdg\dS\to \prfdg\dR.$ Any such morphism induces  derived functors
\[
\bL \mf^*:=\bL F^*: \D(\dS)\lto \D(\dR)
\quad
\text{and}
\quad
\bR\mf_*:=\bR F_*: \D(\dR)\lto \D(\dS)
\]
that are defined for any quasi-functor $\mF$ in (\ref{derived_quasi}) and will be called the inverse image and the direct image functors, respectively.
We also have the functor $\mf^{!}:=\bR F^{!}$ that is  right adjoint to $\bR\mf_*.$
The inverse image functor $\bL \mf^*$ sends perfect modules to perfect modules, and its restriction to $\prf\dS$ is isomorphic to the homotopy functor
$\H^0(\mF):\prf\dS\to \prf\dR.$ The direct image functor $\bR\mf_*: \D(\dR)\to \D(\dS)$ commutes with arbitrary direct sums.

The most important morphisms for us are those for which the inverse image functor $\bL \mf^*$ is fully faithful.

\begin{definition} A morphism $\mf:\dX \to \dY$ of noncommutative schemes is called an {\sf ff-morphism (fully faithful morphism)}
if the inverse image functor $\bL\mf^*$ is fully faithful.
\end{definition}
Note that the functor $\bL\mf^*:\D(\dS)\to \D(\dR)$ is fully faithful if and only if its restriction
$\bL\mf^*:\prf\dS\to \prf\dR$ is fully faithful (see, e.g., Proposition \ref{Keller2} and Remark \ref{rem_quasi}). Furthermore, for any ff-morphism $\mf:\dX \to \dY$ the functor
$\mf^!:\D(\dS)\to \D(\dR)$ is fully faithful too, because the composition functor $\bR \mf_* \mf^!$ is right adjoint to $\bR \mf_*\bL \mf^*\cong\id.$

It is easy to see that by projection formula a morphism $f:X\to Y$ between usual commutative schemes is an ff-morphism if and only if the direct image $\bR f_*\cO_X$ is isomorphic to $\cO_Y.$ Only in this case the inverse image functor
$\bL f^*: \D_{\Qcoh}( Y)\to \D_{\Qcoh}( X)$ is fully faithful.

Another class of morphisms that can be extended to derived noncommutative schemes is the so called perfect proper morphisms.

\begin{definition} A morphism $\mf:\dX \to \dY$ of derived noncommutative schemes is called a {\sf pp-morphism (perfect proper morphism)}
if the direct image functor $\bR\mf_*$ sends perfect modules to perfect ones.
\end{definition}

This also means that the inverse image functor $\bL \mf^*$ as a functor from $\prf\dS$ to $\prf\dR$ has a right adjoint
$\bR \mf_*: \prf\dR\to\prf\dS.$ As a consequence, the right adjoint to the corresponding to $\mf$ quasi-functor $\mF$ induces a quasi-functor
$\mG: \prfdg\dR\to\prfdg\dS$ for which $\H^0(\mG)\cong \bR \mf_*.$
Thus, in this case we obtain a morphism $\mg:\dY\to\dX$ that can be called a ``right adjoint'' morphism to the morphism  $\mf.$
If a morphism $\mf$ is an ff-morphism and a pp-morphism simultaneously, we obtain an isomorphism of morphisms $\mf\ldot\mg\cong\mId_{\dY}.$

Perfect morphisms of schemes  were defined in \cite[III]{Il} as  pseudo-coherent morphisms of locally finite Tor-dimension.
For a locally noetherian scheme $Y$ a pseudo-coherent morphism
$f:X\to Y$ is the same as a morphism locally of finite type and in this case if $f$ is perfect and proper then the direct image functor $\bR f_*$ sends perfect complexes on $X$
to perfect complexes on $Y$ (see \cite[III]{Il} and \cite[2.5.5]{TT}).

We can also define other types of morphisms. For example, we can say that a morphism $\mf:\dX \to \dY$ is {\sf quasi-affine} if the image of perfect modules $\prf\dS$ under the  inverse image functor $\bL\mf^*$
classically generate the category $\prf\dR.$ Taking into account Proposition \ref{Keller2} we see that a quasi-affine ff-morphism  $\mf:\dX \to \dY$ induces an equivalence
$\bL\mf^*:\prf\dS\to\prf\dR$ and  a quasi-equivalence between the DG categories $\prfdg\dS$ and $\prfdg\dR.$
Hence, we conclude that the noncommutative schemes $\dX$ and $\dY$ are isomorphic in this case.

In commutative situation smooth projective varieties $X$ and $Y$ that have equivalent triangulated categories $\prf X$ and $\prf Y$ are called Fourier-Mukai partners.
Since any equivalence is represented by an object on the product \cite{O_main, LO},  the DG categories $\prfdg X$ and $\prfdg Y$ are quasi-equivalent, and Fourier-Mukai partners
determine the same derived noncommutative scheme. The most famous example due to S.~Mukai \cite{Mu} and is given by an abelian variety $A$ and the dual abelian variety $\widehat{A}$ (see also \cite{O_ab, O_diss}).

One can introduce the {\sf triangulated category of cohomologically bounded pseudo-coherent complexes} on $\dX$ and denote it by $\D^b(\coh\dX),$ because it is equivalent to
a bounded derived category of coherent sheaves on a noetherian scheme in the commutative case.
At first, let us consider the full triangulated subcategory $\D^b(\dR)\subset\D(\dR)$  consisting of all
DG $\dR$\!--modules $\mN$ cohomology of which are bounded.
Now we say that an object $\mM\in\D^b(\dR)$ belongs to  the full subcategory $\D^b(\coh\dX)\subset\D^b(\dR)$
if for any sufficiently large $N\in\NN$ there are a perfect module $\mP\in\prf\dR$ and a morphism $\mP\to \mM$ such that the induced maps $H^k(\mP)\to H^k(\mM)$ are isomorphisms for all $k\ge -N$ (see, e.g., \cite{Ne_new}).
Through the definition of cohomology $H^k(\mM)=\Hom_{\D(\dR)}(\dR, \mM)$ depends on $\dR,$ the definition of a pseudo-coherent module does not depend on the choice of a classical generator for $\prf\dR.$ It is evident that there is a natural inclusion of triangulated categories
$\prf\dR\subseteq\D^b(\coh\dX).$

Any morphism $\mf:\dX \to \dY$ is a quasi-functor $\mF:\prfdg\dS\to \prfdg\dR$ and is given by a DG $\dS\hy\dR$\!--bimodule $\mT.$
Since $\dR$ is cohomologically bounded and $\mT$ is perfect as DG $\dR$\!--module, the direct image functor
$\bR\mf_*=\bR\Hom_{\dR}(\mT, -)$ sends $\D^b(\dR)$ to $\D^b(\dS)$ for any morphism $\mf.$
Now we can say that a morphism $\mf: \dX \to \dY$ is of {\sf finite Tor-dimension} if the inverse image functor $\bL\mf^*:\D(\dS)\to \D(\dR)$ sends $\D^b(\dS)$ to $\D^b(\dR)$ and we  say that $\mf$ is {\sf proper} if the direct image functor $\bR\mf_*$ sends $\D^b(\coh\dX)$ to $\D^b(\coh\dY).$
\begin{definition}\label{im_morph}
A morphism of derived noncommutative schemes $\mf:\dX \to \dY$ is called an {\sf immersion} if the direct image functor
$\bR\mf_*: \D(\dR)\to\D(\dS)$ is fully faithful.
\end{definition}
In this case the derived category $\D(\dR)$ can be obtained as a Bousfield localization (and colocalization) of $\D(\dS)$ (see, e.g. \cite[Ch.9]{Ne_book}).
An example of such a morphism in commutative algebraic geometry  is given by a usual open immersion $j: U\hookrightarrow Y.$

\begin{definition}\label{comp}
Let $\dX$ be a noncommutative scheme. A {\sf compactification} of $\dX$ is a morphism $\mf: \dX \to \widebar{\dX}$
such that $\mf$ is an immersion, and the noncommutative scheme $\widebar\dX$ is proper.
\end{definition}
Another important notion is a resolution of singularities, or desingularization.

\begin{definition}\label{des}
Let $\dX$ be a derived noncommutative scheme. A {\sf regular (resp. smooth) desingularization} of $\dX$ is a morphism $\mf: \widetilde\dX \to \dX$
such that $\mf$ is an ff-morphism, and the noncommutative scheme $\widetilde\dX$ is regular (resp. smooth).
\end{definition}

We have to note that a usual resolution of singularities $f:\widetilde X\to X$ of a commutative scheme is not necessarily a desingularization in the noncommutative case, because we ask that
the inverse image functor $\bL f^*$ is fully faithful. This condition is fulfilled only if $\bR f_* \cO_{\widetilde X}\cong \cO_X,$ i.e. if $X$ has rational singularities.
However, it follows from the main theorem of \cite[Th.1.4]{KL} that a  separated scheme of finite type   $X$  over a field of characteristic $0$
has a smooth desingularization as in Definition \ref{des}.
By the construction of this  resolution, there is a quasi-functor $\mG: \prfdg X\to\dD,$
where $\dD$ is a gluing of DG categories of perfect complexes on smooth proper schemes, and
the induced homotopy functor $G:\prf Y\to \D$ is fully faithful.

\subsection{Categories of morphisms and Serre functors}
It was discussed above that quasi-functors form a DG category. Therefore, morphisms between derived noncommutative schemes $\dX$ and $\dY$ also form a DG category
which will be denoted by $\dMor(\dX, \dY).$ For any derived noncommutative schemes $\dX=\prfdg\dR$ and $\dY=\prfdg\dS$ we have a quasi-equivalences of DG categories:
\[
\dMor(\dX, \dY):=\dRep(\prfdg\dS, \prfdg\dR)\cong\dRep(\dS, \prfdg\dR)\subset \SF(\dS^{\op}\otimes_{\kk}\dR).
\]

The DG category $\dMor(\dX, \dY)$ is pretriangulated and there is a triangulated category of morphisms $\Mor(\dX, \dY)=\H^0(\dMor(\dX, \dY)).$
In particular, we can add any two morphisms and consider morphisms between morphisms.
It is evident that the DG category $\dMor(\dX, \pt)$ of morphisms from $\dX$ to the point $\pt$ is quasi-equivalent to $\prfdg\dR.$
Besides, the DG category $\dMor(\pt, \dX)$ consists of DG $\dR^{\op}$\!--modules $\mN$ that are perfect as complexes of $\kk$\!--vector spaces, i.e.
$\dim_{\kk} \bigoplus_i H^i(\mN)< \infty.$

By Theorem \ref{saturated}, if a triangulated category $\T$ is regular, proper, and idempotent complete, then
any exact functor from $\T^{\op}$ to $\prf\kk$ is representable.
The proof of  Theorem \ref{saturated} (see \cite{BVdB}) works for DG categories without any changes and, moreover, can be deduced from it.

\begin{proposition}\label{dg_saturated}\cite[Th. 3.18]{O_glue} Let $\dA$ be a small DG category that is regular and proper.
Then a DG module $\mM$ is perfect if and only if $\dim\bigoplus_i  H^i(\mM(X))<\infty$
for all $X\in\dA.$
\end{proposition}
In particular, we obtain the following corollary.
\begin{corollary}\label{representable}
Let $\dY$ be a derived noncommutative scheme that is regular and proper. Then there is a quasi-equivalence $\dMor(\pt, \dY)\cong \dY^{\op}.$
\end{corollary}

If the noncommutative scheme $\dY=\prfdg\dS$ is proper, then the DG category $\dMor(\dX, \dY)$ contains $\prfdg (\dS^{\op}\otimes_{\kk}\dR)$ as a DG subcategory,
 because the perfect DG $\dS\hy\dR$\!--bimodule $\dS\otimes_{\kk}\dR,$ which classically generates the category $\prfdg (\dS^{\op}\otimes_{\kk}\dR),$  is perfect as a DG $\dR$\!--module too.

Besides, if a noncommutative scheme $\dY$ is smooth, then we have an opposite inclusion $\dMor(\dX, \dY)\subseteq\prfdg (\dS^{\op}\otimes_{\kk}\dR).$
Indeed, for a smooth $\dY$ the DG $\dS\hy\dS$\!--bimodule $\dS$ is perfect. The category of perfect $\dS\hy\dS$\!--bimodules is classically  generated by the bimodule
$\dS\otimes_{\kk}\dS.$ Thus, any DG $\dS\hy\dR$\!--bimodule $\mT\cong \dS\otimes_{\dS}\mT$ belongs to the subcategory
generated by the DG $\dS\hy\dR$\!--bimodule
\[
(\dS\otimes_{\kk}\dS)\otimes_{\dS}\mT\cong\dS\otimes_{\kk}\mT.
\]
Suppose that $\mT$ is perfect as a DG $\dR$\!--module, then $\dS\otimes_{\kk}\mT$ is perfect as DG $\dS\hy\dR$\!--bimodule, and, hence,  $\mT$ is perfect as DG $\dS\hy\dR$\!--bimodule too. Thus, we obtain an inclusion $\dMor(\dX, \dY)\subseteq\prfdg (\dS^{\op}\otimes_{\kk}\dR).$
Finally, we obtain the following proposition:
\begin{proposition} Let $\dY$ be  a derived noncommutative scheme that is smooth and proper.   Then there is a natural quasi-equivalence
\begin{equation}
\dMor(\dX, \dY):=\prfdg (\dS^{\op}\otimes_{\kk}\dR)\cong \dY^{\op}\otimes_{\kk}\dX.
\end{equation}
for any derived noncommutative scheme $\dX.$
\end{proposition}

At the same time, there is the following theorem due to B.~To\"en.
\begin{theorem}\cite{To}\label{toen}
Let $X$ and $Y$ be smooth projective schemes over a field $\kk.$
Then there is a canonical quasi-equivalence
\[
\dMor(X, Y)\cong\dRep(\prfdg Y,\; \prfdg X)\cong\prfdg\,(Y\times_{\kk} X).
\]
In particular, the DG category $\prfdg\, (Y\times_{\kk} X)$ is quasi-equivalent to the DG category of perfect
DG  modules over $(\prfdg Y)^{\op}\otimes_{\kk}(\prfdg X).$
\end{theorem}

The quasi-equivalence between the DG category $\prfdg (Y\times_{\kk} X)$ and
the DG category of perfect DG  modules over $(\prfdg Y)^{\op}\otimes_{\kk}(\prfdg X)$  can be described explicitly.
Consider DG functors
\[
\mpr_1^*: \prfdg Y\lto \prfdg(Y\times_{\kk} X),\quad \text{and}\quad \mpr_2^*: \prfdg X\lto \prfdg (Y\times_{\kk} X)
\]
induced by the projections $pr_1:Y\times_{\kk} X\to Y$ and $pr_2:Y\times_{\kk} X\to X.$
For any perfect complex $\mE$ on the product $Y\times_{\kk} X$ we can define a bimodule
$\mT_{\mE}$ by the rule
\[
\mT_{\mE}(\mN, \mM)\cong \dHom_{\prfdg(Y\times X)}(\mpr_2^* \mM,\; \mpr_1^* \mN\otimes \mE),
\quad\text{where}
\quad
\mM\in \prfdg X,\; \mN\in\prfdg Y.
\]
This is exactly the required quasi-equivalence.

Let $\dX=\prfdg\dR$ be a derived noncommutative scheme. If it is proper and regular, then, by Proposition \ref{serre}, the triangulated category $\prf\dR$ has a Serre
functor. Recall that an autoequivalence $S_{\dX}$ is a Serre functor  if it induces bifunctorial isomorphisms
\[
\Hom_{\prf\dR}(\mN, S_{\dX}(\mM))\stackrel{\sim}{\lto}\Hom_{\prf\dR}(\mM, \mN)^*.
\]
for any $\mM, \mN\in \prf\dR$ (see \cite{BK, BO, BO2} and also \cite{Shk} for DG case).
For a usual regular projective scheme  $X$ the Serre functor is isomorphic to $(-)\otimes\omega_X[n],$ where $\omega_X$ is the canonical sheaf and $n$ is the dimension of $X.$

For any right DG $\dR$\!--module $\mM$ we can define left DG $\dR$\!--modules $\mM^*:=\dHom_{\kk}(\mM, \kk)$ and
$\mM^{\vee}:=\dHom_{\dR}(\mM, \dR).$
Consider the DG $\dR\hy\dR$\!--bimodule $\dR^*:=\dHom_{\kk}(\dR, \kk)$ and the derived functor $(-)\stackrel{\bL}{\otimes}_{\dR}\dR^*$
from $\D(\dR)$ to itself.
If $\mP$ is a perfect DG $\dR$\!--module, then $(\mP^{\vee})^{\vee}\cong \mP$ and  there is an isomorphism $\mP\stackrel{\bL}{\otimes}_{\dR}\dR^*\cong(\mP^{\vee})^*,$ because
\[
\mP\stackrel{\bL}{\otimes}_{\dR}\dR^*\cong \dHom_{\dR}(\dR, \mP\stackrel{\bL}{\otimes}_{\dR}\dR^*)\cong\dHom_{\dR}(\mP^{\vee}, \dHom_{\kk}(\dR, \kk))\cong\dHom_{\kk}(\mP^{\vee}, \kk)\cong(\mP^{\vee})^*.
\]

Let us consider the functor $(-)\stackrel{\bL}{\otimes}_{\dR}\dR^*$ as a functor from $\prf\dR$ to $\D(\dR).$
There is the following sequence of isomorphisms
\[
\Hom_{\dX}(\mP, \mM)^*\cong \Hom_{\kk}(\mM\stackrel{\bL}{\otimes}_{\dR}\mP^{\vee}, \kk)\cong\Hom_{\dX}(\mM, \dHom_{\kk}(\mP^{\vee}, \kk))\cong\Hom_{\dX}(\mM, \mP\stackrel{\bL}{\otimes}_{\dR}\dR^*),
\]
where the notation $\Hom_{\dX}$ is used for a space of morphisms in the triangulated category $\prf\dR.$
These isomorphisms show us that the functor $(-)\stackrel{\bL}{\otimes}_{\dR}\dR^*: \prf\dR\to\D(\dR)$ induces Serre duality. If it sends $\prf\dR$ to itself,
 then we obtain a Serre functor on the category $\prf\dR.$
This happens when $\dR^*$ is a perfect (right) DG $\dR$\!--module.

Furthermore, a perfect DG module $\mP\in\prfdg\dR$ defines a morphism $\mf: \dX\to\pt.$ By the definition of this morphism, we have $\bL\mf^*\kk\cong\mP.$ Let us consider
the functor $\mf^! :\D(\kk)\to\D(\dR)$ and apply it to $\kk.$ The following sequence of isomorphisms
\[
\Hom_{\dX}(\mM, \mf^{!}\kk)\cong\Hom_{\kk}(\bR\mf_{*}\mM, \kk)\cong\Hom_{\kk}(\kk, \bR\mf_{*}\mM)^*\cong\Hom_{\dX}(\bL\mf^*\kk, \mM)^*\cong\Hom_{\dX}(\mP, \mM)^*
\]
shows us that the object $\mf^!\kk$ is isomorphic to $\mP\stackrel{\bL}{\otimes}_{\dR}\dR^*.$ When the category $\prf\dR$ has a Serre functor $S_{\dX},$ we obtain that $S_{\dX}(\mP)\cong S_{\dX}(\bL\mf^*\kk)\cong \mf^!\kk.$
For example, this holds when $\dX$ is regular and proper.

Let us consider a morphism $\mf: \dX\to\dY$ of noncommutative schemes $\dX=\prfdg\dR$ and $\dY=\prfdg\dS$ and assume that the category $\prf\dS$ has a Serre functor $S_{\dY}.$
Then there is the following sequence of isomorphisms
\[
\Hom_{\dX}(\bL\mf^*\mN, \mM)^*\cong\Hom_{\dY}(\mN, \bR\mf_{*}\mM)^*\cong\Hom_{\dY}(\bR\mf_{*}\mM, S_{\dY}\mN)\cong\Hom_{\dX}(\mM, \mf^{!}S_{\dY}\mN),
\]
where $\mN\in\prf\dS$ and $\mM\in\prf\dR.$
It shows that there is a relation between the Serre functor on $\dY$ and a Serre functor on $\dX$ if it exists.
Thus, we obtain the following proposition:

\begin{proposition}\label{serre_prop}
Let $\mf: \dX\to\dY$ be a morphism of noncommutative schemes $\dX=\prfdg\dR$ and $\dY=\prfdg\dS,$ which possess  Serre functors $S_{\dX}$ and $S_{\dY},$ respectively.
Then for any $\mN\in\prfdg\dS$ there is an isomorphism
\begin{align}\label{serre_rel}
S_{\dX}(\bL\mf^*\mN)\cong \mf^!(S_{\dY}\mN).
\end{align}
in the triangulated category $\prf\dR.$
\end{proposition}

Serre functor is an intrinsic invariant of a triangulated category that, in some cases, allow to talk about certain notion of dimension for noncommutative schemes and many other things.
For example, a proper  noncommutative scheme $\dX$ will be called a {\sf weak Calabi-Yau variety of dimension $m$} if it has a Serre functor $S_{\dX}$ that is isomorphic to the shift
functor $[m].$

\subsection{Geometric realizations of derived noncommutative schemes}

The most interesting derived noncommutative schemes appear as full DG subcategories of the DG categories $\prfdg Z,$ where $Z$ is some usual commutative scheme.
Moreover, any such realization carries an important geometric meaning and gives us a new way to look at these derived noncommutative schemes.

\begin{definition}\label{geom_real}
A {\sf geometric realization} of a  derived noncommutative scheme $\dX=\prfdg\dR$
is a usual commutative scheme $Z$ and a localizing subcategory $\L\subseteq \D_{\Qcoh}(Z)$ such that its natural enhancement
$\dL$ is quasi-equivalent
to $\SF\dR.$
\end{definition}

Thus, the derived category $\D(\dR)$ is equivalent to the localizing subcategory $\L$ and there is a fully faithful
functor $\D(\dR)\to \D_{\Qcoh}(Z)$  the image of which coincides with the subcategory $\L.$ Moreover, since this inclusion functor preserves direct sums and
the category $\D(\dR)$ is compactly generated, there exists a right adjoint to the inclusion functor as a consequence of Brown representability theorem (see \cite[9.1.19]{Ne_book}).
This implies that the category $\L\cong\D(\dR)$ can be realized as a Verdier quotient (localization) of the category $\D_{\Qcoh}(Z).$

The most important class of  geometric realizations is given by the ff-morphisms $\mf: Z\to\dX.$
In this case the inverse image functor
$\bL \mf^*: \D(\dR)\to \D_{\Qcoh}(Z)$ gives a geometric realization for $\dX.$
The functor $\bL \mf^*$ gives a full embedding $\prf\dR\hookrightarrow\prf Z.$
Thus, it is exactly a case when the localizing subcategory $\L$ is compactly generated by some perfect complexes on $Z$
and, hence, there is an inclusion of subcategories of compact objects $\L^c\cong\prf\dR\subset \prf Z.$

Conversely, if we have an inclusion $\L^c\cong\prf\dR\subset \prf Z,$ then it gives
an ff-morphism $\mf: Z\to\dX.$ In this case the inverse image functor $\bL \mf^*:\D(\dR)\to \D_{\Qcoh}(Z)$ gives a geometric realization
for $\dX$ and the direct image functor $\bR\mf_*$ realizes $\D(\dR)$
as a localization of the category $\D_{\Qcoh}(Z).$ Moreover, it is a so-called Bousfield localization because there exists
a right adjoint functor $\mf^{!}: \D(\dR)\to \D_{\Qcoh}(Z).$ Such geometric realizations will be called {\sf {plain}}.
If, in addition, the morphism $\mf$ is a pp-morphism, i.e. the inclusion $\prf\dR\subset \prf Z$ has a right adjoint, then such a geometric realization will be called
{\sf perfectly plain.}

\begin{example}
{\rm
Let $\pi:\wt{X}\to X$ be a proper birational morphism that is a resolution of singularities of $X.$ Assume that $\bR\pi_{*}\cO_{\wt{X}}\cong\cO_X.$
Then $\pi$ is an ff-morphism and the inverse image functor $\bL\pi^*$ gives a geometric realization for $X.$ This realization is plain but it is not perfectly plain when $X$ is singular.
}
\end{example}

Another class of geometric realizations is given by immersive morphisms $\mj: \dX\to Z$ (see Definition \ref{im_morph}).
In this situation
the direct image functor $\bR \mj_* : \D(\dR)\to \D_{\Qcoh}(Z)$ is fully faithful and the inverse image functor
$\bL \mj^*: \D_{\Qcoh}(Z)\to \D(\dR)$ is a Bousfield localization. Such geometric realizations will be called {\sf {immersive}}.

Suppose now that an immersive morphism $\mj: \dX\to Z$ is a pp-morphism, i.e. $\bR \mj_*$ is fully faithful and it sends perfect modules to perfect ones.
Thus, there is a quasi-functor $\mG:\prfdg\dR\to\prfdg Z$ such that $\H^0(\mG)\cong \bR \mj_*.$ The quasi-functor $\mG$ gives an ff-morphism $\mg: Z\to \dX$ by itself.
 We have an isomorphism $\mg\ldot\mj\cong\mId_{\dX}.$ In this case the geometric realization is connected with a pair of adjoint morphisms $(\mj, \mg),$ where $\bR \mj_*\cong \bL\mg^*$ is fully faithful. It will be  called {\sf perfectly immersive}.

 If, in addition, the morphism $\mg$ is a pp-morphism, then
 the subcategory $\prf\dR\subset \prf Z$ is admissible. In this case
the geometric realization is perfectly plain and perfectly immersive simultaneously and it will be called {\sf {pure}}.

\begin{example}{\rm
Let $j: U\to X$ be an open immersion. Then the functor $\bR j_*$ is fully faithful and give a geometric realization for $U$ in $X.$ This realization is immersive but it is not perfectly immersive in general.
}
\end{example}

Many interesting examples of noncommutative schemes naturally appear as admissible subcategories
$\N\subset\prf X$ for some smooth projective scheme $X.$
More precisely, for any such subcategory we can consider its DG enhancement $\dN\subset\prfdg X.$
It is a DG category that has a generator and, hence, can be realized as $\prfdg \dE$ for some cohomologically bounded
DG algebra $\dE.$ Indeed, since $\N$ is admissible in $\prf X,$  the inclusion functor has right and left adjoint projections.
Thus, a projection of a generator in $\prf X$ to $\N$ gives a classical generator for $\N.$
Moreover, the noncommutative scheme $\dN$ is proper being a full subcategory of the proper category $\prfdg X.$
Furthermore, it is regular, because a projection of a strong generator gives a strong generator in $\N.$
It also can be shown that the noncommutative scheme $\dN$ is smooth as an admissible subcategory of the smooth category $\prfdg X$ (see Proposition \ref{smooth_glue}).
Note that, by construction, the derived noncommutative scheme $\dN\subset\prfdg X$ is coming with geometric realization and, moreover, this geometric realization is pure
because $\N$ is admissible.

\begin{example}{\rm
Let $X$ be a proper scheme such that $H^0(X, \cO_X)\cong\kk$ and $H^i(X, \cO_X)=0,$ when $ i>0.$ Then the structure sheaf $\cO_X$ is an exceptional.
The subcategory $\langle\cO_X\rangle\subset \prf X$ is admissible and gives a pure geometric realization for the point  $\pt.$ Let us consider
the left and right orthogonals $\N\cong{}^{\perp}\langle\cO_X\rangle$ and $\M\cong \langle\cO_X\rangle^{\perp},$ respectively.
The right orthogonal $\M$ is left admissible in $\prf X.$ It defines a derived noncommutative scheme $\dM\subset\prfdg X$ together with a perfectly immersive geometric realization.
Besides the left orthogonal $\N$ is right admissible in $\prf X.$ It defines a derived noncommutative scheme $\dN\subset\prfdg X$ together with a perfectly plain geometric realization.
When $X$ is smooth, the geometric realizations of $\dM$ and $\dN$ are pure as well. However, for a singular scheme $X$ these realizations are not necessarily pure.
}
\end{example}

All other geometric realizations, which are neither  plain nor immersive,  will be called {\sf {mixed}}. We can obtain different mixed realizations as  compositions
of functors of the form $\bL\mf^*$ and $\bR \mj_*$ for ff-morphisms $\mf$ and immersive morphisms $\mj,$ where the target category of the last functor is
$\D_{\Qcoh}(Z)$ for some commutative scheme $Z.$
Any such composed functor
preserves all direct sums and has a right adjoint functor that is the composition of the functors of the form $\bR\mf_*$ and $\mj^{!}.$

\section{Gluing of derived noncommutative schemes and geometric realizations}

\subsection{Gluing of differential graded categories}
Let $\dA$ and $\dB$ be two small DG categories and let $\mT$ be a DG $\dB\hy\dA$\!--bimodule,
i.e. a DG  $\dB^{\op}\otimes\dA$\!--module.
We construct the so called lower triangular DG category corresponding to  the data
$(\dA, \dB; \mT).$

\begin{definition}\label{upper_tr}
Let $\dA$ and $\dB$ be two small DG categories and let $\mT$ be a DG $\dB\hy\dA$\!--bimodule.
The {\sf lower triangular} DG category $\dC=\dA\underset{\mT}{\with}\dB$ is defined as follows:
\begin{enumerate}
\item[1)] $\Ob(\dC)=\Ob(\dA)\bigsqcup\Ob(\dB),$

\item[2)]
$
\dHom_{\dC}(\mX, \mY)=
\begin{cases}
 \dHom_{\dA}(\mX, \mY), & \text{ when $\mX, \mY\in\dA$}\\
     \dHom_{\dB}(\mX, \mY), & \text{ when $\mX, \mY\in\dB$}\\
      \mT(\mY, \mX), & \text{ when $\mX\in\dA, \mY\in\dB$}\\
      0, & \text{ when $\mX\in\dB, \mY\in\dA$}
\end{cases}
$
\end{enumerate}
with the composition law coming from the DG categories $\dA, \dB$ and the bimodule structure on $\mT.$
\end{definition}

The lower triangular DG category $\dC=\dA\underset{\mT}{\with}\dB$ is not necessarily pretriangulated even if the components $\dA$ and $\dB$
are pretriangulated. To make this operation well-defined on the class of pretriangulated categories, we introduce gluing
of pretriangulated categories (see \cite{Ta, KL, Ef, O_glue}).

\begin{definition}\label{gluing_cat}
Let $\dA$ and $\dB$ be two small pretriangulated DG categories and let $\mT$ be a DG $\dB\hy\dA$\!--bimodule.
The {\sf gluing} $\dA\underset{\mT}{\oright}\dB$ of DG categories $\dA$ and $\dB$ via $\mT$ is defined as the pretriangulated hull of $\dA\underset{\mT}{\with}\dB,$ i.e. $\dA\underset{\mT}{\oright}\dB=(\dA\underset{\mT}{\with}\dB)^{\ptr}.$
\end{definition}

The natural fully faithful DG inclusions $\ma: \dA\hookrightarrow \dA\underset{\mT}{\with}\dB$ and
$\mb: \dB\hookrightarrow \dA\underset{\mT}{\with}\dB$ induce the fully faithful extension DG
functors $\ma^*: \dA\hookrightarrow \dA\underset{\mT}{\oright}\dB$ and
$\mb^*: \dB\hookrightarrow \dA\underset{\mT}{\oright}\dB.$
These quasi-functors induce exact functors
\[
a^*:\Ho(\dA)\lto\Ho(\dA\underset{\mT}{\oright}\dB), \quad b^*:\Ho(\dB)\lto\Ho(\dA\underset{\mT}{\oright}\dB)
\]
between triangulated categories which are fully faithful.
The following proposition is almost obvious.

\begin{proposition}\label{dg_semiorhtogonal} \cite[Prop. 3.7]{O_glue} Let the DG category $\dE$ be a gluing $ \dA\underset{\mT}{\oright}\dB.$
Then the DG functors $\ma^*: \dA\to\dE$ and $\mb^*: \dB\to\dE$ induce a semi-orthogonal decomposition
for the triangulated category $\Ho(\dE)$ of the form
$\Ho(\dE)=\langle\Ho(\dA), \Ho(\dB)\rangle.$
\end{proposition}

Furthermore, we can show that any enhancement of a  triangulated category with a semi-orthogonal decomposition
can be obtained as a gluing of enhancements of the summands.

\begin{proposition}\label{gluing_semi-orthogonal}\cite[Prop 3.8]{O_glue}
Let $\dE$ be a pretriangulated DG category. Suppose we have a semi-orthogonal decomposition
$\Ho(\dE)=\langle \A, \B\rangle.$ Then the DG category $\dE$ is quasi-equivalent to a gluing
$\dA\underset{\mT}{\oright}\dB,$ where $\dA, \dB\subset\dE$ are full DG subcategories
with the same objects as $\A$ and $\B,$ respectively, and the DG $\dB\hy\dA$\!--bimodule is given by the rule
\begin{equation}\label{bimodule}
\mT ( \mY, \mX)=\dHom_{\dE}(\mX, \mY), \quad \text{with}\quad \mX\in\dA \;\text{and}\; \mY\in\dB.
\end{equation}
\end{proposition}

Actually, we can show much more. The following proposition is not very difficult to prove.

\begin{proposition}\label{gluing_quasifunctors}\cite[Prop 3.11]{O_glue}
Let $\ma:\dA\to \dA'$ and $\mb:\dB\to \dB'$ be  quasi-functors between small DG categories. Let $\mT$ and $\mT'$ be
DG modules over $\dB^{\op}\otimes\dA$ and ${\dB'}^{\op}\otimes\dA',$ respectively. Suppose there is a map
$\phi:\mT\to \bR (\mb\otimes\ma)_*\mT'$ in $\D(\dB^{\op}\otimes\dA).$ Then there are quasi-functors
\[
\ma\underset{\phi}{\with}\mb: \dA\underset{\mT}{\with}\dB \lto \dA'\underset{\mT'}{\with}\dB'
\quad
\text{and}\quad
\ma\underset{\phi}{\oright}\mb:\dA\underset{\mT}{\oright}\dB \lto \dA'\underset{\mT'}{\oright}\dB'.
\]
Moreover, assume that $\phi$ is a quasi-isomorphism and the homotopy functors $a:\Ho(\dA)\to\Ho(\dA')$ and $b:\Ho(\dB)\to \Ho(\dB')$ are fully faithful.
Then the induced functors
\[
a\underset{\phi}{\with}b: \Ho(\dA\underset{\mT}{\with}\dB) \lto \Ho(\dA'\underset{\mT'}{\with}\dB')
\quad
\text{and}\quad
a\underset{\phi}{\oright}b:\Ho(\dA\underset{\mT}{\oright}\dB) \lto \Ho(\dA'\underset{\mT'}{\oright}\dB')
\]
are fully faithful too.
If $\ma, \mb$ are quasi-equivalences, then both
$\ma\underset{\phi}{\with}\mb$ and $\ma\underset{\phi}{\oright}\mb$ are quasi-equivalences.
\end{proposition}

It is easy to see that the restriction functor $\mb_*: \Mod\,(\dA\underset{\mT}{\with}\dB) \to \Mod\dB$
sends semi-free DG modules to semi-free DG modules as well as finitely generated semi-free DG modules to finitely generated ones.
Thus, we obtain a DG functor
$\SFf(\dA\underset{\mT}{\with}\dB)\to\SFf\dB.$ By assumption $\dB$ is pretriangulated, and
we know that a pretriangulated hull is quasi-equivalent to the DG category of finitely generated semi-free DG modules.
Thus, we obtain a quasi-functor $\mb_*: \dA\underset{\mT}{\oright}\dB \to \dB$ that is  right adjoint to
$\mb^*.$

\subsection{Gluing of derived noncommutative schemes}\label{glu_der_schem}

Using the construction above we can define  gluing of derived noncommutative schemes.
Let $\dX=\prfdg\dR$ and $\dY=\prfdg\dS$ be two derived noncommutative schemes and let $\mT$ be a DG $\dY\hy\dX$\!--bimodule. By the construction above
we can define a DG category
\[
\dZ:=\dX\underset{\mT}{\oright}\dY:=\prfdg\dR\underset{\mT}{\oright}\prfdg\dS
\]
which will be called the {\sf gluing} of $\dX$ and $\dY$ via (or with respect to) $\mT.$
Since by Proposition \ref{Keller2} and Remark \ref{rem_quasi} there is a quasi-equivalence between semi-free DG $\dY\hy\dX$\!--bimodules and semi-free DG $\dS\hy\dR$\!--bimodules,
the DG category $\dZ$ is quasi-equivalent to $\prfdg\,(\dR\underset{\mT}{\with}\dS).$ (We use the same letter $\mT$ for a DG $\dY\hy\dX$\!--bimodule and for its restriction as the DG $\dS\hy\dR$\!--bimodule.)
The DG category $\dZ$ is a derived noncommutative scheme according to our definition if and only if the DG algebra $\dR\underset{\mT}{\with}\dS$ is cohomologically bounded, i.e. the  DG bimodule $\mT$ belongs to the bounded derived category $\D^b(\dS^{\op}\otimes\dR)$ of $\dS\hy\dR$\!--bimodules.
When $\mT\cong \m0,$ the gluing will be denoted by $\dX\oplus\dY,$ and it is the biproduct of $\dX$ and $\dY$ in the ``world'' of derived noncommutative schemes.

The natural inclusions $\prfdg\dR\hookrightarrow\prfdg\,(\dR\underset{\mT}{\with}\dS)$ and $\prfdg\dS\hookrightarrow \prfdg\,(\dR\underset{\mT}{\with}\dS)$ define morphisms
$\mp_{\dX}:\dZ\to\dX$ and $\mp_{\dY}:\dZ\to\dY$
 that will be called projections. Both $\mp_{\dX}$ and $\mp_{\dY}$ are ff-morphisms of noncommutative schemes and, moreover, the projection $\mp_{\dY}$ is a pp-morphisms, because
 the direct image functor $\bR\mp_{\dY*}$ sends perfect objects to perfect ones. This also implies that we have a right adjoint morphism $\mr_{\dY}:\dY\to\dZ$ which is called  a right section and for which there is an isomorphism $\bL\mr_{\dY}^*\cong \bR\mp_{\dY*}.$
 The composition $\mp_{\dY}\ldot\mr_{\dY}$ is isomorphic to the identity $\mId_{\dY}.$
By symmetry, the natural inclusion $\prfdg\dR\hookrightarrow\prfdg\,(\dR\underset{\mT}{\with}\dS)$ has a left adjoint functor. It defines a morphism $\ml_{\dX}:\dX\to\dZ$
which is called a left section and the composition $\mp_{\dX}\ldot\ml_{\dX}$ is isomorphic to the identity $\mId_{\dX},$ while the compositions
$\mp_{\dX}\ldot\mr_{\dY}$ and $\mp_{\dY}\ldot\ml_{\dX}$ are equal to $0$\!--morphisms.
Thus, the gluing $\dZ=\dX\underset{\mT}{\oright}\dY$  coming with the set of morphisms $(\ml_{\dX},\mp_{\dX}; \mp_{\dY},\mr_{\dY})$
which form a diagram of morphisms of derived noncommutative schemes
\begin{align}\label{localiz}
\xymatrix{
\dX \ar@<-0.3ex>[rr]_{\ml_{\dX}}&& \dZ\ar@<-0.7ex>[ll]_{\mp_{\dX}}\ar@<-0.3ex>[rr]_{\mp_{\dY}} && \dY\ar@<-0.7ex>[ll]_{\mr_{\dY}},
}
\end{align}
with the following properties:
\begin{itemize}
\item[(a)]
$\ml_{\dX}$ is a left adjoint section for the projection $\mp_{\dX},$
\item[(b)]
$\mr_{\dY}$ is a right adjoint section for the projection $\mr_{\dY},$
\item[(c)]
$\mp_{\dX}\ldot\ml_{\dX}\cong\mId_{\dX}$ and $\mp_{\dY}\ldot\mr_{\dY}\cong\mId_{\dY},$
\item[(d)]
$\mp_{\dX}\ldot\mr_{\dY}=0$ and, hence, by adjointness $\mp_{\dY}\ldot\ml_{\dX}=0.$ Moreover, the kernel of the functor
$\bL\mr^*_{\dY}$ is essentially the image of the functor $\bL\mp^*_{\dX}.$
\end{itemize}

In particular, we have a semi-orthogonal decomposition $\langle\prf\dR, \prf\dS\rangle$ for the triangulated category of perfect objects of the noncommutative scheme $\dZ,$
and by Proposition \ref{gluing_semi-orthogonal} it characterizes  the noncommutative scheme $\dZ$ as a gluing of $\dX$ and $\dY.$

Consider two morphisms of derived noncommutative schemes $\mf:\dX'\to\dX$ and $\mg: \dY'\to\dY.$
They induce a morphism $(\mg^{\op}\otimes\mf): \dY^{'\op}\otimes\dX'\to \dY^{\op}\otimes\dX.$
Let $\mT$  be a DG $\dY\hy\dX$\!--bimodule and $\mT'$ be a DG $\dY'\hy\dX'$\!--bimodule. Consider the respective gluings $\dX\underset{\mT}{\oright}\dY$ and $\dX'\underset{\mT'}{\oright}\dY'.$
By Proposition \ref{gluing_quasifunctors}, to any map $\phi: \mT\to \bR(\mg^{\op}\otimes\mf)_{*}\mT'$ in the derived category $\D(\dS^{\op}\otimes\dR)$ we can attach a morphism between the gluings
$(\mf\underset{\phi}{\oright}\mg): \dX'\underset{\mT'}{\oright}\dY'\to \dX\underset{\mT}{\oright}\dY.$
When $\dX'\cong\dX$ and $\dY'\cong\dY,$ for any map $\phi: \mT\to\mT'$ in $\D(\dS^{\op}\otimes\dR)$ there exists a morphism $\dX\underset{\mT'}{\oright}\dY\to \dX\underset{\mT}{\oright}\dY$
that is $(\id\underset{\phi}{\oright}\id).$ As a special case we obtain morphisms $\dX\oplus\dY\to\dX\underset{\mT}{\oright}\dY$ and $\dX\underset{\mT}{\oright}\dY\to \dX\oplus\dY$ the composition of which is
the identity morphism of $\dX\oplus\dY.$

A particular case of gluing of two noncommutative schemes is related to a morphism $\mf: \dX\to \dY.$ Any such morphism is a quasi-functor
$\mF:\prfdg\dS\to \prfdg\dR$ that is represented by a DG $\dY\hy\dX$\!--bimodule $\mT.$ The gluing $\dX\underset{\mT}{\oright}\dY$ with respect to this DG bimodule $\mT$
will be also denoted by $\dX\underset{\mf}{\oright}\dY.$
Let $\mf': \dX\to\dY$ be another morphism of the noncommutative schemes that is represented by a DG $\dY\hy\dX$\!--bimodule $\mT'.$ A map
$\phi: \mf\to\mf'$ in the triangulated category $\Mor(\dX, \dY)$ is a map between DG bimodules $\phi: \mT\to\mT'$ in the derived  category $\D(\dS^{\op}\otimes\dR).$
As above with any such map $\phi: \mf\to\mf'$ we can associate a morphism $\dX\underset{\mf'}{\oright}\dY\to \dX\underset{\mf}{\oright}\dY$ of the gluings.

Let us describe morphisms between a noncommutative scheme $\dV$ and a gluing $\dX\underset{\mf}{\oright}\dY$ of two other noncommutative schemes
(see, e.g., \cite[7.2]{KL} for details).

\begin{proposition}\label{morph_glu}
Let $\mf: \dX\to \dY$ be a morphism of derived noncommutative schemes and $\dZ=\dX\underset{\mf}{\oright}\dY$ be the gluing.
For any noncommutative scheme $\dV$ a morphism $\mv: \dZ\to\dV$ is a triple $(\mg, \mh; \phi),$ where $\mg:\dX\to\dV$ and $\mh:\dY\to\dV$
are morphisms of  noncommutative schemes and $\phi:\mh\ldot\mf\to\mg$ is a map between morphisms
in the category $\Mor(\dX, \dV).$

Similarly, a morphism $\mv': \dV'\to\dZ$ is a triple $(\mg', \mh'; \psi),$ where $\mg':\dV\to\dX$ and $\mh':\dV\to\dY$
are morphisms of noncommutative schemes and $\psi:\mf\ldot\mg'\to\mh'$ is a map between morphisms
in the category $\Mor(\dV', \dX).$
\end{proposition}

A composition $\mv\ldot\mv':\dV'\to\dV$ of any two morphisms $\mv': \dV'\to\dZ$ and $\mv: \dZ\to\dV$ is a morphism that can be obtained as
a cone of the map $\mh{\ldot}\mf\ldot\mg'\stackrel{\wt{\phi}+\wt{\psi}}{\lto}\mg\ldot\mg'\oplus\mh\ldot\mh'$ in the triangulated category $\Mor(\dV', \dV),$
where $\wt{\phi}:\mh\ldot\mf\ldot\mg'\to \mg\ldot\mg'$ and $\wt{\psi}: \mh\ldot\mf\ldot\mg'\to \mh\ldot\mh'$ are maps induced by $\phi: \mh\ldot\mf\to \mg$ and $\psi: \mf\ldot\mg'\to \mh',$
respectively.

Under the description above the projections $\mp_{\dX}$ and $\mp_{\dY}$ are equal to
$(\mId_{\dX}, \m0; 0)$ and $(\mf, \mId_{\dY}; \id),$ respectively, while the sections
$\ml_{\dX}$ and $\mr_{\dY}$ coincide with the morphisms $(\mId_{\dX}, \m0; 0)$ and $(\m0, \mId_{\dY}; 0),$ respectively.
Furthermore, in this situation there is a right adjoint morphism $\mr_{\dX}:\dX\to\dZ$ for which $\bL\mr_{\dX}^*\cong \bR\mp_{\dX*}.$
It is given by the triple $(\mId_{\dX}, \mf; \id).$ There is also another projection $\overline{\mp}_{\dY}$ which is determined by the triple
$(\m0, \mId_{\dY}; 0).$ For this morphism we have isomorphisms of the funtors $\bL\overline{\mp}_{\dY}^{*}\cong\bR\mr_{\dY*}\cong\mp_{\dY}^{!}.$
Finally, we obtain the following recollement for derived noncommutative schemes (see \cite{BBD}):
\begin{align}
\xymatrix{
\dX \ar@<1.5ex>[rr]^{\mr_{\dX}}\ar@<-1.5ex>[rr]_{\ml_{\dX}}&& \dZ\ar[ll]|{\, \mp_{\dX}\, }\ar@<1.5ex>[rr]^{\overline{\mp}_{\dY}}\ar@<-1.5ex>[rr]_{\mp_{\dY}} && \dY\ar[ll]|{\, \mr_{\dY}\, },
}
\end{align}
with the following properties:
\begin{itemize}
\item[(a)]
$\ml_{\dX}$ and $\mr_{\dX}$ are left and right adjoint sections for the projection $\mp_{\dX},$
\item[(b)]
$\mp_{\dY}$ and $\overline{\mp}_{\dY}$ are left and right adjoint projections for the section $\mr_{\dY},$
\item[(c)]
$\mp_{\dX}\ldot\ml_{\dX}\cong\mId_{\dX}\cong\mp_{\dX}\ldot\mr_{\dX}$ and $\mp_{\dY}\ldot\mr_{\dY}\cong\mId_{\dY}\cong\overline{\mp}_{\dY}\ldot\mr_{\dY},$
\item[(d)]
$\mp_{\dX}\ldot\mr_{\dY}=0$ and, hence, by adjointness $\mp_{\dY}\ldot\ml_{\dX}=0=\overline{\mp}_{\dY}\ldot\mr_{\dX}.$ Moreover, the kernel of the functor
$\bL\mr^*_{\dY}$ is essentially the image of the functor $\bL\mp^*_{\dX}.$
\end{itemize}

In addition, we have an isomorphism $\mp_{\dY}\ldot\mr_{\dX}\cong\mf,$ while $\overline{\mp}_{\dY}\ldot\ml_{\dX}\cong\mf[1].$

Let us consider the direct image functor $\bR\mf_*: \D(\dR) \to \D(\dS).$  It is isomorphic to $\bR\Hom_{\dR}(\mT, -),$ where $\mT$ is the corresponding quasi-functor.
Since $\mT$ is perfect as DG $\dR$\!--module,
this functor commutes with direct sums.
Therefore, there is a DG $\dR\hy\dS$\!--bimodule $\mU$ such that the functor $\bR\mf_*$ is represented as $(-)\stackrel{\bL}{\otimes}_\dR \mU.$
In fact, the DG $\dR\hy\dS$\!--bimodule $\mU$ is isomorphic to $\dHom_{\dR}(\mT, \dR).$
We can consider another gluing $\dZ'=\dY\underset{\mU}{\oright}\dX.$
However, the gluing $\dZ'=\dY\underset{\mU}{\oright}\dX$ is quasi-equivalebt to $\dZ=\dX\underset{\mT}{\oright}\dY$ and determines the same noncommutative scheme, which will be also denoted as $\dY\overset{\mf}{\oright}\dX.$
In fact, to obtain the decomposition $\dY\underset{\mU}{\oright}\dX$ for our noncommutative scheme $\dZ$ we have to consider the diagram (\ref{localiz}) with the set of morphisms
$(\mr_{\dY}, \overline{\mp}_{\dY}; \mp_{\dX},\mr_{\dX})$ instead of the set $(\ml_{\dX},\mp_{\dX}; \mp_{\dY},\mr_{\dY})$ for $\dZ=\dX\underset{\mT}{\oright}\dY.$
Thus, the noncommutative schemes $\dX\underset{\mf}{\oright}\dY$ and $\dY\overset{\mf}{\oright}\dX$ are isomorphic despite the constructions and the decompositions being different.

Note that if the morphism $\mf:\dX\to\dY$ is a pp-morphism and $\mg:\dY\to\dX$ is the right adjoint morphism, then, by construction, we have $\dY\overset{\mf}{\oright}\dX\cong \dY\underset{\mU}{\oright}\dX\cong \dY\underset{\mg}{\oright}\dX.$

The following example is coming from usual commutative geometry.

\begin{example}\label{ex2}
{\rm
Let $\operatorname{i}: Z\hookrightarrow Y$ be a closed immersion of a smooth proper scheme $Z$ into the smooth proper scheme $Y$
such that $Z$ is of codimension $2.$ Denote by $\wt{Y}$ the blowup of $Y$ along the closed subscheme $Z.$
A blow up formula from \cite{Blow} gives a semi-orthogonal decomposition of the category $\prf \wt{Y}$ in
the form $\langle \prf Y, \prf Z\rangle.$ Moreover, by Proposition \ref{gluing_semi-orthogonal} the DG category $\prfdg \wt{Y}$ is quasi-equivalent to the gluing
$(\prfdg Y) \underset{\mU}{\oright}(\prfdg Z),$ where $\mU$ is a DG bimodule of the form
\[
\mU(\mP, \mQ)\cong\dHom_{\prfdg X}(\mi^* \mQ,\;  \mP),\quad\text{with}\quad \mP\in\prfdg Z,\; \mQ\in\prfdg Y,
\]
and $\mi^*$ is the inverse image quasi-functor from $\prfdg Y$ to $\prfdg Z.$
The gluing $(\prfdg Y) \underset{\mU}{\oright}(\prfdg Z)$ is actually the gluing $(\prfdg Y) \underset{\operatorname{i}}{\oright}(\prfdg Z)$
along the morphism $\operatorname{i}: Z\hookrightarrow Y.$
}
\end{example}

\subsection{Properties of gluings}

Let us now discuss some properties of derived noncommutative schemes that are obtained by gluing.
First of all, it is easy to see when a gluing $\dX\underset{\mT}{\oright}\dY$ is proper and regular (see \cite[Prop. 3.20, 3.22]{O_glue}).

\begin{proposition}\label{prop_glue}
Let $\dX=\prfdg\dR$ and $\dY=\prfdg\dS$ be two derived noncommutative schemes over $\kk$ and let $\mT$ be a DG $\dY\hy\dX$\!--bimodule.
Then the following conditions are equivalent:
\begin{enumerate}
\item the  gluing $\dX\underset{\mT}{\oright}\dY$ is proper,
\item $\dX, \dY$ are proper, and  $\dim_{\kk}\bigoplus_i  H^i(\mT(\mQ, \mP))<\infty$
for all $\mP\in\prfdg\dR, \mQ\in\prfdg\dS.$
\end{enumerate}
\end{proposition}

Note that it is sufficient to check the property of $\mT$ mentioned above in the case $\mP=\dR$ and $\mQ=\dS,$ i.e.  $\dim_{\kk}\bigoplus_i  H^i(\mT)<\infty $ for $\mT$ as DG $\dS\hy\dR$\!--bimodule.

\begin{proposition}\label{reg_glue}
Let $\dX=\prfdg\dR$ and $\dY=\prfdg\dS$ be two derived noncommutative schemes over $\kk$ and let $\mT$ be a DG $\dY\hy\dX$\!--bimodule.
Then the following conditions are equivalent:
\begin{enumerate}
\item the  gluing $\dX\underset{\mT}{\oright}\dY$ is regular,
\item $\dX$ and $\dY$ are regular.
\end{enumerate}
\end{proposition}
Here we see that the property of being regular does not depend on the DG bimodule $\mT.$

On the other hand the property of being smooth depends on the DG bimodule $\mT.$
Since $\dR\underset{\mT}{\with}\dS$ and $\dX\underset{\mT}{\oright}\dY$ are Morita equivalent,
smoothness of $\dR\underset{\mT}{\with}\dS$ and $\dX\underset{\mT}{\oright}\dY$ holds simultaneously.
Further, we can compare smoothness of a gluing with smoothness of the summands. We obtain the following.

\begin{proposition}\cite[3.24]{LS}\label{smooth_glue}
Let $\dX=\prfdg\dR$ and $\dY=\prfdg\dS$ be two derived noncommutative schemes over $\kk$ and let $\mT$ be a DG $\dY\hy\dX$\!--bimodule.
Then the following conditions are equivalent:
\begin{enumerate}
\item the gluing $\dX\underset{\mT}{\oright}\dY$ is smooth,
\item $\dX$ and $\dY$ are smooth and $\mT$ is a perfect as DG $\dY\hy\dX$\!--bimodule.
\end{enumerate}
\end{proposition}

Let us consider the point $\pt=\prfdg\kk$ and a noncommutative scheme $\dX$ that is the gluing $\pt \underset{V}{\oright}\pt,$ where $V$ is a $\kk$\!--vector space.
If the vector space $V$ is infinite dimensional, then the noncommutative scheme $\dX$ is not smooth in spite of it being regular. Of course, in this case $\dX$ is also not proper.

\subsection{Geometric realizations of gluings}\label{real_gluings}

Let $X$ and $Y$ be two usual smooth irreducible projective schemes over a field $\kk.$ Let $\mE\in\prfdg(X\times_{\kk} Y)$ be a perfect complex
on the product $X\times_{\kk} Y.$
Note that in the  case of projective varieties any perfect complex  is globally (not only locally) quasi-isomorphic to
 a strictly perfect complex, i.e. a bounded complex of locally free sheaves of finite type (see, e.g. \cite[2.3.1]{TT}).

Let us consider the DG category that is obtained as the gluing $(\prfdg X) \underset{\mE}{\oright}(\prfdg Y).$ It is a derived noncommutative scheme
which will be denoted by $\dZ:=X\underset{\mE}{\oright} Y.$   Taking into account Theorem \ref{toen} and Propositions \ref{prop_glue}, \ref{smooth_glue},
we can deduce that the noncommutative scheme $\dZ$ is smooth and proper. The derived noncommutative scheme $\dZ$ is not commutative in general. However, it is natural to ask about existence of a geometric realization for
such derived noncommutative schemes.
The following theorem is proved in \cite{O_glue}.

\begin{theorem}\label{main}\cite[Th. 4.11]{O_glue}
Let $X$ and $Y$ be smooth irreducible projective schemes over a field $\kk$ and let $\mE$ be a perfect complex
on the product $X\times_{\kk} Y.$ Let $\dZ=X \underset{\mE}{\oright} Y$ be the derived noncommutative scheme that is the gluing of $\dX$ and $\dY$ via $\mE.$
Then there exist a smooth projective scheme $V$ and an ff-morphism $\mf: V\to\dZ,$ which give a pure geometric realization for the noncommutative scheme $\dZ.$
\end{theorem}

A proof of this theorem can be found in \cite{O_glue}. It is constructive and it
is useful to take in account that the category $\prf V$ from  Theorem \ref{main} has a semi-orthogonal decomposition of the form
$
\prf V=\langle \N_1,\dots \N_k\rangle
$
 such that each $\N_i$ is equivalent to one of the four categories: namely,
$\prf\kk,\; \prf X,\; \prf Y,$ and  $\prf (X\times_{\kk} Y).$

Now we can extend this result  to the case of derived noncommutative schemes.
Let $X_i,\; i=1,\dots,n$ be smooth and projective schemes and let $\dN_i\subset\prfdg X_i,\; i=1,\dots, n$ be full pretriangulated DG subcategories.
Assume that the homotopy triangulated  categories $\N_i=\Ho(\dN_i)$ are admissible in $\prf X_i.$
By Propositions \ref{prop_glue} and \ref{smooth_glue} these conditions imply that the derived noncommutative schemes $\dN_i$
are proper and smooth. Moreover, they are coming with pure geometric realizations.

\begin{theorem}\label{main2}\cite[Th. 4.15]{O_glue}
Let the DG categories $\dN_i,\; i=1,\dots, n$ and the smooth projective schemes $X_i,\; i=1,\dots, n$ be as above.
Let $\dX=\prfdg\dR$ be a proper derived noncommutative scheme with full embeddings
 of the DG categories $\dN_i\subset \prfdg\dR$ such that $\prf\dR$ has a semi-orthogonal decomposition of the form
$
\langle\N_1, \N_2,\dots, \N_n\rangle,
$
where $\N_i=\Ho(\dN_i).$ Then there exist a smooth projective scheme $X$ and  an ff-morphism $\mf: X\to \dX$ which give
a pure geometric realization for the noncommutative scheme $\dX.$
\end{theorem}

Note that in this case the derived noncommutative scheme $\dX$ is also smooth. Indeed, it is  a gluing of smooth proper noncommutative schemes $\dN_i$
with respect to  DG bimodules that are DG functors from $\dN_j\otimes\dN_i^{\op}$ to $\prfdg\kk.$ By Proposition \ref{dg_saturated} all such DG bimodules are  perfect because
$\dN_i$ are smooth and proper.
Theorem \ref{main2} implies that the world of all smooth proper geometric noncommutative schemes is closed under gluing
via perfect bimodules.

These theorems have useful applications. Using results of \cite{KL} we obtain that
 for any usual proper scheme
$Y$ over a field of characteristic $0$ there is a full embedding of $\prf Y$ into $\prf V,$ where $V$ is smooth and projective.

\begin{corollary}\label{emb}\cite[Cor 4.16]{O_glue}
 Let $Y$ be a proper scheme over a field of characteristic $0.$ Then there are
a smooth projective scheme $X$ and a quasi-functor $\mF: \prfdg Y\to \prfdg X$ such that
the induced functor $F:\prf Y\to \prf X$ is fully faithful, i.e. $Y$ has a plain geometric realization that is a smooth desingularization $\mf: X\to Y.$
\end{corollary}

When a proper derived noncommutative scheme $\dX=\prfdg\dR$ has a full exceptional collection, there is another and  more useful procedure of constructing a smooth projective
geometric realization.
Any such derived noncommutative scheme $\dX$ is smooth and could be obtained by a procedure of sequential gluing of copies the point $\pt.$  In this case one can find a usual smooth projective scheme $X$ and an exceptional collection of line bundles
$\sigma=(\L_1,\dots, \L_n)$ on $X$ such that the DG subcategory $\dN\subset\prfdg X$
generated by $\sigma$ is quasi-equivalent to $\prfdg\dR.$
Moreover, by construction, the scheme $X$ is rational and has a full exceptional collection.

\begin{theorem}\label{exc_col}\cite[Th. 5.8]{O_glue}
Let $\dX=\prfdg\dR$ be a proper derived noncommutative scheme over $\kk$ such that the homotopy category
$\prf\dR$ has a full exceptional collection
$
\prf\dR=\langle E_1,\dots, E_n\rangle.
$
Then there are a smooth projective scheme $X$ and an exceptional collection of line bundles
$\sigma=(\L_1,\dots, \L_n)$ on $X$ such that the DG subcategory of $\prfdg X$
generated by $\sigma$ is quasi-equivalent to $\prfdg\dR.$
Moreover, $X$ can be chosen in such way that it is a tower of projective bundles and has a full exceptional collection.
\end{theorem}

The scheme $X$ has a full exceptional collection as a tower of projective bundles
(see \cite{Blow}). Furthermore, it follows from the construction that a full exceptional collection on $X$ can be chosen in such a way that
it contains the collection $\sigma=(\L_1,\dots, \L_n)$ as a subcollection.

In the proof of this theorem one constructed  a quasi-functor from the DG category $\prfdg\dR$ to the DG category
$\prfdg X$ that sends the exceptional objects $E_i$ to shifts of the line bundles $\L_i[r_i]$ for some
integers $r_i.$ In other words, we have an ff-morphism $\mf: X\to\dX$ that gives a pure geometric realization for
the noncommutative scheme $\dX$ and $\bL\mf^* E_i\cong \L_i[r_i].$
 Of course, we can not expect in general that $E_i$ go to line bundles without shifts.
On the other hand, in the case of  strong exceptional collections it is natural to seek  geometric realizations as
 collections of vector bundles (without shifts) on  smooth projective varieties. It can be shown that in general
 we can not realize a strong exceptional collection as a collection of unshifted line bundles,
but trying to present it in terms of vector bundles seems quite reasonable.

\begin{theorem}\label{exc_colllection_bundles}\cite[Cor. 2.7]{O_q}
Let $\dX=\prfdg\dR$ be a proper derived noncommutative scheme such that the  category
$\prf\dR$ has a full strong exceptional collection
$
\prf\dR=\langle E_1,\dots, E_n\rangle.
$
Then there exist a smooth projective scheme $X$ and an ff-morphism  $\mf: X\to \dX$
such that the functor $\bL\mf^*$ sends the exceptional objects $E_i$ to vector bundles $\E_i$ on $X.$
\end{theorem}

A special class of derived noncommutative schemes is related to finite dimensional algebras.
Let $\La$ be a finite dimensional algebra over a base field $\kk.$
Consider the derived noncommutative scheme $\dV=\prfdg\La.$ This noncommutative scheme  is proper for any such $\La.$
It is regular if and only if
the algebra $\La$ has  finite global dimension.
Denote by $\rd$ the (Jacobson) radical of $\La.$ We know that $\rd^n=0$ for some $n.$
Let $S$ be the quotient algebra $\La/\rd.$ It is semisimple and has only a finite number of simple non-isomorphic modules.

Recall that a semisimple algebra $S$ over a field $\kk$ is called separable over $\kk$ if it is  projective as
an $S\hy S$\!--bimodule. It is well-known that a semisimple algebra $S$ is separable if
it is a direct sum of simple algebras whose centers are separable extensions of the field $\kk.$
It also means that the noncommutative scheme $\prfdg S$ is smooth. Moreover, the noncommutative scheme $\dV=\prfdg\La$
is smooth over $\kk$ if $\La$ has  finite global dimension and $S=\La/\rd$ is separable (see, e.g. \cite{Ro}).

\begin{theorem}\label{algebra}\cite[Th. 5.3]{O_glue}
Let $\La$ be a finite dimensional algebra over $\kk.$ Assume that the semisimple algebra $S=\Lambda/\rd$ is $\kk$\!--separable.
Then there are a smooth projective scheme
$X$ and a perfect complex $\mE\in\prf X$ such that $\End(\mE)\cong\Lambda$ and $\Hom(\mE, \mE[l])=0$
for all $l\ne 0.$
\end{theorem}

\begin{corollary}\label{algebra_inclusion}\cite[Th. 5.4]{O_glue}
Let $\dV=\prfdg\La$ be a derived noncommutative scheme, where $\Lambda$ is a finite dimensional algebra over $\kk$ for which $\Lambda/\rd$ is $\kk$\!--separable.
Then there are a smooth projective scheme $X$ and an ff-morphism $\mf: X\to\dV$ which give a plain geometric realization for $\dV.$
Moreover, if $\La$ has   finite global dimension, then this realization is pure.
\end{corollary}
Note that over a perfect field all semisimple algebras are separable.
Thus, if $\kk$ is perfect, then these results can be applied  to any finite-dimensional $\kk$\!--algebra.

Consider a smooth and proper noncommutative scheme $\dX=\prfdg\dR$ such that the  category
$\prf\dR$ has a full strong exceptional collection
$
\sigma=\langle E_1,\dots, E_n\rangle.
$
The object $E=\bigoplus\limits_{i=1}^n E_i$ is a generator of $\prf\dR,$ and the DG category $\prfdg\dR$ is quasi-equivalent to the DG category $\prfdg\La,$ where $\La=\End(\mathop{\bigoplus}\limits_{i=1}^n E_i)$ is the algebra of endomorphisms of the collection $\sigma.$
It is evident that the algebra $\La$ is a quiver algebra
on $n$ directed vertices.
Recall that $\La$ is  called a {\sf quiver algebra on $n$ directed vertices}, if  $\La\cong \kk Q/I,$  where $Q$ is a quiver
for which $Q_0=\{1,\dots , n\}$ is the ordered set of $n$ elements  and  for any arrow $a\in Q_1$ the source $s(a)\in Q_0$ is less than target $t(a)\in Q_0,$
while $I$ is a two-sided ideal of the path algebra $\kk Q$ generated by a subspace of $\kk Q$ spanned by linear
combinations of paths of length at
least $2$ having a common source and a common target (see, e.g., \cite{O_q}).

On the other hand, any quiver algebra $\La$ on $n$ directed vertices has finite global dimension and,
moreover, the category $\prf\La$ has a strong full exceptional collection
consisting of the indecomposable  projective modules $P_i$ for $i=1,\dots, n.$ The algebra $\La$ is exactly
the algebra of endomorphisms of this full strong exceptional collection.
Thus, Theorem \ref{exc_colllection_bundles} implies that for any quiver algebra  $\La$  on $n$ directed vertices there exist a smooth projective scheme
$X$ and a vector bundle $\E$ on $X$ such that $\End_X(\E)=\La$ and $\Ext^p_X(\E, \E)=0$ for all $p\ne 0.$
Moreover, they can be chosen so that the rank of $\E$ is equal to the dimension of $\La$ (see \cite[Cor. 2.8]{O_q}).

\subsection{Quasi-phantoms and phantoms}\label{phantoms}
Let $\T$ be a triangulated category and let ${\T}=\left\langle{\N}_1, \dots, {\N}_n\right\rangle$ be a semi-orthogonal decomposition.
Any such decomposition induces the following decomposition for  the Grothendieck group
\[
 K_0(\T)\cong K_0(\N_1)\oplus K_0(\N_2)\oplus\cdots\oplus K_0(\N_n).
\]
In the particular case of a full exceptional collection we obtain an isomorphism of the Grothendieck group
$K_0(\T)$ with a free abelian group $\ZZ^n.$

For any small DG category  $\dA$ we can define K-theory spectrum $K(\dA)$  by applying
Waldhausen's construction to a certain category with cofibrations and weak equivalences that can be obtained from the DG category $\prfdg\dA$ (see \cite{DS, Sch, Ke2, Ta}).
More precisely, the objects of this category are perfect DG modules,  the cofibrations are the morphisms of DG $\dA$\!--modules of
degree zero that admit retractions as morphisms of graded modules, and the weak equivalences are the quasi-isomorphisms.
This construction is invariant under quasi-equivalences between $\prfdg\dA$ and $\prfdg\dB.$
Thus, to any derived noncommutative scheme $\dX=\prfdg\dR$ we can attach a K-theory spectrum  $K(\dX):=K(\dR)=K(\prfdg\dR).$
K-theory gives us an additive invariant for derived noncommutative schemes in the sense that for any gluing $\dZ=\dX\underset{\mT}{\oright}\dY$
there is an isomorphism $K(\dZ)\cong K(\dX)\oplus K(\dY).$

Other natural additive invariants are given by Hochschild and cyclic homology.
Hochschild homology $\mathrm{HH}_{*}(\dX)$ can be defined as
\[
\mathrm{HH}_{*}(\dX)\cong H^{-*}(\dR \mathop{\overset{\bL}{\otimes}}_{\dR^{\op}\otimes\dR}\dR),
\]
and for any gluing $\dZ=\dX\underset{\mT}{\oright}\dY$
there is an isomorphism $\mathrm{HH}_{*}(\dZ)\cong \mathrm{HH}_{*}(\dX)\oplus \mathrm{HH}_{*}(\dY).$
Recall that for a smooth projective scheme $X$ over a field of characteristic $0$ there is an isomorphism $\mathrm{HH}_i(X)=\bigoplus_p H^{p+i}(X, \Omega^p_X),$ which allows to
describe the Hochschild homology in terms of usual cohomology of the scheme $X$ (see \cite{Sw}).

\begin{definition} A smooth and proper derived noncommutative scheme
$\dX$ will be called  a {\sf quasi-phantom} if
$\mathrm{HH}_{*}(\dX)=0$ and $K_0(\dX)$ is a finite abelian group.
It will be  called a {\sf phantom} if, in addition, $K_0(\dX)=0.$
\end{definition}
This definition is a result of the successive study and appearance of such objects, but a more natural and more important definition is the following definition of a universal phantom.

\begin{definition}\label{def:uphantom}
We say that a phantom $\dX$ is a {\sf universal phantom} if
$\dX\otimes_{\kk} \dY$ is also phantom
for any smooth and proper noncommutative scheme $\dY.$
\end{definition}
It can be shown that it is sufficient to verify this property for $\dY=\dX.$ Moreover, it is also known
that any universal phantom $\dX$ has a trivial K-motive and, hence, its K-theory $K(\dX)$ vanishes as well (see \cite{GO}).

Different examples of geometric quasi-phantoms were constructed as semi-orthogonal complements
to exceptional collections of maximal length on some smooth projective surfaces of general type with $q=p_g=0$ for which Bloch's conjecture holds, i.e. the Chow group $CH^2(S)\cong \ZZ.$
In more detail, let $S$ be such a surface. In this case the Grothendieck group
$K_0(S)$  is  isomorphic to $\ZZ\oplus \Pic(S)\oplus \ZZ\cong \ZZ^{r+2}\oplus \Pic(S)_{\tors},$ where
$r$ is the rank of the Picard lattice $\Pic(S)/\Pic(S)_{\tors}.$
Since the Picard group $\Pic(S)$ of this surface is isomorphic to $H^2(S(\CC), \ZZ),$ we obtain
that $\Pic(S)_{\tors}$ is finite and there are the following relations
$
r+2=b_2+2=e,
$
where $b_2$ is the second Betti number and $e$ is the topological Euler characteristic of $S.$
Assume that the triangulated category
 $\prf S$ has an exceptional collection $(E_1,\ldots, E_{e})$ of the maximal possible length $e.$
In this case there is a semi-orthogonal
decomposition of the form
\[
\prf S=\langle E_1,\ldots, E_{e}, \N \rangle,
\]
where $\N$ is the left orthogonal to the subcategory $\T=\langle E_1,\ldots, E_{e} \rangle$ generated by the exceptional collection.
We have $K_0(\T)\cong\ZZ^{e}$ and, hence, $K_0(\N)\cong\Pic(S)_{\tors}.$

Consider now the DG category $\prfdg S$ and its DG subcategory $\dN\subset \prfdg S$ that has the same objects as $\N\subset\prf S.$
The DG category $\dN$ is a derived noncommutative scheme, which is smooth and proper because the subcategory $\N$ is admissible.
We already mentioned that $K_0(\dN)\cong\Pic(S)_{\tors}.$ Moreover, it is evident that the Hochschild homology
$\mathrm{HH}_{*}(\dN)$ is trivial.
Thus, the noncommutative scheme $\dN$ is a quasi-phantom  coming with a pure geometric realization $\dN\subset\prfdg S.$

At this moment there are a lot of different examples of quasi-phantoms constructed as described above.
First example was constructed in \cite{BGS} for the classical Godeaux surface $S$ that is the
$\ZZ/5\ZZ$-quotient of the Fermat quintic in~$\PP^3$ . In this case $e=11,$
and the Grothendieck group $K_0(\dN)$ is isomorphic to the cyclic group $\ZZ/5\ZZ.$

The next examples were Burniat surfaces with $e=6,$  exceptional collections of maximal length for which  were constructed
in \cite{AO}. In this case we have a 4-dimensional family of such surfaces
and we obtain a 4 dimensional family of quasi-phantoms $\dN$ with $K_0(\dN)=(\ZZ/2\ZZ)^6.$
(It was proved in the paper  \cite{Ku} that the second Hochschild cohomology
of the quasi-phantoms $\dN$ coincides with the second Hochschild cohomology of the related Burniat surfaces $S.$)

Suppose we have two different quasi-phantoms $\dN$ and $\dN'.$ It is natural to consider their tensor product $\dN\otimes_{\kk}\dN'.$
If the orders of the Grothendieck groups $K_0(\dN)$ and $K_0(\dN')$ are coprime, we can hope that the Grothendieck
group of $\dN\otimes_{\kk}\dN'$ will be trivial. In the case of surfaces it can be proved.

\begin{theorem}\label{main_phant}\cite{GO}
Let $S$ and $S'$ be  smooth projective surfaces over $\CC$ with $q=p_g=0$ for which Bloch's conjecture for 0-cycles holds.
Assume that the categories $\prf S$ and $\prf S'$ have exceptional collections of maximal
lengths $e(S)$ and $e(S'),$ respectively. Let $\dN\subset\prfdg S$ and $\dN'\subset\prfdg S'$ be the left orthogonals
to these exceptional collections. If the orders of $\Pic(S)_{\tors}$ and $\Pic(S')_{\tors}$ are coprime, then the noncommutative scheme
$\dN\otimes_{\kk}\dN'\subset \prfdg(S\times_{\kk} S')$
is a universal phantom.
\end{theorem}

This theorem  also tells us that the noncommutative scheme $\dN\otimes_{\kk}\dN'$ has a trivial K-motive, i.e.
it is in the kernel of the natural map from the world of smooth and proper derived noncommutative schemes to the world of K-motives (they are
called noncommutative motives now) and, in particular, it has trivial K-theory, i.e. $K_i(\dN\otimes\dN')=0$ for all $i$ (see \cite{GO}).
It is known that K-theory is a universal additive invariant \cite{Ta_new, Ta} and, hence, all additive invariants of universal phantoms vanish.

\begin{corollary}\cite{GO} Let $S$ be a Burniat surface with $e=6$ and let $S'$ be the classical Godeaux surface over $\CC.$
Let $\dN\subset\prfdg S$ and $\dN'\subset\prfdg S'$ be quasi-phantoms that are the left orthogonals
to exceptional collections of maximal lengths.
Then the derived noncommutative scheme
${\dN\otimes_{\kk}\dN'} \subset \prfdg(S\times_{\kk} S')$ is a universal phantom, and $K_i(\dN\otimes_{\kk}\dN')=0$ for all $i\in \ZZ.$
\end{corollary}

Another type of a geometric phantom was constructed in \cite{BGKS} as  a semi-orthogonal complement to an exceptional collection of maximal length on the determinantal  Barlow surface.
Since the Barlow surface is simply connected, it does not have torsion in Picard group. In this case any quasi-phantom constructed as it was described above
is actually a phantom because the Grothendieck group is trivial. The results of the paper \cite{GO} applied to a phantom coming from a Barlow surface give us that this phantom is universal.

\subsection{Krull--Schmidt partners}\label{Kru_Sch}

Let $\dX=\prfdg\dR$ and $\dY=\prfdg\dS$ be two derived noncommutative schemes and let $\mf:\dX\to\dY$ be a morphism that is represented by a DG $\dY\hy\dX$\!--bimodule $\mT$ as a quasi-functor
$\mF:\prfdg\dS\to \prfdg\dR.$
Consider the gluing $\dZ=\dX\underset{\mf}{\oright}\dY,$ which is by definition the gluing
$
\dX\underset{\mT}{\oright}\dY=\prfdg(\dR\underset{\mT}{\with}\dS)
$
of $\dX$ and $\dY$ via $\mT.$

Any morphism $\mc: \dX\to \dV$ induces a morphism $\underline{\mc}:\dZ\to\dV$ that is the composition of $\mc$ and the natural projection $\mp_{\dX}:\dZ\to\dX.$
By Proposition \ref{morph_glu}, the morphism $\underline{\mc}$ is given by the triple $(\mc, \m0; 0).$ If the morphism $\mc$
is represented by a DG
$\dV\hy\dX$\!--bimodule $\mP,$ then the morphism $\underline{\mc}$ is related to  a DG $\dV\hy\dZ$\!--bimodule
$\underline{\mP}=(\mP, 0),$ which coincides with $\mP$ on the subcategory $\prfdg\dR\subset\prfdg(\dR\underset{\mT}{\with}\dS)$ and is equal to $0$
on the subcategory $\prfdg\dS\subset\prfdg(\dR\underset{\mT}{\with}\dS).$

Let $\mg: \dX\to\dY$ be another morphism that is represented by a DG $\dY\hy\dX$\!--bimodule $\mU.$
Suppose there is a map $\phi: \mf\to \mg$ between the morphisms.
By Proposition \ref{morph_glu}, the map $\phi: \mf\to \mg$ induces a morphism $\wt{\mg}_{\phi}:\dZ\to\dY$ that is given by the triple $(\mg, \mId_{\dY}; \phi).$ The morphism $\wt{\mg}_{\phi}$ is represented by
a DG $\dY\hy\dZ$\!--bimodule
$\wt{\mU}_{\phi}:=(\mU, \dS)$ that coincides with
$\mU$ on the subcategory $\prfdg\dR\subset\prfdg(\dR\underset{\mT}{\with}\dS),$ and
$\wt{\mU}_{\phi}(-, \bullet)=\dHom_{\dY}(\bullet , -)$ on the subcategory $\prfdg\dS\subset\prfdg(\dR\underset{\mT}{\with}\dS).$  The map $\phi: \mT\to \mU$ allows to define a natural structure of
a $\dZ$\!--module on $\wt{\mU}_{\phi}=(\mU, \dS).$
There is an isomorphism $\wt{\mf}_{\id}\cong\mp_{\dY}.$ Moreover, it is easy to see that, by Proposition \ref{morph_glu}, the composition of $\wt{\mg}_{\phi}:\dZ\to\dY$ with the right section $\mr_{\dX}:\dX\to\dZ$
is exactly the morphism $\mg:\dX\to\dY.$

Let us denote by $\mc$ the cone of the map $\phi: \mf\to \mg.$ The map $\phi$ induces a map $\wt{\phi}: \mp_{\dY}\to \wt{\mg}_{\phi}.$ The cone of this map is isomorphic to
the morphism
$\underline{\mc}=\mc\mp_{\dX}.$
Thus, we have two exact triangles of morphisms
\begin{align*}
(\mT\mr1)\quad  \mc[-1]\lto \mf\stackrel{\phi}{\lto} \mg\lto \mc \qquad\text{and}\qquad & (\mT\mr2)\quad
\underline{\mc}[-1]\lto \mp_{\dY}\stackrel{\wt{\phi}}{\lto} \wt{\mg}_{\phi}\lto \underline{\mc}
\end{align*}
in the triangulated categories $\Mor(\dX, \dY)$ and $\Mor(\dZ, \dY).$ Moreover, the triangle $(\mT\mr1)$ can be obtained from $\mT\mr2$ by applying the section $\mr_{\dX},$
i.e. we have $(\mT\mr1)=(\mT\mr2)\ldot\mr_{\dX}.$

Let us describe a case when the constructed morphism $\wt{\mg}_{\phi}:\dZ\to\dY$ is an ff-morphism, i.e. it induces a fully faithful embedding
$\bL\wt{\mg}_{\phi}^{*}:\prf\dS\to \prf(\dR\underset{\mT}{\with}\dS).$ By construction, the projection $\wt{\mf}_{\id}\cong\mp_{\dY}$ is such a morphism.

\begin{theorem}\label{main1_Krull} Let $\dX=\prfdg\dR$ and $\dY=\prfdg\dS$ be noncommutative schemes and let $\phi: \mf\to\mg$ be a map between morphisms
from $\dX$ to $\dY.$ Let $\mc=\mCone(\phi)$ be the cone of $\phi.$ Let the morphisms $\mg, \mc$ are represented by a DG $\dY\hy\dX$\!--bimodules $\mU$ and $\mP,$ respectively.
Then the following properties are equivalent:
\begin{itemize}
\item[(1)]  $\wt{\mg}_{\phi}:\dZ=\dX\underset{\mf}{\oright}\dY\to\dY$ is an ff-morphism, i.e. $\bL\wt{\mg}^{*}_{\phi}$ is fully faithful;
\item[(2)]
$
\Hom_{\dZ}(\bL\underline{\mc}^*\dS,\; \bL\wt{\mg}^{*}_{\phi}\dS[m])\cong\Hom_{\dZ}(\underline{\mP},\; \wt{\mU}_{\phi}[m])=0
$
for all $m\in\ZZ;$
\item[(3)]
$
\Hom_{\dX}(\bL\mc^*\dS,\; \bL\mg^{*}\dS[m])\cong\Hom_{\dX}(\mP,\; \mU[m])=0
$
for all $m\in\ZZ.$
\end{itemize}

\end{theorem}

\begin{proof}
The triangle $(\mT\mr2)$ of morphisms from $\dZ$ to $\dY$ induces an exact triangle
\begin{align}\label{triangle}
\bL\underline{\mc}^*\dS[-1]\lto \bL\mp_{\dY}^*\dS\stackrel{\wt{\phi}_{\dS}}{\lto} \bL\wt{\mg}^{*}_{\phi}\dS\lto \bL\underline{\mc}^*\dS
\end{align}
in the triangulated category $\prf(\dR\underset{\mT}{\with}\dS).$
Since $\wt{\mg}_{\phi}\ldot\mr_{\dX}=\mg$ and $\underline{\mc}=c\ldot\mp_{\dX}$ and taking into account that $\bL\mr_{\dX}^*\cong\bR\mp_{\dX*},$
we obtain the following sequence of isomorphisms
\begin{multline*}
\Hom_{\dZ}(\bL\underline{\mc}^*\dS,\; \bL\wt{\mg}^{*}_{\phi}\dS[m])
\cong
\Hom_{\dZ}(\bL\mp^*_{\dX}\bL\mc^*\dS,\; \bL\wt{\mg}^{*}_{\phi}\dS[m])
\cong\\
\cong\Hom_{\dX}(\bL\mc^*\dS,\; \bR\mp_{\dX*}\bL\wt{\mg}^{*}_{\phi}\dS[m])
\cong
\Hom_{\dX}(\bL\mc^*\dS,\; \bL\mg^{*}\dS[m]).
\end{multline*}
Thus,   condition (2) is equivalent to condition (3).

Consider now the map $\wt{\phi}:\mp_{\dY}\to \wt{\mg}_{\phi}$ between morphism from $\dZ$ to $\dY.$ It induces a natural transformation
$\bL^*\mp_{\dY}\to\bL^*\wt{\mg}_{\phi}$ beteen inverse image functors.
In particular, for any pair of objects $\mM,\mN\in\prf\dS$ and a morphism $u:\mM\to\mN$ there is a commutative diagram
\[
\begin{CD}
\bL^*\mp_{\dY}(\mM)@>\bL^*\mp_{\dY}(u)>>\bL^*\mp_{\dY}(\mN)\\
@V\wt{\phi}_{\mM}VV @VV\wt{\phi}_{\mN}V\\
\bL^*\wt{\mg}_{\phi}(\mM)@>\bL^*\wt{\mg}_{\phi}(u)>>\bL^*\wt{\mg}_{\phi}(\mN)
\end{CD}
\]
Putting $\mM=\dS$ and $\mN=\dS[m]$ we obtain  the following commutative diagram
\begin{align}\label{diagr}
\begin{CD}
\Hom_{\dY}(\dS,\; \dS[m])@>{{\bL\wt{\mg}^{*}_{\phi}}}>>
\Hom_{\dZ}(\bL\wt{\mg}^{*}_{\phi}\dS,\; \bL\wt{\mg}^{*}_{\phi}\dS[m])\\
@V{\bL\mp^{*}_{\dY}} V{\wr}V  @VV\mh^{\bullet}(\wt{\phi}_{\dS})V\\
\Hom_{\dZ}(\bL\mp_{\dY}^*\dS,\; \bL\mp_{\dY}^*\dS[m]) @>\sim>\mh_{\bullet}(\wt{\phi}_{\dS[m]})>
\Hom_{\dZ}(\bL\mp_{\dY}^*\dS,\; \bL\wt{\mg}^{*}_{\phi}\dS[m])
\end{CD}
\end{align}
The left vertical arrow is an isomorphism because the functor $\bL\mp^{*}_{\dY}$ is fully faithful.
The bottom horizontal arrow is also an isomorphism for all $m\in\ZZ.$ Indeed, let us apply the functor $\Hom_{\dZ}(\bL\mp_{\dY}^*\dS, -)$ to the exact triangle (\ref{triangle}).
Taking into account semi-orthogonal decomposition for perfect complexes on $\dZ,$ we obtain $\Hom_{\dZ}(\bL\mp_{\dY}^*\dS, \bL\underline{\mc}^*\dS)=0$ because $\bL\underline{\mc}^*\dS\cong \bL\mp_{\dX}^*\ldot\bL\mc^*\dS.$
This implies that the bottom arrow is an isomorphism.

(1)$\Leftrightarrow$ (2). Now, the top horizontal arrow in diagram (\ref{diagr}) is an isomorphism for all $m\in\ZZ$ if and only if the right vertical arrow is an isomorphism for all $m\in\ZZ.$
But it is equivalent to the property $\Hom_{\dZ}(\bL\underline{\mc}^*\dS, \bL\wt{\mg}^{*}_{\phi}\dS[m])=0$ for all $m\in\ZZ$ because the right vertical arrow
$\mh^{\bullet}(\wt{\phi}_{\dS})$ is a part of the long exact sequence obtained by applying the functor $\Hom_{\dZ}(-, \bL\wt{\mg}^{*}_{\phi}\dS)$ to exact triangle (\ref{triangle}). Thus, if the functor $\bL\wt{\mg}^{*}_{\phi}$ is fully faithful, then the top horizontal arrow is an isomorphism and this implies condition (2).

Conversely, if property (2) holds for all $m\in\ZZ,$ then the right vertical arrow in diagram (\ref{diagr}) is an isomorphism for all $m\in\ZZ.$
Hence, the top horizontal arrow is an isomorphism for all $m\in\ZZ.$ The object $\dS$ is a classical generator for the category $\prf\dS.$ Thus, by
Proposition \ref{Keller2}, the functor $\bL\wt{\mg}^{*}_{\phi}$ is fully faithful and $\wt{\mg}_{\phi}:\dZ=\dX\underset{\mf}{\oright}\dY\to\dY$ is an ff-morphism.
\end{proof}

Let $\dX=\prfdg\dR$ and $\dY=\prfdg\dS$ be derived noncommutative schemes and let $\phi: \mf\to\mg$ be a map between morphisms
from $\dX$ to $\dY.$ Suppose, as in Theorem \ref{main1_Krull}, the morphism
$\wt{\mg}_{\phi}:\dZ=\dX\underset{\mf}{\oright}\dY\to\dY$ is an ff-morphism. Assume also that the morphism
$\wt{\mg}_{\phi}$ is a pp-morphism and, hence, has a right section $\mr'_{\dY}: \dY\to\dZ.$
This implies that the embedding functor $\bL\wt{\mg}^{*}_{\phi}$ realizes the category $\prf\dS$ as a right admissible subcategory in
$\prf(\dR\underset{\mf}{\with}\dS).$ Thus, we obtain another semi-orthogonal decomposition of the form  $\prf(\dR\underset{\mf}{\with}\dS)=\langle \prf\dS^{\perp},\; \prf\dS\rangle.$
The natural enhancement of the subcategory $\prf\dS^{\perp}$ gives a derived noncommutative scheme $\dX'=\prfdg\dR'.$ By Proposition \ref{gluing_semi-orthogonal},
 the decomposition above says that the noncommutative scheme $\dZ$ can be also represented as a gluing $\dX'\underset{\mT'}{\oright}\dY,$
where $\mT'$ is a DG $\dY\hy\dX'$--bimodule. The morphism $\wt{\mg}_{\phi}$ is a new projection $\mp'_{\dY}.$
The composition $\mp'_{\dY}\ldot\mr'_{\dY}$ is the identity by adjointness, but the composition $\mp'_{\dY}\ldot\mr_{\dY}$ is also isomorphic to the identity
by construction. This implies that the adjoint composition $\overline{\mp}_{\dY}\ldot\mr'_{\dY}$ is isomorphic to the identity too.

In the case when $\dX$ and $\dY$ are regular and proper, all projections have right and left adjoint sections by Proposition \ref{admissible}.
The noncommutative schemes $\dZ$ and $\dX'$ are also regular and proper by Propositions \ref{prop_glue}, \ref{reg_glue}.

\begin{definition} Let $\dX=\prfdg\dR$ and $\dX'=\prfdg\dR'$ be two smooth and proper derived noncommutative schemes.
Suppose there are another smooth and proper noncommutative scheme $\dY=\prfdg\dS,$ morphisms $\mf:\dX\to\dY,$ $\mf':\dX'\to\dY,$ and isomorphisms $\mv, \mw$
between the gluings
\[
\xymatrix{
\dZ=\dX\underset{\mf}{\oright}\dY\ar@<0.8ex>[r]^{\mv}& \dZ'=\dX'\underset{\mf'}{\oright}\dY\ar@<0.5ex>[l]^{\mw}
}
\]
such that the morphisms $\mp'_{\dY}\ldot\mv\ldot\mr_{\dY}$ and $\mp_{\dY}\ldot\mw\ldot\mr'_{\dY}$ are isomorphisms of $\dY.$
In this case we call $\dX$ and $\dX'$ {\sf Krull--Schmidt partners.}
\end{definition}

Since the composition $\mp'_{\dY}\ldot\mv\ldot\mr_{\dY}$ is an isomorphism,  K-theories of $\dX$ and $\dX'$ are isomorphic. Moreover, their K-motives  are isomorphic too.

Let $X$ be a smooth projective scheme and let $\dR$ be a DG algebra such that $\prfdg X$ is quasi-equivalent to $\prfdg\dR.$ Let $\mP_s\in \prfdg X,\; s=1,2$ be two perfect complexes
such that their supports $\supp\mP_1, \supp\mP_2\subset X$ do not meet.
Put $\mT=\mP_1\oplus\mP_2$ and consider the gluing $\dZ=X\underset{\mT}{\oright}\pt.$ Let us take $\mU=\mP_2$ and $\mP=\mP_1.$
Since  condition (3) of  Theorem \ref{main1_Krull} holds for $\mU=\mP_2$ and $\mP=\mP_1,$ the morphism $\dZ\to\pt,$ which is given by the object $\wt{\mU},$ is an ff-morphism.
Therefore, we obtain another semi-orthogonal decomposition
$\langle \prf\dR', \prf\kk\rangle$ for the category $\prf(\dR\underset{\mT}{\with}\kk).$
Thus, we get a Krull--Schmidt partner $\dX':=\prfdg\dR'$ for the usual commutative scheme $X.$
In general this Krull--Schmidt partner is not isomorphic to $X.$

For example, let $X$ be a smooth projective curve of genus $g,$ and let $\mP_1, \mP_2$ be torsion coherent sheaves of lengths
$l_s=\length P_s,\; s=1,2$  such that
$\supp\mP_1\cap\supp\mP_2=\emptyset.$
It can be easily checked that the Krull--Schmidt partner $\dX'$ is not isomorphic to $X.$
Indeed, the integral bilinear form $\chi(E, F)=\sum_m (-1)^m \dim\Hom(E, F[m])$ on $K_0(X)$ goes through $\ZZ^2=H^{ev}(X, \ZZ).$
In this case the forms $\chi$ for $X$ and for $\dX'$ are respectively equal to
\begin{align}
\chi_{X}=\begin{pmatrix}
1-g & 1\\
-1 & 0
\end{pmatrix}\qquad\text{and}\qquad
\chi_{\dX'}=\chi_t:=\begin{pmatrix}
t & 1\\
-1 & 0
\end{pmatrix}, \qquad\text{where}\quad t=1-g -l_1 l_2.
\end{align}
The integral bilinear forms $\chi_t$ are not equivalent for different $t.$
  Hence, the categories $\prf X$ and $\prf\dR'$ are not equivalent.
Moreover, the categories $\prf\dR'$ for different $t$ are not equivalent to each other.
Thus, for any smooth projective curve $X$ we obtain infinitely many different Krull--Schmidt partners.
For the same $t$ we have many Krull-Schmidt partners that depend on the torsion sheaves $\mP_1$ and $\mP_2.$
It is reasonable to expect that these Krull-Schmidt partners have nontrivial moduli spaces.
The case of $X=\PP^1$ and two points $\mP_s=p_s,\; s=1,2$ is discussed in \cite[3.1]{O_q}.

\section{Finite dimensional algebras, quasi-hereditary algebras, and gluing}\label{quasi_her_alg}

\subsection{Finite dimensional algebras}
In this section we consider derived noncommutative schemes which are related to finite dimensional algebras.

Let $\kk$ be a field and let $\Ga$ be a finite dimensional $\kk$\!--algebra with Jacobson radical $\rad.$
The quotient algebra $\overline{\Ga}=\Ga/\rad$ is semisimple. We will assume that the algebra $\Ga$ is basic. This means that the algebra $\overline{\Ga}$ is isomorphic to $\underbrace{\kk\times\cdots\times \kk}_{n},$ where $\kk$ is the base field.
Denote by $\{e_1,\ldots, e_n\}$ a complete sequence of primitive orthogonal idempotents of $\Ga,$
so that $\sum_{i=1}^ne_i=1.$
Let $\Pi_i=e_i\Ga$ with $1\le i\le n$ be the corresponding indecomposable projective right $\Ga$\!--modules, and let $\Si_i=e_i\Ga/e_i\rad$  be the simple
right $\Ga$\!--modules. The quotient algebra $\overline{\Ga}=\Ga/\rad$ is isomorphic to
$\bigoplus_{i=1}^n \Si_i$ as $\Ga$\!--module. Since $\Ga$ is basic, all $\Si_i$ are one dimensional as $\kk$\!--vector spaces.

Denote by $\Md\Ga$
the category of right $\Ga$\!--modules. The full subcategory of finitely generated
right $\Ga$\!--modules will be denoted by $\md\Ga.$
Any algebra $\Ga$ can be considered as a DG algebra. The derived category $\D(\Ga)$ of all DG modules over this DG algebra
is nothing but the unbounded derived category $\D(\Md\Ga)$ and the DG category $\SF\Ga$ of semi-free DG modules is an enhancement of
this triangulated category. The triangulated category $\prf\Ga$ of perfect DG modules (from now on, perfect complexes)
consists of all bounded complexes of finitely generated projective modules. The DG category $\prfdg\Ga$ is a natural enhancement of
$\prf\Ga,$ and it defines a derived noncommutative scheme $\dW=\prfdg\Ga$ (see Definition \ref{noncommutative_scheme}).
Since the algebra $\Ga$ is finite dimensional, the noncommutative scheme $\dW$ is proper.

We can also consider the bounded derived category $\D^{b}(\md\Ga)$ of finitely generated $\Ga$\!--modules. It contains the triangulated category $\prf\Ga$
as a full triangulated subcategory.
Moreover, the subcategory of perfect complexes $\prf\Ga\subseteq \D^{b}(\md\Ga)$ is equivalent to the whole bounded derived category $\D^{b}(\md\Ga)$ if and only if the algebra
$\Ga$ is of finite global dimension. In this case the derived noncommutative scheme $\dW=\prfdg\Ga$ is regular (see \cite{Ro}).

\subsection{Quasi-hereditary algebras}

Let $N_1, \ldots , N_k$ be finitely generated right $\Ga$\!--modules. We denote by $\Fl(N_1,\ldots , N_k)$
the full subcategory of the abelian category $\md\Ga$ consisting of all modules $M$ that admit a finite filtration
$0 = M_0\subseteq M_1\subseteq\cdots\subseteq M_s = M$ such that each factor $M_p/M_{p-1}$ is isomorphic to an object of
the form $N_j.$

Most of our constructions will depend on a linear order of the sequence of
idempotents. This ordering will be denoted by $\be = (e_1, e_2, \ldots , e_n).$
The idempotents $e_i+e_{i+1}+\cdots+e_n$ for $1\le i\le n$
will be denoted by $\varepsilon_i$ and we put $\varepsilon_{n+1} = 0.$

\begin{definition}
For each $1\le i \le n$ define the {\sf standard module} $\Delta_i$ as the largest
quotient of the projective module $\Pi_i$ having no simple composition factors $\Si_j$ with
$j > i.$
\end{definition}
In other words, we have the following definition of the standard modules $\Delta_i=e_i\Ga/e_i\Ga\varepsilon_{i+1}\Ga$
and this definition depends on the ordering $\be.$

Let us denote by $\Theta_i$ the kernels of the natural epimorphisms $\Pi_i\twoheadrightarrow\Delta_i,$ and by  $\Xi_i$
the kernels of the natural surjections
$\Delta_i\twoheadrightarrow \Si_i.$ Thus, for all $1\le i\le n$
there are short exact sequences
\begin{align}\label{first_seq}
0\lto \Theta_i\lto \Pi_i\lto \Delta_i\lto 0\\
\label{second_seq} 0\lto \Xi_i\lto \Delta_i\lto \Si_i\lto 0
\end{align}

\begin{definition}\cite{CPS}\label{quasi-her}
The algebra $\Ga$ is called {\sf quasi-hereditary} (with respect to the ordering $\be$) if the following conditions hold:
\begin{itemize}
\item[1)] the modules $\Xi_i$ belong to $\Fl(\Si_1,\ldots, \Si_{i-1})$ for all $1\le i\le n;$
\item[2)] the modules $\Theta_i$ belong to $\Fl(\Delta_{i+1},\ldots,\Delta_n)$ for all $1\le i\le n.$
\end{itemize}
\end{definition}

Condition 1) of Definition \ref{quasi-her} means that all the modules $\Delta_i$ are Schurian, i.e.  $\End_{\Ga}(\Delta_i)$ are division rings for all $1\le i\le n.$
In our case, when $\Ga$ is a basic algebra, condition 1) implies that $\End_{\Ga}(\Delta_i)\cong\kk$
for all $i.$ In particular,
there is an isomorphism $\Delta_1\cong \Si_1,$ while, by definition, $\Delta_n\cong \Pi_n.$
Property 2) also implies that for every $1\le i\le n$ the projective module $\Pi_i$ belongs to the subcategory $\Fl(\Delta_i,\ldots,\Delta_n).$
The following proposition is well-known.
\begin{proposition}
Let $\Ga$ be a basic quasi-hereditary algebra. Then $\Ga$ has finite global dimension and the sequence of the standard modules $(\Delta_1,\ldots, \Delta_n)$ forms a full exceptional collection in the triangulated category $\prf\Ga.$
\end{proposition}
\begin{proof}
A descending induction and short exact sequences (\ref{first_seq}) give us that all $\Delta_i$ belong to $\prf\Ga,$ but an ascending
induction and short exact sequences (\ref{second_seq}) show us that all simple modules $\Si_i$ also belong to  $\prf\Ga.$
This implies that $\Ga$ has finite global dimension and $\prf\Ga\cong \D^b(\md\Ga).$

Since $\Delta_i$ belongs to $\Fl(\Si_1,\ldots, \Si_{i}),$ we have $\Hom(\Pi_j, \Delta_i)=0$ when $j>i.$
Applying this to $\Delta_n=\Pi_n$ we obtain that $\Ext^k(\Delta_n, \Delta_i)=0$ for all $k\ge 0$ and $i<n.$
Now by descending induction by $j$ we can show that $\Ext^k(\Delta_j, \Delta_i)=0$ for all $k\ge 0$ and $i<j.$
Indeed, we know that $\Ext^k(\Pi_j, \Delta_i)=0$ for all $k\ge 0$ and $i<j.$ Moreover, by induction hypothesis
$\Ext^k(\Theta_j, \Delta_i)=0$ for all $k\ge 0$ and $i<j,$ because $\Theta_i$ belongs to $\Fl(\Delta_{i+1},\ldots,\Delta_n).$
Now the short exact sequence (\ref{first_seq}) for $\Delta_j$ implies vanishing of all Ext's from $\Delta_j$ to $\Delta_i$
when $j>i.$

Since $\Theta_i$ belongs to $\Fl(\Delta_{i+1},\ldots,\Delta_n),$ we have $\Ext^k(\Theta_i, \Delta_i)=0$ for all $k\ge 0.$ Hence, for any $i$ we obtain
$
\Ext^k(\Delta_i, \Delta_i)\cong \Ext^k(\Pi_i, \Delta_i)\cong \Ext^k(\Pi_i, \Si_i).
$
This implies that $\Ext^k(\Delta_i, \Delta_i)=0,$ when $k>0$ and $\End(\Delta_i)\cong\kk.$ Thus,
the sequence of the standard modules $(\Delta_1,\ldots, \Delta_n)$ forms a full exceptional collection in the triangulated category $\prf\Ga\cong\D^b(\md\Ga).$
\end{proof}

The Proposition above implies that the derived noncommutative scheme $\dW=\prfdg\Ga$ is smooth and proper and can be obtained as a gluing of $n$ copies of the point $\pt.$
For any $1\le i\le n$ we denote by $\T_i\subseteq\prf\Ga$  the full admissible subcategory
 $\T_i=\langle \Delta_1,\ldots, \Delta_i\rangle$
 that is generated by the exceptional subcollection
$(\Delta_1,\ldots,\Delta_i).$
Let $\U_{i+1}={}^{\perp}\T_i$ be the left orthogonal, then it is generated by the exceptional subcollection $(\Delta_{i+1},\ldots,\Delta_n),$
i.e.  $\U_{i+1}=\langle\Delta_{i+1},\ldots,\Delta_n\rangle$ and it is also admissible. The following Lemma is evident.

\begin{lemma}\label{lemma_quasi} For every $1\le i\le n$ the subcategories $\T_i,\; \U_i\subseteq\prf\Ga$ satisfy the following properties:
\begin{itemize}
\item[1)] $\T_i$ contains the simple modules
$\{\Si_1,\ldots, \Si_i\}$ and is generated by this set of objects;
\item[2)] $\U_i$ contains the projective modules
$\{\Pi_{i},\ldots, \Pi_n\}$ and is generated by this set of objects.
\end{itemize}
\end{lemma}

Let $\dT_i$ and $\dU_i$ be full DG subcategories of the DG category of perfect complexes $\prfdg\Ga$
with the same objects as $\T_i$ and $\U_i,$ respectively, i.e. $\dT_i$ and $\dU_i$ are the induced DG enhancements for $\T_i$ and $\U_i.$

There are two recursive constructions of quasi-hereditary algebras described
in the literature. We will use the construction based on extensions of centralizers, described in  \cite{DR1}.
Actually, when we have an algebra $\Ga,$ we can consider a sequence of algebras $\Ga_k\cong\varepsilon_k\Ga\varepsilon_k,$
where $1\le k\le n.$ We obtain that $\Ga_n$ is isomorphic to the field $\kk,$ while the algebra $\Ga_1$ is isomorphic to $\Ga.$
The algebras $\Ga_k$ are the algebras of endomorphisms $\Ga_k\cong\End_{\Ga}(\bigoplus_{i=k}^{n}\Pi_i).$

\begin{proposition}\label{decomp_hereditary}
Let $(\Ga, \be)$ be a basic quasi-hereditary algebra with indecomposable projective modules
$\Pi_1,\ldots,\Pi_n.$ Let $\Ga_k$ be the algebra of endomorphisms $\End_{\Ga}(\bigoplus_{i=k}^{n}\Pi_i).$ Then the following properties hold.
\begin{itemize}
\item[1)] For any $1\le k\le n$ the algebra
$\Ga_k$ is basic and quasi-hereditary.

\item[2)] For any $1\le k\le n$ the DG category $\prfdg\Ga_k$ is quasi-equivalent to the DG subcategory
$\dU_k\subseteq\prfdg\Ga.$
Moreover, under this quasi-equivalence the indecomposable projective $\Ga_k$\!--modules go to indecomposable projective $\Ga$\!--modules, and standard $\Ga_k$\!--modules
go to standard $\Ga$\!--modules.

\item[3)] For any $1\le k < n$ the DG category of perfect complexes $\prfdg\Ga_k$ is quasi-equivalent to the gluing $\pt\underset{\;\mT_k}{\oright}\prfdg\Ga_{k+1}$
via the left DG $\Ga_{k+1}$\!--module $\mT_k=\dHom_{\prfdg\Ga}(\Delta_k, \bigoplus_{i=k+1}^{n}\Pi_i).$
\end{itemize}
\end{proposition}
\begin{proof}
By Lemma \ref{lemma_quasi} the projective module $\bigoplus_{i=k}^{n}\Pi_i$ belongs to $\U_k$ and generates it. Moreover, the DG algebra of endomorphisms of this object in the DG category
$\dU_k$ is quasi-isomorphic to the algebra $\Ga_k.$ Hence, by Proposition \ref{Kell1}, there is a quasi-equivalence between $\dU_k$ and $\prfdg\Ga_k.$
This quasi-equivalence is actually given by the DG functor $\dHom_{\Ga}(\bigoplus_{i=k}^{n}\Pi_i, -).$ Under this DG functor the projective modules $\Pi_j, j\ge k$ go to
the projective $\Ga_k$\!--module $\dHom_{\Ga}(\bigoplus_{i=k}^{n}\Pi_i, \Pi_j).$
Note that the projective modules $\Pi_j, j\ge k$ belong to $\dU_k.$

 The simple modules $\Sigma_j, j\ge k$ go to the simple $\Ga_k$\!--modules,
while the simple modules $\Sigma_j, j<k$ go to $0.$ However, the simple modules $\Sigma_j, j\ge k$ do not necessarily belong to $\dU_k$ and, hence, the simple
$\Ga_k$\!--modules do not go to the simple $\Ga$\!--modules.
Besides, the standard modules $\Delta_j=e_j\Ga/e_j\Ga\varepsilon_{j+1}\Ga$ also go to the standard $\Ga_k$\!--modules, when $j\ge k,$ and go to $0,$
when $j<k.$ Moreover, the standard modules $\Delta_j, j\ge k$ belong to $\dU_k$ and, therefore, the standard $\Ga_k$\!--modules correspond to standard $\Ga$\!--modules $\Delta_j, j\ge k.$
We also have that exact sequences  (\ref{first_seq}) and (\ref{second_seq}) for $i\ge k$ go to the same exact sequence in $\md\Ga_k$ and  conditions 1) and 2) of Definition \ref{quasi-her}
hold. Thus, the algebra $\Ga_k$ is also  quasi-hereditary.
By Proposition \ref{gluing_semi-orthogonal}, the semi-orthogonal decomposition $\U_k=\langle\Delta_k, \U_{k+1}\rangle$ implies property 3).
\end{proof}

Any path algebra of a directed quiver with relations is quasi-hereditary in two different ways.
 First, we can take all the simple
modules as the standard modules. In this case the category $\Fl(\Delta_1,\ldots,\Delta_n)$ coincides
 with the whole abelian category $\md\Ga.$
 The second way is to take the
indecomposable projective modules as the standard modules. In this case the
subcategory $\Fl(\Delta_1,\ldots,\Delta_n)$ contains only
the finitely generated projective modules.

\subsection{Well-formed quasi-hereditary algebras}

Let $\Ga$ be a basic quasi-hereditary algebra and let $(\Delta_1,\ldots, \Delta_n)$ be the complete sequence of the standard $\Ga$\!--modules that forms a full exceptional collection
in the triangulated category of perfect complexes $\prf\Ga.$

\begin{definition}\label{wf}   An algebra $\Ga$ is {\sf well-formed}
if for every $1\le i\le n$ there exists a right $\Ga$\!--module $\Psi_i\in\Fl(\Delta_{i+1},\ldots, \Delta_n)$ and a morphism
$\pi_i: \Pi_i\to \Psi_i$ such that the canonical morphisms of the functors $\Hom(\Psi_i,-)\to \Hom(\Pi_i,-)$
  is an isomorphism on the subcategory $\U_{i+1}\subset\prf\Ga.$
\end{definition}
\begin{remark}
{\rm
In other words this property means that the complex  $\Pi_i\stackrel{\pi_i}{\to} \Psi_i$ belongs to the left orthogonal ${}^{\perp}\U_{i+1}$
in the category $\U_{i},$ and the corresponding projection of $\Pi_i$ on $\U_{i+1}$  is a module from the subcategory $\Fl(\Delta_{i+1},\ldots, \Delta_n).$
Note that the right orthogonal $\U_{i+1}^{\perp}$ in $\U_i$ is exactly the subcategory generated by the exceptional object $\Delta_i.$
}
\end{remark}

Note that $\Psi_n=0.$
Moreover, for $n-1$ the module $\Psi_{n-1}$ also exists for any quasi-hereditary algebra $\Ga.$
It is isomorphic to $\Pi_n^{m},$
where $m$ is the dimension of the space of homomorphisms $\Hom_{\Ga}(\Pi_{n-1}, \Pi_n).$
However, already for $i=n-2$ the existence of the module $\Psi_{n-2}\in\Fl(\Delta_{n-1}, \Delta_n)$ with the property given in Definition \ref{wf}
is an additional restrictive condition.

Recall that any exceptional collection in a proper triangulated category has the right and left dual exceptional collections (see, e.g., \cite{Bo}).
First, let us consider and  describe the left dual for the collection
$(\Delta_1,\ldots, \Delta_n).$
Denote by $\Iota_i$ the injective envelope of the simple module $\Si_i$ for all $1\le i\le n$ and define the costandard
modules $\nabla_i$ as the maximal submodules of $\Iota_i$ having no composition factors $\Si_j$ with $j>i.$
It is evident that $\nabla_1\cong\Delta_1\cong\Si_1.$ Moreover, it is not difficult to check that the collection $(\nabla_n,\ldots,\nabla_1)$ is a full exceptional collection in the triangulated category $\prf\Ga$ that is left dual to the collection
$(\Delta_1,\ldots, \Delta_n).$ The last property means that the following conditions hold
\[
\Ext^{l}(\Delta_i, \nabla_j)\cong
\begin{cases}
\kk, & \text{when}\quad i=j,\quad \text{and}\quad l=0,\\
0, & \text{otherwise.}
\end{cases}
\]
In this case we get that for every $1\le i\le n$ the admissible subcategory $\T_i=\langle\Delta_1,\ldots,\Delta_i\rangle$ coincides with the subcategory $\langle\nabla_i, \ldots, \nabla_1\rangle$ and the object $\nabla_{i+1}$ generates the right orthogonal
to the subcategory $\T_i$ in the category $\T_{i+1}.$

Now we consider a full exceptional collection $(\K_n,\ldots,\K_1)$ of objects in
$\prf\Ga$ that is right dual to the collection $(\Delta_1,\ldots, \Delta_n).$ Thus, we have
\[
\Ext^{l}(\K_i, \Delta_j)\cong
\begin{cases}
\kk, & \text{when}\quad i=j,\quad \text{and}\quad l=0,\\
0, & \text{otherwise.}
\end{cases}
\]
In particular, there are isomorphisms $\K_n\cong\Delta_n\cong\Pi_n.$
It directly follows from definition that the objects $\K_i$ are isomorphic to the complexes
$\{\Pi_i\stackrel{\pi_i}{\lto} \Psi_i\}.$ Hence, the property for a quasi-hereditary algebra to be  well-formed
can be also considered as a property of the right dual exceptional collection $(\K_n,\ldots,\K_1).$
Denote by $\psi_i: \Theta_i\to \Psi_i$ the composition of $\pi_i$ and the natural inclusion of $\Theta_i$ to $\Pi_i,$ and let
$\Upsilon_i$ be the cone of $\psi_i.$
For any $1\le i\le n$ we have the following commutative diagram  of exact triangles in the triangulated category $\prf\Ga.$
\begin{align}\label{diagr_wf}
\begin{CD}
&& \K_i @=\K_i\\
&& @VVV @VVV\\
\Theta_i @>>> \Pi_i @>>> \Delta_i\\
 @| @V\pi_i VV  @VVV\\
\Theta_i @>\psi_i >> \Psi_i @>>> \Upsilon_i\\
\end{CD}
\end{align}
and algebra $\Ga$ is well-formed if the object $\Psi_i$ belongs not only $\U_{i+1}$ but it is also in $\Fl(\Delta_{i+1},\ldots, \Delta_n)$ for any $1\le i\le n.$
We also have that the objects $\Theta_i, \Psi_i, \Upsilon_i$ belong to $\U_{i+1},$ while $\Delta_i$ and $\K_i$ generate the right and left orthogonals $\U_{i+1}^{\perp}, {}^{\perp}\U_{i+1}$ in $\U_i,$ respectively.

\begin{remark}{\rm
Let us consider the algebra $\Ga=\kk Q/I$ of a directed quiver with relations $(Q, I)$ for which the set of vertices $Q_0=\{1,\dots , n\}$ is the ordered set of $n$ elements  and  $s(a)> t(a)$
for any arrow $a\in Q_1,$ where $s, t : Q_1\rightrightarrows  Q_0$
are the maps associating to each arrow its source and target.
In this case the algebra $\Ga$ is quasi-hereditary with respect to the ordering $\be = (e_1, e_2, \ldots , e_n).$
Moreover, for such ordering the standard modules $\Delta_i$ are isomorphic to the simple modules $\Si_i$ for all $1\le i\le n.$
The quasi-hereditary algebra $(\Ga, \be)$ is well-formed in this case because the projective module $\Pi_i$ belongs to the
subcategory ${}^{\perp}\U_{i+1}\subseteq\T$ for each $i.$  Thus, for all $1\le i\le n$ we have $\Psi_i=0.$
}
\end{remark}
\begin{remark}\label{quiver_not_wf}{\rm
Note that if we take the opposite ordering of idempotents, then  the algebra $\Ga$ is also quasi-hereditary, but in this case the standard modules are the indecomposable projective modules, and
the algebra is not necessarily well-formed.
The simplest example is a quiver $Q=(\xymatrix{
{\bullet}\ar[r]^{a} &{\bullet} \ar[r]^{b}&{\bullet}})$
with three vertices, two arrows $a, b,$ and with a single relation $ba=0.$
}
\end{remark}

\section{Geometric realizations of finite dimensional algebras}

\subsection{Geometric realizations of well-formed quasi-hereditary algebras}
In this section we discuss some geometric realizations for finite dimensional algebras.
We consider the case when an algebra is basic. By Theorem \ref{algebra}, any such algebra has a plain geometric realization.
However, Theorem \ref{algebra} only tells us that for any such algebra $\La$ there is
a perfect complex $\mE$ on a smooth projective scheme $X$ for which $\End(\mE)\cong\Lambda$ and $\Hom(\mE, \mE[l])=0$
when $l\ne 0.$
A reasonable question about finite dimensional algebras here is to find such a geometric realization, where the perfect complex $\mE$
is a vector bundle $\E$ on $X.$
We know that this question has a positive answer for any quiver algebra $\La$ by  Theorem \ref{exc_colllection_bundles}.
Moreover, in that case we can find $X$ and a vector bundle $\E$ such that the rank of $\E$ is equal to the dimension of $\La$ (see \cite[Cor. 2.8]{O_q}).

At first, we consider the case of a quasi-hereditary algebras and we will try to find a geometric realization such that
all standard modules $\Delta_i$ go to line bundles on $X.$

\begin{definition}\label{good_geom_real}
 Let $\Ga$ be a basic quasi-hereditary algebra over $\kk,$ and let $X$ be a smooth projective scheme.
 Let $\mG: \prfdg\Ga\to\prfdg X$ be a quasi-functor.
 We say that $\mG$ satisfies property (V) if the following conditions hold:
\begin{equation*}\label{V}
\parbox{.93\textwidth}{\begin{enumerate}
\item[(1)] the functor $G=\Ho(\mG): \prf\Ga\to \prf X$ is fully faithful, i.e. $\mG$ gives a plain geometric realization of
$\prfdg\Ga;$
\item[(2)] the standard modules $\Delta_i,\; 1\le i\le n$  go to line bundles $\L_i$ on $X$ under $G;$
\item[(3)] there is a line bundle $\N$ on $X$ such that $\N\in {}^{\perp}G(\prf\Ga)\subset\prf X,$ the line bundles
$\N\otimes \L_i^{-1}$ are generated by global sections,  and
$H^j(X, \N\otimes \L_i^{-1})=0$ when $j\ge 1$ for all $i=1,\dots,n.$
\end{enumerate}
}\leqno {\rm (V)}
\end{equation*}
\end{definition}

Since any module $M\in \Fl(\Delta_1,\ldots,\Delta_n)$ has a filtration with successive quotients being standard modules,
condition (2) of Definition \ref{good_geom_real} implies that any such $\Ga$\!--module goes to a vector bundle under the functor $G.$
Moreover, by the same reasoning, the vector bundle $G(M)$ has a filtration with successive quotients isomorphic to the line bundles
$\L_i.$ Now it is not difficult to check that  condition (3) of (V) implies the following condition:
\begin{equation*}
\parbox{.88\textwidth}{\begin{enumerate}
\item[(3')] there is a line bundle $\N$ on $X$ such that $\N\in {}^{\perp}G(\prf\Ga)$ and vector bundles
$\N\otimes G(M)^{\vee}$ are generated by global sections for all $M\in\Fl(\Delta_1,\ldots,\Delta_n)$ and have no higher cohomology.
\end{enumerate}}
\end{equation*}

Note that the quasi-functor $\mG: \prfdg\Ga\to\prfdg X$ gives an ff-morphism $\mg: X\to \dW$ where $\dW=\prfdg\Ga$ is the derived noncommutative scheme
related to the algebra $\Ga.$ Property 2) tells us that the standard modules $\Delta_i\in\prf\Ga$ go to the line bundles $\L_i$ on $X$
under the inverse image functor $\bL\mg^*.$

For any quasi-hereditary algebra  $\Ga$ the projective modules $\Pi_i$ belong to the subcategory $\Fl(\Delta_1,\ldots,\Delta_n)$
and, hence, they go to vector bundles under the functor $G.$ Denote by $\P_i$ the vector bundles $G(\Pi_i)$ for all $i=1,\dots, n.$
Since the functor $G$ is fully faithful, the vector bundles $\E_k=\bigoplus_{i=k}^n \P_i$ possess the following properties:
\[
\Hom_X(\E_k, \E_k)\cong \Ga_k \quad\text{and}\quad \Ext^l_X(\E_k, \E_k)=0 \quad\text{ for all}\quad l\ne 0.
\]
Thus, if a quasi-functor $\mG: \prfdg\Ga\to\prfdg X$ satisfies property (V), then  the algebra $\Ga$ goes to a vector bundle $\E_1$ under such a geometric realization.

The following proposition gives us an inductive step of a general construction.

\begin{proposition}\label{step}
Let $(\Ga, \be)$ be a basic well-formed quasi-hereditary algebra with indecomposable projective modules
$\Pi_1,\ldots,\Pi_n.$ Let $\Ga_{l}$ for $l=1,\dots, n$ be the algebras of endomorphisms $\End_{\Ga}(\bigoplus_{i=l}^{n}\Pi_i).$
Suppose there exist a smooth projective scheme $X_{k+1}$ and
 a quasi-functor
$\mG_{k+1}:\prfdg\Ga_{k+1}\to \prfdg X_{k+1}$ that satisfies property (V).
Then there are a smooth projective scheme $X_{k}$ and a quasi-functor $\mG_{k}: \prfdg\Ga_{k}\to \prfdg X_{k}$ that also satisfies
property (V). Moreover, the scheme $X_k\cong\PP(\F_k)$ is a projective  vector bundle over $X_{k+1},$ and the restriction
of $G_k$ on the subcategory $\prf\Ga_{k+1}\subset\prf\Ga_{k}$ is isomorphic to $\bL p^*\ldot G_{k+1},$
where $p:\PP(\F_k) \to X_{k+1}$ is the natural projection.
\end{proposition}
\begin{proof}
Consider the quasi-functor $\mG_{k+1}:\prfdg\Ga_{k+1}\to \prfdg X_{k+1}.$
By Proposition \ref{decomp_hereditary}, we can identify the DG category $\prfdg\Ga_{k+1}$ with the DG subcategory $\dU_{k+1}\subset\prfdg\Ga$ that is generated by the standard modules $\Delta_{k+1},\dots, \Delta_n.$
Consider the quasi-functor $\mG_{k+1}$ as a
quasi-functor from $\dU_{k+1}$ to $\prfdg X_{k+1}.$

Since $\mG_{k+1}$ satisfies property (V), the standard modules $\Delta_{k+1},\ldots,\Delta_n$ go to line bundles
$\L_{k+1},\ldots, \L_{n},$ respectively.
The DG subcategory of $\prfdg X_{k+1}$ that is generated by the line bundles $\L_{k+1},\ldots, \L_{n}$ is quasi-equivalent to $\dU_{k+1}$ and the quasi-functor $\mG_{k+1}$
establishes this quasi-equivalence.

Any module $M\in \Fl(\Delta_{k+1},\ldots,\Delta_n)$ goes to a vector bundle
under the functor $G_{k+1}=\Ho(\mG_{k+1}).$
Take now the projective module $\Pi_k$ and consider the short exact sequence
\begin{align}\label{res_stand}
0\lto \Theta_k\lto \Pi_k\lto\Delta_k\lto 0.
\end{align}
By the definition of a quasi-hereditary algebra, the module $\Theta_k$ belongs to $\Fl(\Delta_{k+1},\ldots,\Delta_n).$
The algebra $\Ga$ is well-formed and, hence, there is a module $\Psi_k\in \Fl(\Delta_{k+1},\ldots,\Delta_n)$
such that the complex $\K_k=\{\Pi_k\stackrel{\pi_k}{\lto}\Psi_k\}$ belongs to ${}^{\perp}\U_{k+1}.$
As in diagram (\ref{diagr_wf}) we denote by $\psi_k: \Theta_k\to \Psi_k$ the composition of $\pi_k$ with the natural inclusion of $\Theta_k$ into $\Pi_k.$
By Proposition \ref{decomp_hereditary}, the DG category $\dU_{k}\subset\prfdg \Ga$ is quasi-equivalent to a gluing
$\pt\underset{\;\mT_k}{\oright}\dU_{k+1},$ where the left DG $\dU_{k+1}$\!--module $\mT_k$ is
$\dHom_{\dU_k}(\Delta_k, \mQ)$ with $\mQ\in\dU_{k+1}.$ Using (\ref{res_stand}) as the resolution for the standard module $\Delta_k,$ we obtain the following
 quasi-isomorphism of  left DG $\dU_{k+1}$\!--modules:
 \[
\mT_k= \dHom_{\dU_k}(\Delta_k, \mQ)\cong \dHom_{\dU_{k+1}}(\Upsilon_k, \mQ)\quad\text{ with}\quad \mQ\in\dU_{k+1},
 \]
where $\Upsilon_k\in\dU_{k+1}$ is the complex $\{\Theta_k\stackrel{\psi_k}{\lto} \Psi_k\}$ concentrated in degrees $-1$ and $0.$

Consider the vector bundle $G_{k+1}(\Theta_k).$
The morphism $\psi_k$ induces a map $G_{k+1}(\psi_k)$ between the vector bundles $G_{k+1}(\Theta_k)$ and $G_{k+1}(\Psi_k)$ and also a map
$G_{k+1}(\Psi_k)^{\vee}\to G_{k+1}(\Theta_k)^{\vee}$
between the dual vector bundles.
By assumption (3) of (V) and its consequence (3'), there is a surjection
$(\N^{-1})^{\oplus m}\twoheadrightarrow G_{k+1}(\Theta_k)^{\vee}$ for some $m\in \NN.$

Consider the induced map $(\N^{-1})^{\oplus m}\oplus G_{k+1}(\Psi_k)^{\vee} \to G_{k+1}(\Theta_k)^{\vee}$ which is also
a surjection.
Denote by $\F$ the vector bundle on $X$ that is dual to the kernel of this surjection.
We obtain the following exact sequence of vector bundles on $X_{k+1}$
\begin{equation}\label{ext}
0\lto G_{k+1}(\Theta_k) \lto \N^{\oplus m}\oplus G_{k+1}(\Psi_k) \lto \F \lto 0.
\end{equation}
Since the line bundle $\N$ belongs to the orthogonal ${}^{\perp}G_{k+1}(\U_{k+1}),$ we obtain quasi-isomorphisms of
left DG $\dU_{k+1}$\!--modules
\[
\dHom_{\prfdg X_{k+1}}(\F, \mG_{k+1}(\mQ))\cong \dHom_{\dU_{k+1}}(\Upsilon_k, \mQ)\cong \dHom_{\dU_k}(\Delta_k, \mQ) \quad\text{ with}\quad \mQ\in\dU_{k+1}.
\]

Taking  $m$ sufficiently large, we can assume that the rank of $\F$ is greater  than $2.$
Let us consider the projective bundle $p: \PP(\F)\to X_{k+1}$ and denote it by $X_{k}.$
There are natural exact sequences on $X_{k}$ of the following form:
\[
0\to \varOmega_{X_{k}/X_{k+1}}(1)\to p^*\F^{\vee}\to \cO_{X_{k}}(1)\to 0\quad\text{and}\quad
0\to \cO_{X_{k}}(-1)\to p^*\F\to \T_{X_{k}/X_{k+1}}(-1)\to 0,
\]
where $\cO_{X_{k}}(-1)$ is the tautological line bundle, and $\T_{X_{k}/X_{k+1}}, \Omega_{X_{k}/X_{k+1}}$ are the relative tangent and the relative cotangent bundles, respectively.
We have $\bR p_*\cO_{X_{k}}(1)\cong \F^{\vee}$ and $\bR p_*\cO_{X_{k}}(-1)=0.$

Denote by  $\wt{\L}_i$ the pull back line bundles $p^*\L_i$ for $i=k+1,\dots, n$ and consider the DG subcategory $\wt\dU_{k+1}\subset\prfdg X_k$
that is generated by these line bundles. Since the functor $\bL p^*$ is fully faithful, the DG category $\wt\dU_{k+1}$ is quasi-equivalent to $\dU_{k+1}.$
Put  $\wL_{k}=\cO_{X_k}(-1)$ and consider the DG subcategory $\dV\subset\prfdg X_k$ generated by the line bundles $\wt{\L}_k,\ldots,\wt{\L}_n.$
Since $\bR p_*\cO_{X_{k}}(-1)=0,$ the line bundles $\wt{\L}_k,\ldots,\wt{\L}_n$ form a full exceptional collection
in the homotopy category $\V=\Ho(\dV)$ and the category $\V\subset\prf X_k$ is admissible.

By Proposition \ref{gluing_semi-orthogonal}, the DG category $\dV$ is quasi-equivalent to a gluing $\pt\underset{\;\mS}{\oright}\wt\dU_{k+1}$
(where $\pt=\prfdg\kk$)
with respect to a left DG $\wt\dU_{k+1}$\!--module $\mS=\dHom_{\dV}(\wt\L_k, -).$
There is the following sequence of quasi-isomorphisms of left DG $\dU_{k+1}$\!--modules
\begin{multline*}
\mS=\dHom_{\dV}(\wt\L_k, \bL p^*\mG_{k+1}(\mQ))\cong \dHom_{\prfdg X_k}(\cO_{X_k}, \bL p^*\mG_{k+1}(\mQ)\otimes \cO_{X_k}(1) )\cong\\
\cong
\dHom_{\prfdg X_{k+1}}(\cO_{X_{k+1}}, \mG_{k+1}(\mQ)\otimes \F^{\vee} )
\cong \dHom_{\prfdg X_{k+1}}(\F, \mG_{k+1}(\mQ))\cong \dHom_{\dU_k}(\Delta_k, \mQ)\cong\mT_{k}.
\end{multline*}

Thus, by Proposition \ref{gluing_quasifunctors}, we obtain that the DG category $\dV\cong\pt\underset{\;\mS}{\oright}\wt\dU_{k+1}$ is quasi-equivalent to the DG category
$\dU_k\cong \pt\underset{\;\mT_k}{\oright}\dU_{k+1}.$ By Proposition \ref{decomp_hereditary}, these DG categories are also quasi-equivalent to $\prfdg\Ga_k.$
Therefore, we obtain a quasi-functor $\mG_k:\prfdg\Ga_k\to \prfdg X_k$ that establishes a quasi-equivalence between $\prfdg\Ga_k$ and $\dV.$
By construction,  conditions (1) and (2) of (V) hold for the quasi-functor $\mG_k.$

Finally, we have to show that condition (3) also holds for an appropriate line bundle $\wt \N$ on
$X_k.$  Choosing $\wt \N$ as a line bundle of the form $\cO_{X_k}(1)\otimes p^*\R^{\otimes s},$ where
$\R$ is an ample line bundle on $X_{k+1}$ and $s$ is sufficiently large, we can guarantee that  condition (3) will hold.
Indeed, since the rank of $\F$ is greater than $2,$
the line bundle $\wt \N$ belongs to ${}^{\perp}G_k(\prf\Ga_k).$
Moreover, for $k<i\le n$ we have isomorphisms
\[
H^j(X_k, \; \wt{\L}_i^{-1}\otimes \wt{\N})\cong H^j(X_{k+1}, \L^{-1}_i\otimes\F^{\vee}\otimes\R^{\otimes s}).
\]
Additionally, for the line bundle $\wt{\L}_{k}^{-1}\otimes \wt{\N}$ we also have
\[
H^j(X_k, \; \wt{\L}_{k}^{-1}\otimes \wt{\N})\cong H^j(X_{k+1}, S^2(\F^{\vee})\otimes\R^{\otimes s}).
\]
Taking a sufficiently large $s,$ we obtain vanishing of all cohomology for $j>0,$ by Serre vanishing Theorem, and can guarantee that
all these bundles are generated by global sections  on $X_k.$
Since the natural maps $p^*\F^{\vee}\to\cO_{X_k}(1)$ and $p^*S^2(\F^{\vee})\to\cO_{X_k}(2)$ are surjective,
condition (3) of (V) also holds for the quasi-functor $\mG_k.$
\end{proof}

We can also give a precise construction of the vector bundles on $X_k$ that are the images of the projective modules under the functor $G_k.$
Denote by $\P_{k+1},\ldots,\P_n$ the vector bundles on $X_{k+1}$ that are the images of the projective modules $\Pi_{k+1},\ldots, \Pi_n\in\Fl(\Delta_{k+1},\ldots,\Delta_n)$ under the functor $G_{k+1}.$
Under the functor $G_k$ these projective modules go to the vector bundles $\wt{\P_i}=p^*\P_i.$
Let us construct the vector bundle $\P_k$ which is the image of $\Pi_k.$
Consider the sequence of isomorphisms
\begin{multline*}
\Ext^1_{X_k}(\cO_{X_k}(-1),\; p^*G_{k+1}(\Theta_k))\cong H^1(X_k, p^*G_{k+1}(\Theta_k)\otimes \cO_{X_k}(1))\cong H^1(X_{k+1}, G_{k+1}(\Theta_k)\otimes \F^{\vee})\\
\cong \Ext^1_{X_{k+1}}(\F,\; G_{k+1}(\Theta_k)).
\end{multline*}
The element $e\in\Ext^1_{X_{k+1}}(\F,\; G_{k+1}(\Theta_k)),$ which defines the short exact sequence (\ref{ext}),
gives some element $e'\in\Ext^1_{X_k}(\cO(-1),\; p^* G_{k+1}(\Theta_k)).$ The element $e'$ induces the following extension
\begin{equation}\label{projective}
0\lto p^*G_{k+1}(\Theta_k)\lto \wt{\P}_{k}\lto\cO_{X_k}(-1)\lto 0
\end{equation}
that can be considered as the definition of the vector bundle $\wt{\P}_{k}.$
Finally, the algebra $\Ga_k$ itself goes to the vector bundle  $\wt{\E_k}=\bigoplus_{i=k}^n \wt{\P}_i$ under the functor $G_k.$

Proposition \ref{step} as an induction step implies the following theorem.

\begin{theorem}\label{main_quiver}
Let $(\Ga, \be)$ be a basic well-formed quasi-hereditary algebra. Then there exist a smooth projective scheme $X$ and a quasi-functor
$\mG: \prfdg\Ga\to\prfdg X$ such that the following conditions hold:
\begin{enumerate}
\item[1)] the induced homotopy functor $G=\Ho(\mG):\prf\Ga\to \prf X$ is fully faithful;
\item[2)] the standard modules $\Delta_i$ go to line bundles $\L_i$ on $X$ under $G;$
\item[3)] the scheme $X$ is a tower of projective bundles and has a full exceptional collection.
\end{enumerate}
\end{theorem}
\begin{proof}
The proof proceeds by induction on $n.$ The base of induction is $n=1$ and $\Ga=\kk.$
In this case $X=\PP^1,$ the quasi-functor $\mG$ sends $\Ga$ to $\cO_{\PP^1},$ and $\N=\cO(1).$
The inductive step is Proposition \ref{step}. By construction, the scheme $X$ is a tower of projective bundles and, hence,
 has a full exceptional collection.
\end{proof}

\begin{remark}
{\rm Let us note that the property of a quasi-hereditary algebra $\Ga$ to be well-formed is essential.
Indeed, consider the algebra of the quiver $Q=(\xymatrix{
{\bullet}\ar[r]^{a} &{\bullet} \ar[r]^{b}&{\bullet}})$
with the relation $ba=0$ as in Remark \ref{quiver_not_wf}. This algebra is quasi-hereditary but not well-formed in the case when the standard modules
are the indecomposable projective modules. It is easy to see that we can not find a geometric realization for which the standard modules go to line bundles.
Indeed,  any non-zero morphism of line bundles on a smooth
irreducible projective scheme is an isomorphism at the generic point and this contradicts the fact that $ba=0.$
}
\end{remark}

\subsection{Auslander construction and geometric realizations of finite dimensional algebras}
Now we will discuss geometric realizations for an arbitrary basic finite dimensional algebra.

Let $\La$ be a basic finite-dimensional algebra over a field $\kk.$
Denote by $\rd$ the Jacobson radical of $\La.$ We know that $\rd^N=0$ for some $N.$
Define the index of nilpotency of $\La$ as the smallest integer $N$
such that $\rd^N=0.$
The following amazing result was proved by M.~Auslander.
\begin{theorem}\cite{Au}
Let $\Lambda$ be a finite-dimensional algebra with index of nilpotency $N.$ Then the finite-dimensional algebra
$\wt{\Ga}=\End_{\Lambda}(\bigoplus_{p=1}^{N} \Lambda/\rd^p)$ has the following properties:
\begin{enumerate}
\item[1)] $\gldim\wt{\Ga}\le N+1;$
\item[2)] there is a finite projective $\wt{\Ga}$\!--module $\Pi$ such that $\End_{\wt{\Ga}}(\Pi)\cong \La.$
\end{enumerate}
\end{theorem}

The algebra $\wt{\Ga}$ is usually not basic, even if the algebra $\La$ is basic.
Indeed, since $\La$ is basic, each non-projective
indecomposable summand occurs in a direct decomposition of $\bigoplus_{p=1}^{N} \Lambda/\rd^p$ with
multiplicity $1,$ whereas each indecomposable projective $\La$\!--module of Loewy length
$l$ occurs with multiplicity $N - l + 1.$ If we delete the repeated copies of the indecomposable projective summands of $\bigoplus_{p=1}^{N} \Lambda/\rd^p,$
one obtains a module whose
endomorphism ring is basic and Morita equivalent to $\wt{\Ga}.$ Let us describe it more precisely.

Let us denote by $\{f_1,\ldots, f_m\}$ a complete sequence of  primitive orthogonal idempotents of the basic algebra $\La$
so that $\sum_{j=1}^m f_j=1.$
Let $P_j=f_j\La$ with $1\le j\le m$ be the corresponding indecomposable projective $\La$\!--modules and let $S_j=P_j/P_j\rd$  be the simple
right $\La$\!--modules. The quotient algebra $\overline{\La}=\La/\rd$ is semisimple, and it is isomorphic to
$\bigoplus_{j=1}^m S_j$ as $\La$\!--module. Moreover, since $\La$ is basic, the algebra $\overline{\La}$ is isomorphic to $\underbrace{\kk\times\cdots\times \kk}_{m}.$

We will choose a linear order on the set of idempotents $\{ f_1,\ldots, f_m\}$ such that $\LL(P_k)\ge \LL(P_l),$ when $k\le l.$ Here $\LL(P_j)$ denotes the Loewy length
of the projective module $P_j,$ i.e. the minimal positive integer $i$ such
that $P_j\rd^i = 0.$  Consider the finite set $T$ consisting of pairs $(j,l)$   with
$1\le j\le m$ and $1\le l\le \LL(P_j).$ The cardinality of $T$ is equal to $n=\sum_{j=1}^{m} \LL(P_j).$
With any element $t=(j,l)$ of $T$ we associate the $\La$\!--module $M_{t}\cong P_j/P_j\rd^l.$
We introduce a linear order on $T$ by the following rule:
$t_1=(j_1, l_1)< t_2=(j_2, l_2)$ if and only if $l_1\ge l_2$ and $j_1< j_2$ in the case $l_1=l_2.$
In particular, we have isomorphisms $M_1=P_1,$ while $M_n=S_m.$

Consider the $\La$\!--module $M=\bigoplus_{t\in T}M_t$ and denote by $\Ga$ the algebra of endomorphisms $\End_{\La}(M).$
Since all the  modules $M_t$ are indecomposable and non-isomorphic to each other,  the algebra $\Ga$ is basic.
Its complete set of  primitive orthogonal idempotents $\{e_1,\ldots, e_n\}$ is in bijection with the set $T$ and  has a linear order
introduced above. We fix this linear order and denote it by $\be.$
The $\La$\!--module $M$ gives a standard functor between the abelian categories of modules
\[
\Hom_{\La}(M, -): \Md\La\lto \Md\Ga
\]
that is left exact. For a $\La$\!--module $N$ we denote by $\widehat{N}$ the $\Ga$\!--module $\Hom_{\La}(M, N)$ for shortness.
Since the $\La$\!--module $M$ is a generator in the abelian category $\Md\La,$ for any two $\La$\!--modules $N_1$ and $N_2$ the canonical map
\begin{equation}\label{funct}
\Hom_{\La}(N_1, N_2)\stackrel{\sim}{\lto}\Hom_{\Ga}(\widehat{N}_1, \widehat{N}_2)
\end{equation}
is an isomorphism.
By construction, the indecomposable projective $\Ga$\!--modules $\Pi_t=e_t\Ga$ are isomorphic to $\widehat{M_t}.$ The submodules $M_t\rd\subset M_t$
induce the submodules $\Theta_t=\widehat{M_t\rd}\subset\Pi_t,$ and we denote by $\Delta_t$ the quotient $\Ga$\!--modules $\Pi_t/\Theta_t.$

\begin{proposition}
Let $\La$ be a basic finite dimensional algebra over a field $\kk.$ Then the algebra $(\Ga, \be)$ constructed above is a basic well-formed quasi-hereditary $\kk$\!--algebra for which the
standard modules are exactly $\Delta_t=\widehat{M_t}/\widehat{M_t\rd},$ with $t\in T.$
\end{proposition}
\begin{proof}
We already know that $\Ga$ is basic. It is proved in \cite{DR2} that the algebra $(\Ga, \be)$ is quasi-hereditary.
However, let us show this.

At first, we calculate $\Hom_{\Ga}(\Pi_{t'}, \Delta_t)$ for any $t', t\in T.$ For each $t=(j, l)\in T$ we have $M_t/M_t\rd\cong P_j/P_j\rd\cong S_j,$ where  $M_t=P_j/P_j\rd^l$ as above. Therefore, there is the following commutative diagram:
\begin{align*}
\xymatrix{
0\ar[r]&\Theta_t\ar[r]\ar@{=}[d]&\Pi_t\ar[r]\ar@{=}[d]&\Delta_t\ar[r]\ar@{_{(}->}[d]&0\\
0\ar[r]&\widehat{M_t\rd}\ar[r]&\widehat{M_t}\ar[r]&\widehat{S}_j
}
\end{align*}
Since $S_j=M_s$ for $s=(j, 1)\in T,$ we obtain that $\widehat{S}_j\cong \Pi_s\cong\Delta_{s}$ is a projective and standard module simultaneously for any $1\le j\le m.$
The isomorphisms
\[
\Hom_{\Ga}(\Pi_{t'}, \widehat{S}_j)\cong\Hom_{\La}(M_{t'}, S_j)\cong\Hom_{\La}(P_{j'}, S_j)
\]
tell us that if $\Hom_{\Ga}(\Pi_{t'}, \widehat{S}_j)\ne 0,$ then it is one-dimensional and  $t'=(j, l')$ with the same $j.$
Since for any $t=(j, l)$ there is an inclusion $\Delta_t\subseteq \widehat{S}_j,$ we conclude that if
$\Hom_{\Ga}(\Pi_{t'}, \Delta_t)\ne 0,$ then $t'=(j, l')$ with the same $j.$
Furthermore, a nontrivial morphism $\Pi_{t'}=\widehat{M}_{t'}\to\Delta_t\subset\widehat{S}_j$ is induced by a nontrivial morphism $M_{t'}\to S_j,$ on the one hand.
On the other hand, it can be lifted to a map $\Pi_{t'}\to\Pi_t=\widehat{M}_t$ because $\Pi_{t'}$ is projective.
Therefore, the map $M_{t'}\to S_j$ should go through a map to $M_t.$ Hence, in this case $l'\ge l.$
Thus, for any $t=(j, l)$ and $t'=(j', l')$ we obtain
\begin{align}\label{hom_pr_st}
\Hom_{\Ga}(\Pi_{t'}, \Delta_t)=
\begin{cases}
\kk,\quad \text{if}\quad j'=j \quad\text{and}\quad l'\ge l,\\
0,\quad \text{otherwise.}
\end{cases}
\end{align}

In particular, we obtain that $\dim_{\kk}\Delta_t=\LL(P_j)-l+1,$ for $t=(j, l).$ Moreover, the sequence
\begin{align}\label{comp_ser}
0\subset\Delta_{(j,\LL(P_j))}\subset\cdots\subset\Delta_{(j, 1)}=\widehat{S}_j\quad\text{with}\quad \Delta_{(j,l)}/\Delta_{(j, l+1)}\cong\Si_{(j,l)},
\end{align}
where $\Si_{(j,l)} $ are the corresponding simple $\Ga$\!--modules, gives a composition series for $\widehat{S}_j\cong \Pi_s\cong\Delta_{s},$ where $s=(j, 1).$
This implies that any simple quotient of a standard module $\Delta_t$ is only $\Si_t,$ and any submodule of the standard module $\Delta_t$ with $t=(j, l)$ is isomorphic to a standard module $\Delta_{t'}$
where $t'=(j, l')$ with $l'\ge l.$
As a consequence, we also obtain that for any $t=(j, l)$ the kernel $\Xi_t$ of the canonical morphism from $\Delta_t$ to the corresponding simple module $\Si_t$  is isomorphic to
the standard module $\Delta_{t'}$ with $t'=(j, l+1).$ Hence, taking into account the composition series (\ref{comp_ser}), we see that $\Xi_t\cong\Delta_{t'}$ belongs to the subcategory $\Fl(\Si_1,\ldots, \Si_{t-1}).$
Thus, condition 1) of Definition \ref{quasi-her} holds.

To prove that the algebra $\Ga$ is quasi-hereditary, we have to show that the modules $\Theta_t$ belong to the subcategories
$\Fl(\Delta_{t+1},\ldots, \Delta_{n}).$ Let us prove that any $\Ga$\!--module of the form $\widehat{N}$ belongs to the subcategory $\Fl(\Delta_{1},\ldots, \Delta_{n}).$
We proceed by induction on the length of a $\La$\!--module $N.$ The base of induction are the simple modules $S_j:$ and we already know that $\widehat{S}_j$ are standard by themselves.
  Consider an exact sequence
\[
0\lto N'\lto N\lto S_k\lto 0
\]
coming from a composition series of $N.$ By induction hypothesis, the $\Ga$\!--module $\widehat{N}'$ belongs to $\Fl(\Delta_{1},\ldots, \Delta_{n}).$
Besides, the quotient module $\widehat{N}/\widehat{N}'$  is standard as a nonzero submodule of $\widehat{S}_k.$ Thus, $\widehat{N}$ belongs to $\Fl(\Delta_{1},\ldots, \Delta_{n})$ too.

Let us consider the $\Ga$\!--module $\Theta_t=\widehat{M_t\rd}$ for $t=(j, l).$ We would like to show that $\Theta_t$ belongs not only to
the subcategory $\Fl(\Delta_{1},\ldots, \Delta_{n})$ but to the subcategory $\Fl(\Delta_{t+1},\ldots, \Delta_{n})$ as well. It is enough to check that any morphism
from $\Pi_{t'}$ to $\Theta_t=\widehat{M_t\rd}$ can be factorized through $\Pi_{s}$ with $s>t.$
The Loewy length of the $\La$\!--modules $M_t\rd=P_j\rd/P_j\rd^l$ is equal to $l-1.$
Hence, any morphism from a projective module $P_k$ to $M_t\rd$ goes through the module $P_k/P_k\rd^{l-1}.$
This implies that $\Theta_t\in \Fl(\Delta_{t+1},\ldots, \Delta_{n}).$
Thus, condition 2) of Definition \ref{quasi-her} also holds,
and the algebra $(\Ga, \be)$ is quasi-hereditary with $\Delta_t, t\in T$ as the standard modules.

Finally, let us check that $\Ga$ is well-formed. Consider the projective $\Ga$\!-module $\Pi_t$ with $t=(j, l).$
The canonical morphism $P_j/P_j\rd^l\to P_j/P_j\rd^{l-1}$
induces a canonical map
$\pi_t:\Pi_t\to \Pi_{t'},$ where $t'=(j, l-1).$ The calculation above
shows that the natural map $\Hom(\Pi_{t'}, \Delta_s)\to \Hom(\Pi_t, \Delta_s)$
  is an isomorphism for any $\Delta_s,$ when $s> t.$ Since $\Delta_{t+1},\ldots,\Delta_n$ generate  the subcategory $\U_{t+1}\subset\prf\Ga,$
  the natural map $\Hom(\Pi_{t'}, -)\to \Hom(\Pi_t, -)$ is an isomorphism on the whole subcategory $\U_{t+1}\subset\prf\Ga.$ Thus, $\Psi_t\cong\Pi_{t'},$ and
the algebra $(\Ga, \be)$ is well-formed.
\end{proof}

The algebra $\Ga=\End_{\La}(M)$ has finite global dimension, and the bounded derived category of finite $\Ga$\!--modules $\bD^b(\md\Gamma)$
is equivalent to the triangulated category
of perfect complexes $\prf\Gamma.$
Consider the $\Ga$\!--module $\Pi=\Hom_{\La}(M, \La)=\widehat{\La}.$ It is projective and is actually isomorphic to $\bigoplus_{j=1}^m \Pi_{s_j},$ where $s_j=(j, l_j)$ and $l_j=\LL(P_j)$ is the Loewy length of $P_j.$
Moreover, we have $\End_{\Ga}(\Pi)\cong\La.$
Thus, the projective $\Ga$\!--modules $\Pi$ is a $\La\hy\Ga$\!--bimodule, and it gives us two functors
\[
(-)\stackrel{\bL}{\otimes}_{\La} \Pi: \D(\Md\La)\lto\D(\Md\Ga)\quad\text{and}\quad
\bR\Hom_{\Ga}(\Pi, -):\D(\Md\Ga)\lto\D(\Md\La)
\]
that are adjoint. The first functor is fully faithful and sends perfect complex to perfect ones, while the second functor is a quotient. They induce the
the following functors:
\[
(-)\stackrel{\bL}{\otimes}_{\La} \Pi: \prf\La\lto\prf\Ga\quad\text{and}\quad
\bR\Hom_{\Ga}(\Pi, -):\prf\Ga\lto\D^b(\md\La)
\]
Thus, the $\La\hy\Ga$\!--bimodule $\Pi$ defines a quasi-functor $\mR: \prfdg\La\to \prfdg\Ga,$ i.e. we have an ff-morphism of noncommutative schemes
$\mr: \dW\to\dV,$ where $\dV=\prfdg\La$ and $\dW=\prfdg\Ga.$ The morphism $\mr$ is the simplest and very important example of a smooth resolution of singularities  of the noncommutative scheme $\dV$
(see Definition \ref{des}).

\begin{theorem}\label{main_fd_algebra}
Let $\La$ be a basic finite dimensional algebra over a field $\kk.$ Then there exist a smooth projective scheme $X$ and a quasi-functor
$\mF: \prfdg\La\to\prfdg X$ such that the following conditions hold:
\begin{enumerate}
\item[1)] the induced homotopy functor $F=\Ho(\mF):\prf\La\to \prf X$ is fully faithful;
\item[2)] the indecomposable projective modules $P_i$ go to vector bundles $\P_i$ on $X$ under $F,$ and the rank of $\P_i$ is equal to $\dim_{\kk} P_i;$
\item[3)] the scheme $X$ is a tower of projective bundles and has a full exceptional collection.
\end{enumerate}
\end{theorem}
\begin{proof}
The quasi-functor $\mF$ can be defined as the composition of the quasi-functor $\mR:\prfdg\La\to\prfdg\Ga$ defined above and the  quasi-functor $\mG:\prfdg\Ga\to\prfdg X$ constructed in Theorem \ref{main_quiver}.
Thus, conditions 1) and 3) hold. Moreover, the indecomposable projective $\La$\!--modules $P_i$ go to projective $\Ga$\!--modules under the quasi-functor $\mR,$ and
after that they go to vector bundles which will be denoted as $\P_i.$

Let us consider the $\Ga$\!--module $\Pi$ which as a $\La\hy\Ga$\!--bimodule defines the quasi-functor $\mR.$
It is the direct sum $\bigoplus_{j=1}^m \Pi_{s_j},$ where $s_j=(j, l_j)$ and $l_j=\LL(P_j)$ is the Loewy length of $P_j.$
We have isomorphisms  $\mR(\La)\cong\Pi,$ while $\mR(P_i)\cong\Pi_{s_i}.$
Property (\ref{hom_pr_st}) implies that
$\Hom_{\Ga}(\Pi, \Delta_t)\cong \kk$ for any $t\in T.$ Since the standard modules go to line bundles under the quasi-functor $\mG,$ for any $\Ga$\!--module $N\in\Fl(\Delta_{1},\ldots, \Delta_{n})$ the rank of the vector bundle $\mG(N)$ is equal to $\Hom_{\Ga}(\Pi, N).$ In particular, we obtain the following equalities:
\[
\operatorname{rk} \mF(P_i)=\operatorname{rk}\mG(\Pi_{s_i})=\dim_{\kk} \Hom_{\Ga}(\Pi, \Pi_{s_i})=\dim_{\kk}\Hom_{\La}(\La, P_i)=\dim_{\kk} P_i.
\]
Thus, condition 2) also holds.
\end{proof}

In the case when the base field $\kk$ is algebraically closed any finite dimensional algebra is Morita equivalent to a basic finite dimensional algebra.
This means that any finite dimensional algebra $A$ is isomorphic to $\End_{\La}(\bigoplus_{i=1}^m P_i^{\oplus k_i}),$ where $\La$ is basic and $P_1,\ldots, P_m$
is the complete set of indecomposable projective $\La$\!--modules.
Therefore, we have the following corollary.

\begin{corollary}
Let $\kk=\overline{\kk}$ be an algebraically closed field. Then for any finite dimensional $\kk$\!--algebra $A$ there exist a smooth projective scheme $X$ and a vector bundle $\E$ on $X$
such that $\End_{X}(\E)\cong A$ and $\Ext^l_{X}(\E, \E)=0$ for all $l>0.$ Moreover, such $X$ can be constructed as  a tower of projective bundles, and $\operatorname{rk}\E\le\dim_{\kk}A.$
\end{corollary}
\begin{remark}
{\rm
There are many reasons to believe that the assertion of Theorem \ref{main_fd_algebra} also holds for an arbitrary algebra with one refinement: the variety
$X$ can not be chosen to be a tower of projective bundles anymore, but an appropriate twisted form of such a variety, i.e. becomes isomorphic to a tower of projective bundles after an extension of the base field.
}
\end{remark}

\end{document}